\documentclass[reqno]{amsart}
\usepackage{amsmath,amsrefs,color}
\usepackage{geometry}
\usepackage{amssymb}
\usepackage{float}
\usepackage{graphicx,picture,%showkeys,
cleveref}
\usepackage[colorinlistoftodos]{todonotes}
\numberwithin{equation}{section}
\usepackage{multirow}
    \usepackage{amsmath}
    \usepackage{array}
\usepackage{makecell}
\usepackage{tabularx}

\renewcommand{\[}{\left[}

\def\neweq#1{\begin{equation}\label{#1}}
\def\endeq{\end{equation}}

\newtheorem{theorem}{Theorem}[section]

\newtheorem{proposition}[theorem]{Proposition}
\newtheorem{lemma}[theorem]{Lemma}
\newtheorem{cor}[theorem]{Corollary}

\newtheorem{rem}[theorem]{Remark}

\newcolumntype{E}{>{$\displaystyle}c<{$}} % E for Equation, centered display math
\newcolumntype{L}{>{\raggedright\arraybackslash}p{3cm}}

\begin{document}

\title[]{Elliptic equations with Hardy potentials and gradient-dependent absorption: \\
existence and refined asymptotics}

\author{Florica C. C\^irstea}

\address{Florica C. C\^irstea, School of Mathematics and Statistics, The University of Sydney, NSW 2006, Australia}
\email{florica.cirstea@sydney.edu.au}

\author{Maria F\u arc\u a\c seanu}

\address{Maria F\u arc\u a\c seanu, School of Mathematics and Statistics, The University of Sydney, NSW 2006, Australia, and \newline
"Gheorghe Mihoc - Caius Iacob" Institute of Mathematical Statistics and Applied Mathematics of the Romanian Academy, 050711 Bucharest, Romania}
\email{farcaseanu.maria@yahoo.com}

\date{}

\begin{abstract} 
Under sharp conditions, we prove the existence and refined asymptotic behaviour near zero (resp., at infinity) for all positive radial solutions to elliptic equations of the form %of nonlinear elliptic equations with singular potentials and gradient-dependent nonlinearities of the form 
\begin{equation}\label{eq11} \tag{*}
  \mathbb L_{\rho,\lambda}(u)=\Delta u+ (2-N-2\rho)\, \frac{x\cdot \nabla u}{|x|^2}+ \frac{\lambda}{|x|^2}u=|x|^{\theta}\,u^q\, |\nabla u|^m\quad \mbox{in } \Omega\setminus\{0\},
\end{equation}
where $\Omega=B_R(0)$ (resp., $\Omega=\mathbb R^{N}\setminus B_{1/R}(0)$) for suitable $R>0$ and 
$N\geq 2$. The dynamics of such solutions is very rich and complicated since we let  
$\rho, \lambda,\theta\in \mathbb R$ be arbitrary, with $ m>0$, $q\geq 0$ and $\kappa:=m+q-1>0$. 
To our knowledge, this is the first study of the local properties of the positive solutions of \eqref{eq11} with arbitrary $m>0$ and $\lambda\not=0$.   
 We identify all profiles near zero (and at infinity via a modified Kelvin transform) under optimal conditions, depending on how $\Theta:=(\theta+2-m)/\kappa$ relates to $0$ or the roots 
$\Theta_\pm$ of $t^2+2\rho t+\lambda$ when $\lambda\leq \rho^2$. For each profile, we advance new methods that unearth the higher order terms in the asymptotic expansion, illuminating the influence of the operator $\mathbb L_{\rho,\lambda}$ and the absorption not visible just from the leading term. Our refined asymptotics is new even when $m=0$ and $\rho=(2-N)/2$.

We highlight two new asymptotic profiles near zero due to
the competition between the Hardy potential with $\lambda>0$ and the gradient-dependent absorption term: (i) a blow-up profile 
$\left[ \lambda \left( \frac{\kappa}{m} \right)^m \right]^{\frac{1}{\kappa}}
| \log |x||^{\frac{m}{\kappa}}
$ when $\Theta=0$  
and (ii) a bounded profile when $\Theta<0$. Any radial solution of \eqref{eq11} with $\lim_{r\to 0^+} u(r)=\gamma\in \mathbb R_+$ satisfies 
$(P_\pm)$ $u(r)=\gamma\pm \lambda^{1/m} \gamma^{1-\kappa/m} (1/\sigma)\,
r^\sigma(1+o(1))$ as $r\to 0^+$, where $\sigma=-\kappa \Theta/m$. Moreover,  for any $\gamma\in \mathbb R_+$, there exists $R>0$ such that \eqref{eq11} has a radial solution (and infinitely many) satisfying $(P_-)$ (and $(P_+)$).  
We discover three different higher order asymptotic corrections, each unveiling the influence of the Laplacian and the second term in $\mathbb L_{\rho,\lambda}(\cdot)$.

  \end{abstract}
\maketitle

\tableofcontents

\section{Introduction and main results} \label{Sec1}

The classification of local and global solutions and their symmetry properties are central themes in the study of nonlinear elliptic and parabolic equations (see, e.g., \cite{Soup} and \cite{bov}). The last decades have brought remarkable achievements in the study of isolated singularities, see V\'eron \cite{bov}. A key question, which has generated extensive research, is 
deciphering how gradient-dependent lower order terms affect the local and global classification of singularities. 
Answering this question will also shed light on the conditions leading to Liouville type results. 
For recent contributions in these directions, see e.g., \cites{BiGV,BVGV,BGV0,BGV3,BGV1,CCM,BNV1,BVer1,CC2015,CC2,FPS,FPS2}.

In this paper, 
we give the existence and refined asymptotics near zero (resp., at infinity) of positive radial solutions for nonlinear elliptic equations with Hardy potentials and gradient-dependent nonlinearities of the form
\begin{equation}\label{eq1}
  \mathbb L_{\rho,\lambda}(u):= \Delta u+ (2-N-2\rho)\, \frac{x\cdot \nabla u}{|x|^2}+\frac{\lambda}{|x|^2}u=|x|^{\theta}\,u^q\, |\nabla u|^m\quad \mbox{in } \Omega\setminus\{0\},
\end{equation}
where $\Omega=B_R(0)$, the ball in $\mathbb R^N$ centered at zero of radius $R>0$ (resp., $\Omega=\mathbb R^N\setminus B_{1/R}(0)$) and $N\geq 2$.   
 We assume throughout that $\rho,\lambda,\theta,m,$ and $q $ are real parameters satisfying
\begin{equation}\label{cond1}
 m>0,\quad q\geq 0\quad \mbox{and}\quad \kappa:=m+q-1>0.
\end{equation}

The main feature of our study is the introduction of the gradient factor $|\nabla u|^m$ in the absorption term of \eqref{eq1}, allowing any $m>0$, in the presence of the Hardy potential. To our best knowledge, this is the first work dealing with the existence and local behaviour of the positive solutions of \eqref{eq1} near zero (resp., at infinity). 

An important innovation of our study is the introduction of the second-term in the left-hand side of \eqref{eq1}, where $\rho\in \mathbb R$ is arbitrary, along with the weight $|x|^\theta$ ($\theta\in \mathbb R$) in the right-hand side of \eqref{eq1}. 
Then, via the
modified Kelvin transform, $\widetilde u(x)=u(x/|x|^2)$ for $|x|\not=0$, 
one can obtain the existence and asymptotics at infinity for the positive radial solutions of \eqref{eq1} in
$\mathbb R^N\setminus B_{1/R}(0)$ from our results (Theorems~\ref{greu1}--\ref{zzz0}) applied to an equation of the same type as \eqref{eq1}, namely,
$$ \mathbb L_{-\rho,\lambda}(u)=|x|^{2m-4-\theta} u^q |\nabla u|^m \quad \mbox{in } B_{R}(0)\setminus \{0\}. 
$$
%\eqref{eq1} in $B_{R}(0)\setminus \{0\}$, where $\rho$ and $\theta$ are replaced by $-\rho$ and $2m-4-\theta$, respectively. 
The above modified Kelvin transform was proved very useful in \cite{CCM} for a related but different class of nonlinear problems with gradient-dependent potentials.
We point out that unlike \cite{MF1} where $m=0$, the standard Kelvin transform does not work in our context because of the introduction of the gradient factor $|\nabla u|^m$ ($m>0$) in \eqref{eq1}.

Our main results reveal the existence of positive (radial) solutions with all the possible profiles near zero when $\Omega=B_{R}(0)$, respectively at infinity 
if $\Omega=\mathbb R^N\setminus B_{1/R}(0)$. In a sequel to this paper, we will confirm that the profiles identified in this work (see Tables~\ref{tabel1}--\ref{tabel3}) are exhaustive. 

A distinctive feature of our paper is that we provide a very detailed asymptotic description of the positive radial solutions of \eqref{eq1} near zero (and, hence, at $\infty$), which shed new light even in particular cases studied in the literature \cites{Cmem,MF1,CC2015} for $\rho=(2-N)/2$ and $m\in [0,2)$. Even in these works, the local behaviour near zero of the positive solutions proved to be varied and rich, according to various ranges of the parameters. 
Our problem \eqref{eq1} has five parameters restricted only by \eqref{cond1}, their interplay 
leading to eleven different asymptotic profiles near zero. 

It is interesting and challenging to reveal the dominant player in the various competitions between the terms in the operator $\mathbb L_{\rho,\lambda}(\cdot)$ and the gradient-dependent absorption nonlinearity. We develop new and robust methods for identifying not only the dominant behaviour of the positive radial solutions of \eqref{eq1} near zero  but also higher order terms in the asymptotic expansion.

We note that the existence methods in \cite{CC2015} for $\Delta u=u^q |\nabla u|^m$  used the sub-super-solutions method and gradient estimates, which required the assumption  $m\in (0,2)$.  The existence of positive solutions and their behaviour near zero was left open in the range $m\geq 2$. 
Here, we close this gap for any $m>0$ for radially symmetric solutions, treating a significantly more general class of problems such as \eqref{eq1}, involving Hardy potentials $\lambda u/|x|^2$ for any $\lambda\in \mathbb R$.

Equation \eqref{eq1} with
$\Omega=\mathbb R^N$ is invariant under the scaling
transformation $T_j[u]$ ($j>0$), where 
\begin{equation} \label{scale}
T_j[u](x)=j^\Theta\, u(jx) \quad  \mbox{with } \Theta:=\frac{\theta+2-m}{\kappa}.
\end{equation}

\subsection{New profiles due to the Hardy potential and gradient-dependent absorption} \label{subsec1}

In our framework \eqref{cond1}, we discover new asymptotic profiles of the positive radial solutions of \eqref{eq1} near zero as shown in Table~\ref{tabel1} below. 
The first two profiles, which have no analogues for $m=0$, are generated by the competition between the Hardy potential $\lambda u/|x|^2$ with $\lambda>0$ and the absorption term featuring the gradient factor $|\nabla u|^m$ with $m>0$, see Theorems~\ref{greu1} and \ref{tieo10}. The last two profiles in Table~\ref{tabel1} give the optimal ranges for which \eqref{eq1} has positive radial non-constant solutions converging to a positive constant as $|x|\to 0$, see Theorem~\ref{profff1} and \ref{equa1}.

\begin{table}[H]
\centering
\caption{New asymptotic profiles of solutions near zero}
\label{tabel1}
\small
\begin{tabular}{|>{\raggedright\arraybackslash}p{7cm}|p{4cm}|p{2cm}|}
\hline
Sharp range for the profile & Asymptotic profile & Theorems \\
\hline

\makecell[l]{
$\Theta<0, \lambda>0, \rho\in \mathbb{R}, \displaystyle\sigma:=-\frac{\kappa\Theta}{m}$
}
&
\begin{tabular*}{4cm}{@{}l@{\extracolsep{\fill}}r@{}}
$\displaystyle\lim_{|x|\rightarrow 0}\frac{u(x)-\gamma}{|x|^\sigma}=\pm\frac{\lambda^{\frac{1}{m}}}{\sigma \gamma^{\frac{\kappa}{m}-1}} $\\
$ \mbox{where } \gamma\in (0,\infty) \ \mbox{is arbitrary}$ %& \AddLabel{consta}
\end{tabular*}
& \makecell[l]{Theorem \ref{greu1}} \\ \hline

\makecell[l]{
$E_0(|x|):=\left[ \lambda \left( \frac{\kappa}{m} \right)^m \right]^{\frac{1}{\kappa}}
\left( \log \frac{1}{|x|} \right)^{\frac{m}{\kappa}}$\\
where $\Theta=0, \lambda>0, \rho\in \mathbb{R}$
}
&
\begin{tabular*}{4cm}{@{}l@{\extracolsep{\fill}}r@{}}
$\displaystyle \lim_{|x|\to 0}\frac{u(x)}{E_0(|x|)}=1$ %& \AddLabel{tuna}
\end{tabular*}
& Theorem \ref{tieo10} \\ \hline

\makecell[l]{
%$\Theta<0, \lambda>0, \rho\in \mathbb{R}$\\
$\lambda=0, \beta:=\kappa\,\Theta+2\rho\left(m-1\right)\neq 0$
}
&
\begin{tabular*}{4cm}{@{}l@{\extracolsep{\fill}}r@{}}
$\displaystyle \lim_{|x|\to 0}u(x)=\gamma\in (0,\infty)$ %& \AddLabel{consta}
\end{tabular*}
& \makecell[l]{Theorem \ref{profff1}} \\ \hline

\makecell[l]{
%$\Theta<0, \lambda>0, \rho\in \mathbb{R}$\\
$\lambda=0, \beta:=\kappa\,\Theta+2\rho\left(m-1\right)=0$
}
&
\begin{tabular*}{4cm}{@{}l@{\extracolsep{\fill}}r@{}}
$\displaystyle \lim_{|x|\to 0}u(x)=\gamma\in (0,\infty)$ %& \AddLabel{consta}
\end{tabular*}
& \makecell[l]{Theorem \ref{equa1}} \\ \hline

\end{tabular}
\end{table}

We next introduce the main results corresponding to each of the profiles in Table~\ref{tabel1}.

For the top profile, when $\Theta<0$ and $\lambda>0$, we obtain that for every $\gamma\in (0,\infty)$, there exists $R>0$ small such that \eqref{eq1} in $B_R(0)\setminus \{0\}$
has infinitely many positive solutions satisfying 
\begin{equation} \label{consta}  \lim_{|x|\to 0}u(x)=\gamma\in (0,\infty). \end{equation}
Such solutions arise from the competition between the Hardy potential and the absorption term, which besides the dominant term, 
reveals the next-order term as $|x|\to 0$, namely,  
\begin{equation} \label{hcal1} \tag{$P_{\pm}$} 
u(r)=\gamma \pm \frac{\lambda^{\frac{1}{m}}}{\sigma \gamma^{\frac{\kappa}{m}-1}} r^\sigma (1+o(1)), \quad \mbox{where } \sigma:=-\frac{\kappa \Theta}{m}>0. 
\end{equation}
%\mathcal H_\pm(|x|)=\gamma\pm  \frac{\lambda^{\frac{1}{m}}}{\sigma \gamma^{\frac{\kappa}{m}-1}} |x|^\sigma .
This surprising phenomenon has distinct existence findings for the solutions of \eqref{eq1} satisfying $(P_+)$ compared with those satisfying $(P_-)$ (see Theorem~\ref{greu1} (II)).  Yet, nothing in the asymptotics in \eqref{hcal1} indicates the influence of the Laplacian and the second term in the left-hand side of \eqref{eq1}. To identify such influence, we look at the next terms in the asymptotic expansion of $u$ near zero. To this aim, we define 
\begin{equation} \label{d1d2b1b2}
\begin{aligned}
&d_1=\sigma\left(\sigma-2\rho\right) +\lambda\left(1-q\right), \quad 
d_2=\left(2-m\right)d_1-2\lambda mq +\left(3\sigma-2\rho\right)\sigma, \\
&  \mathfrak b_1=\left( \frac{\lambda^{\frac{1}{m}}}{\sigma \gamma^{\frac{\kappa}{m}}} \right)^2  \frac{d_1}{2\lambda m}\quad \mbox{and}\quad
 \mathfrak b_2=\left( \frac{\lambda^{\frac{1}{m}}}{\sigma \gamma^{\frac{\kappa}{m}}} \right)^3\frac{d_1d_2+\lambda^2 mq \left(1-q\right)}{6\lambda^2m^2}.
\end{aligned}
\end{equation}

Our analysis 
unveils the third and fourth order terms of power-type, see \eqref{firum} when $\mathfrak b_1^2+\mathfrak b_2^2\not=0$. Otherwise, we have  \eqref{iotaa} or (only for $(P_+)$) \eqref{miror} holds, revealing in the latter situation an exponential decay in the higher order asymptotics.

Our main result on the first profile in Table~\ref{tabel1} is as follows.

\begin{theorem} [Constant Profile, $\Theta<0<\lambda$] \label{greu1}
Let \eqref{cond1} hold, $\rho\in \mathbb{R}$, $\Theta<0$ and $\lambda>0$. 
\begin{itemize}
 \item[(I)] 
 Let $u$ be any positive radial solution of \eqref{eq1} in $B_{R}(0)\setminus \{0\}$ with $R>0$ satisfying \eqref{consta}. 
Then, $u$ satisfies either $(P_+)$ or $(P_-)$ as $r\to 0^+$. 
%\begin{equation} \label{frum} \tag{$P_{\pm}$} 
%u(r)=\gamma \left( 1\pm \frac{\lambda^{\frac{1}{m}}}{\sigma \gamma^{\frac{\kappa}{m}}} r^\sigma (1+o(1)) \right), \quad \mbox{where } \sigma:=-\frac{\kappa \Theta}{m}>0. 
%\end{equation}
%We define $ d_1=\sigma\left(\sigma-2\rho\right) +\lambda\left(1-q\right)$ and $d_2=\left(2-m\right)d_1-2\lambda mq +\left(3\sigma-2\rho\right)\sigma$,  
%$$  \mathfrak b_1=\left( \frac{\lambda^{\frac{1}{m}}}{\sigma \gamma^{\frac{\kappa}{m}}} \right)^2  \frac{d_1}{2\lambda m}\quad \mbox{and}\quad
% \mathfrak b_2=\left( \frac{\lambda^{\frac{1}{m}}}{\sigma \gamma^{\frac{\kappa}{m}}} \right)^3\frac{d_1d_2+\lambda^2 mq \left(1-q\right)}{6\lambda^2m^2}.
%$$
\begin{itemize} 
\item[(a)]
If $\mathfrak b_1^2+\mathfrak b_2^2\not=0$, then 
$u$ satisfies as $r\to 0^+$
\begin{equation} \label{firum} 
u(r)=\gamma \left( 1\pm \frac{\lambda^{\frac{1}{m}}}{\sigma \gamma^{\frac{\kappa}{m}}} r^\sigma 
+\mathfrak b_1 r^{2\sigma}\pm \mathfrak b_2 r^{3\sigma}
(1+o(1)) \right).
\end{equation}
\item[(b)] If $\mathfrak b_1^2+\mathfrak b_2^2=0$ (that is, $d_1=0$ and $q\left(1-q\right)=0$) and $u$ satisfies $(P_\pm)$, then  
\begin{equation} \label{iotaa} u(r)= \gamma \left( 1\pm \frac{\lambda^{\frac{1}{m}}}{\sigma \gamma^{\frac{\kappa}{m}}} r^\sigma\right) \quad \mbox{for every } r>0\ \mbox{small} \end{equation}
or (only for $(P_+)$) there exists a constant $C\in \mathbb R\setminus \{0\}$ such that as $r\to 0^+$
\begin{equation} \label{miror} u(r)=\left\{ 
\begin{aligned}
& \gamma+ \frac{\lambda^{\frac{1}{m}}}{\sigma} r^\sigma +C r^{2\sigma+\frac{\lambda m}{\sigma}} 
\exp\left(-\frac{m\gamma \lambda^{1-\frac{1}{m}}}{\sigma} r^{-\sigma}\right)(1+o(1)) && \mbox{if }q=1,&\\
& \gamma+ \frac{(\lambda \gamma)^{\frac{1}{m}}}{\sigma} r^\sigma +C r^{2\rho+\sigma} \exp\left(-\frac{m}{\sigma}\left( \lambda \gamma\right)^{\frac{m-1}{m}} r^{-\sigma}\right) (1+o(1)) && \mbox{if }q=0.&
\end{aligned}
\right. 
\end{equation}
\end{itemize}
\item[(II)] For every $\gamma\in \mathbb R_+$, there exists $R>0$ such that \eqref{eq1} in $B_R(0)\setminus \{0\}$ has a positive radial solution satisfying {\rm ($P_-$)} and infinitely many positive radial solutions satisfying {\rm ($P_+$)}. 
\end{itemize}
\end{theorem}

The second profile in Table~\ref{tabel1}, namely $E_0$, models the blow-up rate of the positive solutions of \eqref{eq1} when $\Theta=0$ and $\lambda>0$. In Theorem~\ref{tieo10}, we establish the existence of infinitely many positive radial solutions of \eqref{eq1} 
satisfying 
 \begin{equation} \label{tuna}
 \lim_{|x|\to 0}\frac{u(x)}{E_0(|x|)}=1,\quad \mbox{where } E_0(|x|)=\left[ \lambda \left( \frac{\kappa}{m} \right)^m \right]^{\frac{1}{\kappa}}
\left( \log \frac{1}{|x|} \right)^{\frac{m}{\kappa}}. \end{equation}
Moreover, we reveal the precise form of the higher order terms in the expansion near zero.  
%that blow-up at zero as a power of a logarithm . 

\begin{theorem}[Profile $E_0$] \label{tieo10}
 Let \eqref{cond1} hold, $\Theta=0$, $\lambda>0$ and $\rho\in \mathbb{R}$. 
 For every $R>0$, equation \eqref{eq1} in $B_R(0)\setminus \{0\}$ has infinitely many positive radial solutions satisfying \eqref{tuna}.
Moreover, let $u$ be any positive radial solution of \eqref{eq1} in $B_R(0)\setminus \{0\}$ for $R>0$ satisfying \eqref{tuna}. 
 
 $\bullet$ If $\rho=0$ and $q=1$, then $\lim_{r\nearrow R} u(r):=u(R)\in [0,\infty)$ and 
\begin{equation} \label{aluxi3} u(r)=\lambda^{\frac{1}{m}}\,\log \,(R/r)+u(R)\quad \mbox{for all }r\in (0,R). \end{equation}  
 
$\bullet$ If $\rho\not=0$, then 
\begin{equation} \label{simp1}
\frac{u(r)}{E_0(r)}=1+\frac{2\rho\,m}{\lambda \kappa^2}\,\frac{\log \log \, (1/r)}{\log \, (1/r)} (1+o(1))\quad \mbox{as } r\to 0^+.
\end{equation}

$\bullet$ If $\rho=0$ and $q\not=1$, then there exists a constant $C\in \mathbb R$ 
such that as $r\to 0^+$ 
 \begin{equation} \label{aluxi1} 
\frac{u(r)}{E_0(r)}=
 1+\frac{C}{\log \,(1/r)} + \left( 
 \frac{m}{\lambda \kappa^3}-\frac{C^2}{2m} 
 \right) \frac{q-1}{\log^2\,(1/r) }(1+o(1)).
\end{equation}
\end{theorem}

The third and fourth profiles in Table~\ref{tabel1} provide the sharp range for the existence of positive {\em non-constant} radial solutions of \eqref{eq1} such that $\lim_{r\to 0^+} u(r)=\gamma\in (0,\infty)$, differentiating between $\beta\not=0$ in Theorem~\ref{profff1} and $\beta=0$ in Theorem~\ref{equa1} to reveal the second / third term in the asymptotic expansion of $u$ near zero. We note that our findings are novel and complement known results from \cites{BiGV,Cmem,CC2015,MF1} on \eqref{eq1} with $\rho=(2-N)/2$ and $\lambda=0$, that is,
 \begin{equation} \label{mirac} \Delta u=|x|^\theta u^q |\nabla u|^m \quad \mbox{in } B_R(0)\setminus \{0\}, \end{equation}
where $q\geq 0$, $0\leq m<2$, $\theta\in \mathbb R$ and $q+m>1$. More precisely, in Theorem~\ref{profff1} (II), we extend with a different method the conclusions of \cite[Proposition 3.1 (b)]{Cmem}, where $m=0$ and $\theta>-2$, as well as of \cite[Theorem 2.2]{CC2015}, where $q\geq 0$, $m\in (0,1)$ and $\theta=0$. Using the Leray--Schauder fixed point theorem, it was proved in \cite[Theorem 2.2]{CC2015} that for every $\eta>0$, there exists a positive radial increasing solution of $\Delta u=u^q|\nabla u|^m$ in $B_1(0)\setminus \{0\}$, subject to $u=\eta$ on $\partial B_1(0)$. The latter Dirichlet boundary condition yields a unique positive value $\gamma$ of the limit of $u$ at zero, where the decay rate of $u(r)-\gamma$ as $r\to 0^+$ remained unknown. We close this gap in Theorems~\ref{profff1} and \ref{equa1}
by starting with an arbitrary $\gamma\in (0,\infty)$ and for $R=R(\gamma)>0$ small, prove under optimal conditions the existence of positive non-constant radial solutions $u$ of \eqref{eq1} in $B_R(0)\setminus \{0\}$ with $\lim_{|x|\to 0} u(x)=\gamma$, along with a refined asymptotics near zero.

In \cite[Proposition 3.1]{BiGV} when $\theta=q=0$ and $ N/(N-1)\leq m<2$ (resp., $1<m<N/(N-1)$), it was proved that the only positive radial solutions $u$ of \eqref{mirac} in $B_1(0)\setminus \{0\}$ (resp., satisfying $\lim_{r\to 0^+} u(r)\in (0,\infty)$) are constant. We mention that \cite{BiGV} studies local and global properties of non-negative and signed solutions for $p$-Laplacian type equations of the form $\Delta_p u=|\nabla u|^m$ in a domain $\Omega\subset \mathbb R^N$ with $N\geq p>1$ and $m>p-1$.

%%%%

%%%%

\vspace{0.2cm}
In Theorem~\ref{profff1} (resp., Theorem~\ref{equa1}) for arbitrary $\gamma\in \mathbb R_+$, we give the existence of positive radial solutions of \eqref{eq1} in $B_R(0)\setminus \{0\}$, where $R=R(\gamma)>0$ is small, such that
 \begin{equation} 
\label{animat}
\lim_{r\to 0^+} u(r)=\gamma\in (0,\infty)\quad \mbox{and}\quad r^*(u)=\sup\,\{r\in (0,R):\ u\equiv \gamma\ \mbox{on } (0,r) \}=0,
\end{equation}
assuming that $\lambda=0$ and $\beta\not=0$ (resp., $\beta=0$), where 
$$\beta:=\kappa \Theta+2\rho \left(m-1\right).$$  

For simplicity of reference, given $\gamma\in \mathbb R_+$ and $R>0$, we define by $\mathcal S_{\gamma,R}$ the set of positive radial solutions of \eqref{eq1} in $B_R(0)\setminus \{0\}$ such that \eqref{animat} holds.  

When $\beta\not=0$, we show that there exists $R>0$ such that $\mathcal S_{\gamma, R}\not=\emptyset$ 
%\eqref{eq1} has a positive radial solution in $B_R(0)\setminus \{0\}$ satisfying \eqref{animat} 
precisely in the following cases
(see Lemma~\ref{nelem} for Case (I) and Lemma~\ref{maria} for Cases (II) and (III)): 

\begin{itemize}
\item[(I)] $\beta<0$ and $m=1$.

 \item[(II)] $\beta\not=0$, $m\not=1$, $\Theta\not=0$ and ${\rm sgn}\,(\Theta)={\rm sgn}\, (1-m)$.

 \item[(III)] $\beta>0$ and $\rho>0$. 
\end{itemize}

In Case {\rm (II)}, we define $\zeta$ as follows
$$\zeta:=\frac{\kappa \,\Theta}{1-m}>0.$$

\begin{theorem}[Constant Profile, $\lambda=0\not=\beta$] \label{profff1}
Let \eqref{cond1} hold, $\lambda=0\not=\beta$ and  
$\gamma\in \mathbb R_+$ be arbitrary. 
\begin{itemize}
%\item[(A)] Then, 
\item[(A)] Let Case {\rm (I)} hold. 
%{\rm (I)} hold and $R>0$. 
Then, given $R>0$, any $u\in \mathcal S_{\gamma,R}$ satisfies $u'>0$ on $(0,R)$ and 
there exists a constant $C>0$ such that 
 \begin{equation}
 \label{iuted}
 u(r)=\gamma+ C r^{2\rho-1-\theta} 
  \exp\left( \frac{\gamma^q}{\theta+1} r^{\theta+1}\right)  (1+o(1))\quad \mbox{as } r\to 0^+. 	
 \end{equation}
 Moreover, there exists $R>0$ such that $\mathcal S_{\gamma,R}$ is infinite. 
%\eqref{eq1} in $B_R(0)\setminus \{0\}$ has infinitely many positive radial solutions satisfying \eqref{animat}. 
\item[(B)] Let Case {\rm (II)} and Case {\rm (III)} hold. Then, given $R>0$, any $u\in \mathcal S_{\gamma,R}$ satisfies either 
 \begin{equation}
 \label{uite1}
 u(r)=\gamma+  \frac{{\rm sgn}\,(\zeta-2\rho)}{\zeta} \left( \frac{\gamma^q}{|\zeta-2\rho|} \right)^\frac{1}{1-m}
r^{\zeta} \,(1+o(1)) \quad \mbox{as } r\to 0^+ 
 \end{equation}
or there exists a constant $C_1\in \mathbb R\setminus \{0\}$ such that 
\begin{equation} \label{refina}
		u(r)=\gamma+ C_1\,r^{2\rho}+\frac{\gamma^q (2\rho\,|C_1|)^m}{\beta \left(2\rho+\beta\right)}r^{2\rho+\beta}(1+o(1))\quad \mbox{as }
r\to 0^+.		
\end{equation} 
\item[(C)] Let Case {\rm (II)} hold. 
Then, if $\beta>0$ (resp., $\beta<0$), then there exist $R>0$
and $u\in \mathcal S_{\gamma,R}$ (resp., infinitely many $u\in \mathcal S_{\gamma,R}$) satisfying \eqref{uite1}. In addition, if Case {\rm (III)} fails, then for $R>0$, any
$u\in \mathcal S_{\gamma,R}$
 %positive radial solution of \eqref{eq1} in $B_R(0)\setminus \{0\}$ with the property \eqref{animat} 
satisfies only \eqref{uite1} as $r\to 0^+$.	
%as $r\to 0^+$
 
 \item[(D)] Let Case {\rm (III)} hold. %but not Case {\rm (II)}. 
% Then, for every $R>0$ and 
%$u\in \mathcal S_{\gamma,R}$, there exists $C_1\in \mathbb R\setminus \{0\}$ such that 	
Then, for every $C_1\in \mathbb R\setminus \{0\}$, there exist $R>0$ and  
$u\in \mathcal S_{\gamma,R} $ so that \eqref{refina} holds. If, in addition, Case {\rm (II)} fails, then for every $R>0$ and 
$u\in \mathcal S_{\gamma,R}$, there exists $C_1\in \mathbb R\setminus \{0\}$ such that \eqref{refina} is satisfied.

 \end{itemize}
	\end{theorem}

%We note that in the framework of Theorem~\ref{profff1}, there are no other possible profiles of solutions satisfying \eqref{animat} than the ones pertaining to Cases (I)--(III). More precisely, 
%under assumptions common to both Case (II) and Case (III), 
%When the assumptions pertain to Case (II) but not Case (III) (respectively, Case (III) and not Case (II)), only the behaviour in \eqref{uite1} (respectively, \eqref{refina}) 
%occurs (see Lemma~\ref{maria}). 

\begin{theorem}[Constant Profile, $\lambda=\beta=0$] \label{equa1}
Let \eqref{cond1} hold, $\lambda=0$ and $\beta=\kappa \Theta+2\rho \,(m-1)=0$. 
Let $u$ be a positive radial solution of \eqref{eq1} in $B_R(0)\setminus \{0\}$ for $R>0$ satisfying \eqref{animat}. 
\begin{itemize}
\item[(A)] Let $m=1$ and $\Theta=0$. 
Then, $u'(r)\not=0$ for all $r\in (0,R)$ and 
$2\rho+\gamma^q {\rm sgn}\, (u'(r))\geq 0$. By letting $u(R^-)=\lim_{r\nearrow R} u(r)$, we have  for all $r\in (0,R)$ 
		\begin{equation} \label{1fip}
\int_{u(r)}^{u(R^-)} \frac{ds}{ 2\rho \left(s-\gamma\right)+ \frac{{\rm sgn}\, (u'(r))}{q+1}\left( s^{q+1}-\gamma^{q+1}\right)}=\log \left(\frac{R}{r}\right). 		
		\end{equation} 		
$\bullet$ If 
 If $2\rho+\gamma^q \,{\rm sgn}\, (u'(r))>0$, then there exists a constant $C_1>0$ such that as $r\to 0^+$
 			\begin{equation} \label{1sof}
			u(r)=\gamma+C_1 {\rm sgn}\, (u'(r)) \,r^{2\rho+\gamma^q \,{\rm sgn}\, (u'(r))}+
q\,\gamma^{q-1}\,(C_1)^2 \,
\frac{r^{2\left(2\rho+\gamma^q\, {\rm sgn}\,(u'(r)) \right)} }
{2\left(2\rho+\gamma^q\,{\rm sgn}\,(u'(r))   \right)} \,  (1+o(1)). 
			\end{equation} 
$\bullet$ If  $2\rho+\gamma^q \,{\rm sgn}\, (u'(r))=0$, then as $r\to 0^+$, 
			\begin{equation} \label{1sof1}
			u(r)=\gamma+\frac{2\, \gamma^{1-q}}{q}\frac{{\rm sgn}\, (u'(r))}{\log\, (1/r)}\left[
1-\frac{2\,(q-1)\, {\rm sgn}\, (u'(r))}{3\,q\,\gamma^q} \frac{\log \,(\log \,(1/r))}{\log \,(1/r)}(1+o(1)) \right]. 
			\end{equation}
			Furthermore, if $q=1$, then it holds 
			\begin{equation} \label{1sof2} 
		 u(r)=\gamma+\frac{2\, {\rm sgn}\, (u'(r))}{\log \left(\frac{R}{r}\right)+
	\frac{2\, {\rm sgn}\, (u'(r))}{u(R^-)-\gamma}}\quad \mbox{for all } r\in (0,R).
			\end{equation}

\item[(B)] Let $\Theta=\rho=0$. Then, necessarily $m\in [1,2)$, 
$u'>0$ on $(0,R)$ and 
\begin{equation} \label{1van0}
\int_{u(r)}^{u(R^-)}\frac{ds}{
\left(s^{q+1}-\gamma^{q+1}\right)^{\frac{1}{2-m}}}=\left( \frac{2-m}{q+1}\right)^{\frac{1}{2-m}}\log \left(\frac{R}{r}\right)\quad \mbox{for all } r\in (0,R).	
\end{equation}
If $m=1$, then there exists a constant $C_1>0$ such that \eqref{1sof} holds as $r\to 0^+$, namely, 
\begin{equation} \label{1sif}
			u(r)=\gamma+C_1 \,r^{\gamma^q}+
q\,\gamma^{q-1}\,(C_1)^2 \,
\frac{r^{2\gamma^q } }
{2\gamma^q } \,  (1+o(1)). 
			\end{equation} 
On the other hand, if $m\in (1,2)$, then as $r\to 0^+$, we have 
\begin{equation}
\label{moss}
u(r)=\gamma+ \frac{(m-1)^{\frac{2-m}{1-m}}}{2-m} \gamma^{\frac{q}{1-m}} \left(\log \,(1/r)\right)^{\frac{2-m}{1-m}}	(1+o(1)).
\end{equation}

\item[(C)] Let  
$\kappa \Theta=2\rho\left(1-m\right)\not=0$. Then, necessarily, $\rho>0$, ${\rm sgn}\, (\Theta)={\rm sgn}\, (1-m)$ and %as $r\to 0^+$, it holds
\begin{equation} 
\label{fivi1} 	
u(r)=\gamma-{\rm sgn}\,(1-m)\,
 \frac{\left( |1-m|\, \gamma^q\right)^{\frac{1}{1-m}} }{2\rho}\, r^{2\rho} \left(\log \frac{1}{r}\right)^{\frac{1}{1-m}}  (1+o(1))\quad \mbox{as } r\to 0^+.
\end{equation}
Moreover, for every $\gamma\in (0,\infty)$, there exists
		$R>0$ small such that \eqref{eq1} in $B_R(0)\setminus \{0\}$ 
has infinitely many positive radial solutions $u$ satisfying \eqref{fivi1}. 
\end{itemize}
\end{theorem}

%%% Second Table
\subsection{Competitive profiles: the operator $\mathbb L_{\rho,\lambda}$ versus the absorption} \label{subsec2}
When $\Omega=\mathbb R^N$, we remark that \eqref{eq1} has an explicit positive radial solution of power type, given by $U_0$ in 
Table~\ref{tabel2}, under the assumption that $\Theta\not=0$ and $\ell>0$. In addition, for every $R>0$, we prove in Theorem~\ref{raf} that \eqref{eq1} in $B_R(0)\setminus \{0\}$ has infinitely many positive radial solutions satisfying
\begin{equation} \label{uoo} u(x)\sim U_0(|x|)\quad \mbox{as }|x|\to 0.
\end{equation}
More importantly, we distinguish these solutions from one another by determining the second term in their asymptotic behaviour near zero.
These findings are new, even in the case $m=0$. %Our findings in Theorem~\ref{raf} also apply for $m=0$, 
 
%asymptotically equivalent to $U_0$ near zero. 

%Moreover, under these assumptions we prove that 
%The numbers $\ell$ and $\Theta$ in \eqref{scale}, together with their sign, will be essential in the sequel.
%For example, if $|\Theta| \,\ell>0$, then a positive radial solution of \eqref{eq1} in $\mathbb R^N\setminus \{0\}$ is given by
%$U_0$ in Table~\ref{tabel2}. 

%\vspace{1cm}
When $\lambda\leq \rho^2$ and $\Theta>\Theta_\pm=-\rho\pm \sqrt{\rho^2-\lambda}$, we prove in Theorems~\ref{prof1}--\ref{mix1} that \eqref{eq1} has 
positive radial solutions, which near zero are modelled by special
solutions $\Phi^\pm_{\rho,\lambda}$ of $\mathbb L_{\rho,\lambda}[\Phi]=0$, see  Table~\ref{tabel2}. These two profiles reflect the dominance near zero of the operator $\mathbb L_{\rho,\lambda}$ over the absorption term. However, it is natural to ask for a refinement of the  asymptotics near zero of the positive radial solutions of \eqref{eq1} to make visible 
the influence of the absorption term. This is a delicate question, which we settle completely in Theorems~\ref{prof1}--\ref{mix1}. 

%s the dominant behaviour 
%the behaviour near zero of the positive
%solutions of \eqref{eq1} is linked with 
%We have $\ell\leq 0$ if and only $\lambda\leq \rho^2$ and $\Theta_-\leq \Theta\leq \Theta_+$, where
%$\Theta_\pm=-\rho\pm\sqrt{\rho^2-\lambda}$. %In particular, we see that $\Theta_\pm=-\rho$ when $\lambda=\rho^2$.

%{\bf 2. Competitive profiles: the operator $\mathbb L_{\rho,\lambda}$ versus the absorption.}

\begin{table}[H]
\centering
\caption{Asymptotic profiles of solutions near zero}
\label{tabel2}
\small
\begin{tabular}{|>{\raggedright\arraybackslash}p{7cm}|p{4cm}|p{2cm}|}
\hline
Definition and sharp range for the profile & Asymptotic profile & Theorems \\
\hline

\makecell[l]{
$U_0(|x|):=(|\Theta|^{-m}\ell)^{\frac{1}{\kappa}}\,|x|^{-\Theta}:=M_0 |x|^{-\Theta}$\\
where $\Theta \neq 0, \ell=\Theta^2+2\rho \Theta+\lambda>0$
}
&
\begin{tabular*}{4cm}{@{}l@{\extracolsep{\fill}}r@{}}
$\displaystyle \lim_{|x|\to 0}\frac{u(x)}{U_0(|x|)}=1$ %& \AddLabel{uoo}
\end{tabular*}
& Theorem \ref{raf} \\ \hline

\makecell[l]{
$\Phi_{\rho,\lambda}^+(|x|):=
|x|^{-\Theta_+}$, where $\Theta_+=-\rho+\sqrt{\rho^2-\lambda},$\\ 
$\lambda<\rho^2, \Theta>\Theta_+\neq0$}
&
\begin{tabular*}{4cm}{@{}l@{\extracolsep{\fill}}r@{}}
$\displaystyle \lim_{|x|\to 0}\frac{u(x)}{\Phi_{\rho,\lambda}^+(|x|)}=\gamma\in (0,\infty)$ %& \AddLabel{minaaa}
\end{tabular*}
& \makecell[l]{Theorem \ref{prof1}} \\ \hline

\makecell[l]{
$\Phi_{\rho,\lambda}^+(|x|):=|x|^{\rho}\log \frac{1}{|x|}$\\ where $\lambda=\rho^2, \Theta>-\rho$
}
&
\begin{tabular*}{4cm}{@{}l@{\extracolsep{\fill}}r@{}}
$\displaystyle \lim_{|x|\to 0}\frac{u(x)}{\Phi_{\rho,\lambda}^+(|x|)}=\gamma\in (0,\infty)$ %& \AddLabel{minaaa}
\end{tabular*}
& \makecell[l]{Theorem \ref{th81}} \\ \hline

\makecell[l]{
$\Phi_{\rho,\lambda}^-(|x|):=|x|^{-\Theta_-}$, where $\Theta_-=-\rho-\sqrt{\rho^2-\lambda}$,\\
$\lambda \leq \rho^2, \Theta>\Theta_-\neq0$
}
&
\begin{tabular*}{4cm}{@{}l@{\extracolsep{\fill}}r@{}}
$\displaystyle \lim_{|x|\to 0}\frac{u(x)}{\Phi_{\rho,\lambda}^-(|x|)}=\gamma\in(0,\infty)$ %& \AddLabel{mina}
\end{tabular*}
& Theorem \ref{mix1} \\ \hline

\end{tabular}
\end{table}

%%%%%

%We first refine the asymptotics results and gradient-dependent absorption term

Assuming $\ell>0$ and $\Theta\not=0$, we next state Theorem~\ref{raf} on the multiplicity of the positive radial solutions of \eqref{eq1} satisfying \eqref{uoo} and their precise two-term asymptotic expansion near zero. A key role in the latter is played by 
$\xi_0$, the unique {\em positive} solution of the quadratic equation in $\xi$ 
\begin{equation} \label{xiz} \xi^2+\left(\frac{\ell\,m}{\Theta}-2\left(\rho+\Theta\right) \right) \xi-\ell\kappa=0. 
	\end{equation}  
	
	\begin{theorem}[Profile $U_0$] \label{raf}
	Let \eqref{cond1} hold,  
	$\Theta\not=0$ and $\ell>0$. 
For every $R>0$, equation \eqref{eq1} in $B_R(0)\setminus \{0\}$ has infinitely many positive radial solutions satisfying \eqref{uoo}. 
Moreover, let $u$ be any positive radial solution of \eqref{eq1} in $B_R(0)\setminus \{0\}$ with $R>0$ such that \eqref{uoo} holds.
	
	If $u\not\equiv U_0$ on 
	any interval $(0,r_*)\subset (0,R)$ then,
	there exists $\mu_0\in \mathbb R\setminus \{0\}$ such that 
		\begin{equation} \label{lov1}   \lim_{r\to0^+} r^{-\xi_0} \left(\frac{ru'(r)}{u(r)}+\Theta\right)=\mu_0
		\quad \mbox{and}\quad
		\frac{u(r)}{U_0(r)}=1+\frac{\mu_0}{\xi_0}\,r^{\xi_0}
		(1+o(1))  \mbox{  as } r\to 0^+. 
	\end{equation}  
%	where $\xi_0$ is the {\em positive} root of the following quadratic equation (in $\xi$)
	
\end{theorem}

%%%%

%%%%%%

The next two results apply when $\lambda\leq \rho^2$ and $\Theta>\Theta_+\not=0$, giving 
the existence and refined asymptotics near zero for the positive radial solutions of \eqref{eq1} with the property 
$$  \lim_{|x|\to 0}\frac{u(x)}{\Phi_{\rho,\lambda}^+(|x|)}=\gamma\in (0,\infty).
$$
As Table~\ref{tabel2} shows, the expression of $\Phi_{\rho,\lambda}^+$ is different when $\lambda<\rho^2$ compared with $\lambda=\rho^2$, leading to different conclusions and treatments. We refer to Theorem~\ref{prof1} for $\lambda<\rho^2$ and Theorem~\ref{th81} for $\lambda=\rho^2$. 
For $\lambda<\rho^2$, 
the next terms in the asymptotics of $u$ near zero involve $\chi$ and $\mathfrak a_2$, depending on the position of $\chi$ with respect to zero, where
we define 
\begin{equation} \label{chidd} 
\chi:=\kappa\left( \Theta-\Theta_+\right)-2\sqrt{\rho^2-\lambda} 
 \quad \mbox{and}\quad 
\mathfrak{a}_2:=\frac{\gamma^{q+m}\,|\Theta_+|^m}{
	\kappa\left(\Theta-\Theta_+\right)}. %\ \ \mbox{for given } \gamma\in (0,\infty).
\end{equation}

\begin{theorem}[Profile $\Phi_{\rho,\lambda}^+$, $\lambda<\rho^2$] \label{prof1} Let \eqref{cond1} hold, $\lambda< \rho^2$ and $\Theta>\Theta_+\not=0$.  

\begin{itemize}
\item[(A)] 
For every positive radial solution $u$ of \eqref{eq1} in $B_{R}(0)\setminus \{0\}$ with $R>0$ satisfying 
\begin{equation} \label{sofi1}
\lim_{|x|\to 0} |x|^{\Theta_+} u(x)=\gamma\in (0,\infty),
\end{equation}
we have the following:
\begin{itemize}
	\item[(I)] If $\chi>0$, then there exists a constant $\mathfrak a_1\in \mathbb R$ such that 
		\begin{equation} \label{miom}
	u(r)=\gamma\, r^{-\Theta_+}+\mathfrak{a}_1\,r^{-\Theta-}+(\mathfrak{a}_2/\chi)\,r^{\chi-\Theta_-}(1+o(1))\quad  \mbox{as } r\to 0^+. 
		\end{equation} 
	\item[(II)] If $\chi<0$, then 
	\begin{equation} \label{ros22}	
	u(r)=\gamma\,r^{-\Theta_+}+(\mathfrak{a}_2/\chi)\,r^{\chi-\Theta_-}(1+o(1))\  \mbox{ as } r\to 0^+.
		\end{equation}
		\item[(III)] If $\chi=0$, then 
	\begin{equation} \label{rosaa3}
	u(r)=\gamma\,r^{-\Theta_+}+\mathfrak{a}_2\,r^{-\Theta_-}\log r\,(1+o(1))\quad \mbox{as } r\to 0^+.
	\end{equation}
	\end{itemize}	
\item[(B)] For every $\gamma\in (0,\infty)$, there exists $R>0$ small such that \eqref{eq1} in $B_R(0)\setminus\{0\}$ has infinitely many positive radial solutions satisfying \eqref{sofi1}. Moreover, when $\chi>0$, then for every $\gamma\in (0,\infty)$ and each $\mathfrak{a}_1\in \mathbb R$, there exists 
$R>0$ such that \eqref{eq1} in $B_R(0)\setminus \{0\}$ has a positive radial solution $u$ satisfying \eqref{miom}. 	
\end{itemize}	
\end{theorem}

The case $m=0$, $\rho=(2-N)/2$ and $\lambda<(N-2)^2/4$ in our Theorem~\ref{prof1} (corresponding to \eqref{miom}) was considered 
in Proposition 3.4 (a) of \cite{Cmem}, although the third asymptotic
term in the expansion of $u(r)$ as $r\to 0^+$ was not determined. 
We point out that even in such a particular framework, we could not find any previous works that correspond to $\chi\leq 0$ in Theorem~\ref{prof1}. 
		
%%%

\begin{theorem}[Profile $\Phi_{\rho,\lambda}^+$, $\lambda=\rho^2$] \label{th81}
	Let \eqref{cond1} hold, $\rho\in \mathbb R$, $\lambda=\rho^2$ and $\Theta>\Theta_\pm=-\rho$. 
	
\begin{itemize}
\item [(A)]	Let $u$ be any positive radial solution of \eqref{eq1} in $B_{R}(0)\setminus \{0\}$ for $R>0$ such that 
\begin{equation} \label{minaaa} 
\lim_{|x|\to 0}\frac{u(x)}{|x|^\rho \log\, (1/|x|)}=\gamma\in (0,\infty). \end{equation} 
Then, by defining $\vartheta=\kappa \left(\Theta+\rho\right)$, there exists $\mathfrak{c}_0\in \mathbb R$ such that as $r\to 0^+$, we have
\begin{equation*} 
u(r)=\left\{ \begin{aligned}
&\gamma \,r^{\rho}\log \,(1/r)+\mathfrak{c}_0\,r^\rho+\frac{\gamma^{q+m} |\rho|^m}{\vartheta^2} \,r^{\vartheta+\rho}\left|\log r\right|^{m+q}(1+o(1))&& \mbox{if } \rho\not=0,&	\\	
& \gamma\,\log \, (1/r)+\mathfrak{c}_0+
\frac{\gamma^{q+m}}{\vartheta^2} \,r^{\vartheta}\left| \log r\right|^q (1+o(1)) && \mbox{if } \rho=0.& 
\end{aligned} \right.
\end{equation*}
\item [(B)]For every $\gamma\in (0,\infty)$, there exists $R>0$ small such that \eqref{eq1} in $B_R(0)\setminus\{0\}$ has infinitely many positive radial solutions satisfying \eqref{minaaa}.
\end{itemize} 	
\end{theorem}
	
The case $m=0$ and $\rho=(2-N)/2$ in Theorem~\ref{th81} is addressed in \cite[Proposition 3.4 (a)]{Cmem}, although the third asymptotic
term in the expansion of $u(r)$ as $r\to 0^+$ remained open there.

%%%

\vspace{0.2cm}
We conclude this subsection with Theorem~\ref{mix1}, which assumes that $\lambda\leq \rho^2$ and $\Theta>\Theta_-\not=0$ and establishes 
the existence %and the two-term asymptotic behaviour near zero 
of a positive radial solution of \eqref{eq1} satisfying 	
	\begin{equation} 
		\label{mina}
	\lim_{|x|\to 0} |x|^{\Theta_-} u(x)=\gamma\in (0,\infty). 
		\end{equation}
We also reveal the second term in the asymptotic expansion near zero of any such solution $u$.  
		
%%%%%
\begin{theorem}[Profile $\Phi_{\rho,\lambda}^-$]
	\label{mix1}
	Let \eqref{cond1} hold, $\lambda\leq \rho^2$ and $\Theta>\Theta_-\not=0$. 
	\begin{itemize}
\item[(I)] Let $u$ be any positive radial solution of \eqref{eq1} in $B_{R}(0)\setminus \{0\}$ for $R>0$ satisfying \eqref{mina}. 
	
We define  $\vartheta=\kappa\left(\Theta-\Theta_-\right)>0$. Then, as $r\to 0^+$,  
\begin{equation} \label{unoa}
 r^{\Theta_-} u(r)=\gamma+ C_\gamma \, r^{\vartheta} (1+o(1)),\quad
 \mbox{where } C_\gamma:=\frac{\gamma^{q+m}\,|\Theta_-|^m}
{\vartheta\left(\vartheta+2\sqrt{\rho^2-\lambda}\right)}>0.
 \end{equation}

\item[(II)]  
For every $\gamma\in (0,\infty)$, 
there exists $R>0$ such that \eqref{eq1} in $B_R(0)\setminus \{0\}$  has a positive radial solution $u$ satisfying \eqref{unoa}. 
\end{itemize}
\end{theorem}

The case $m=0$ and $\rho=(2-N)/2$ in our Theorem~\ref{mix1} is known from  \cite{Cmem}, see
Proposition~3.1 (b) when $\lambda<\rho^2$ and Proposition 3.4 (b) when $\lambda=\rho^2$.  
Our approach, however, is very different for both the existence and asymptotics given that changes of variable akin to those in
\cite{Cmem} don't work in our case of gradient-dependent nonlinearities. 		

%%%%

\subsection{Critical cases: competition between $\mathbb L_{\rho,\lambda}$ and the absorption}  \label{subsec3}

In Table \ref{tabel3}, we present the profiles that appear when considering the 
borderline cases to those in Table \ref{tabel2}, i.e., $\Theta=\Theta_-$ and $\lambda\leq \rho^2$. For each case, we establish the existence of positive radial solutions to \eqref{eq1},  with the prescribed behaviour at the origin, supplemented by refined asymptotics near zero (see Theorems \ref{nom}--\ref{zzz0}). 
Our results complement previous works on \eqref{eq1} known only for the cases $m=0$ and 
$\rho=(2-N)/2$ due to \cite{Cmem} and \cite{MF1}. 

%The profiles listed in Table \ref{tabel3} are only known for $m=0$ and  from , see Corollary 7.4 and Corollary 7.5. %However, their asymptotic expansions were not previously available. 

\begin{table}[H]
\centering
\caption{Critical cases: competition between $\mathbb L_{\rho,\lambda}$ and the absorption}
\label{tabel3}
\small
\begin{tabular}{|>{\raggedright\arraybackslash}p{7cm}|p{4cm}|p{2cm}|}
\hline
Definition and sharp range for the profile & Asymptotic profile & Theorems \\
\hline

\makecell[l]{
$V_0(|x|):=\left(\frac{2\sqrt{\rho^2-\lambda}}{\kappa\,|\Theta_-|^{m} } \right)^{\frac{1}{\kappa}}
|x|^{-\Theta_-}\left(\log \frac{1}{|x|}\right)^{-\frac{1}{\kappa}}$\\
where $\lambda<\rho^2, \Theta=\Theta_-\neq0$
}
&
\begin{tabular*}{4cm}{@{}l@{\extracolsep{\fill}}r@{}}
$\displaystyle \lim_{|x|\to 0}\frac{u(x)}{V_0(|x|)}=1$ %& \AddLabel{mino}
\end{tabular*}
& Theorem \ref{nom} \\ \hline

\makecell[l]{
$W_0(|x|):=\left( \frac{2(\kappa+2)}{\kappa^2\, |\Theta_-|^{m}} \right)^{\frac{1}{\kappa}}
 |x|^{-\Theta_-}\left(\log \frac{1}{|x|}\right)^{-\frac{2}{\kappa}}$\\
where $\lambda=\rho^2, \Theta=\Theta_\pm=-\rho\neq0$
}
&
\begin{tabular*}{4cm}{@{}l@{\extracolsep{\fill}}r@{}}
$\displaystyle \lim_{|x|\to 0}\frac{u(x)}{W_0(|x|)}=1$ %& \AddLabel{wzer}
\end{tabular*}
& Theorem \ref{noma} \\ \hline

\makecell[l]{
$Y_0(|x|):= (2|\rho| )^{\frac{1}{\kappa}}
\left( \frac{|m-1|}{\kappa} \right)^{\frac{1-m}{\kappa}}
\left( \log \frac{1}{|x|}\right)^{\frac{m-1}{\kappa}}$\\
where $\lambda=\Theta=0$ and $\rho(m-1)>0$
}
&
\begin{tabular*}{4cm}{@{}l@{\extracolsep{\fill}}r@{}}
$\displaystyle \lim_{|x|\to 0}\frac{u(x)}{Y_0(|x|)}=1$ %& \AddLabel{yoo}
\end{tabular*}
& Theorem \ref{ytam} \\ \hline

\makecell[l]{
$Z_0(|x|):=\left( \frac{(q+1)(2-m)^{1-m}}{\kappa^{2-m}}\right)^{\frac{1}{\kappa}}
\left( \log \frac{1}{|x|}\right)^{\frac{m-2}{\kappa}}$\\
where $\lambda=\Theta=\rho=0$ and $m\in (0,2)$
}
&
\begin{tabular*}{4cm}{@{}l@{\extracolsep{\fill}}r@{}}
$\displaystyle \lim_{|x|\to 0}\frac{u(x)}{Z_0(|x|)}=1$ %& \AddLabel{zoo}
\end{tabular*}
& Theorem \ref{zzz0} \\ \hline

\end{tabular}
\end{table}

\begin{theorem}[Profile $V_0$] \label{nom} 
Let \eqref{cond1} hold, $\rho \in \mathbb{R}$, $\lambda<\rho^2$ and $\Theta=\Theta_-\neq 0$. 
We define
\begin{equation} \label{vodeg}
\mathfrak{M}:=\frac{\kappa+1}{2\sqrt{\rho^2-\lambda}}+\frac{m}{\Theta_-}.
\end{equation}
For every $R>0$, \eqref{eq1} in $B_R(0)\setminus \{0\}$ has infinitely many positive radial solutions satisfying 
\begin{equation} \label{mino} \lim_{|x|\to 0}\frac{u(x)}{V_0(|x|)}=1.
\end{equation}
Let $u$ be any positive radial solution of \eqref{eq1} in $B_R(0)\setminus \{0\}$ for $R>0$ satisfying \eqref{mino}.

$\bullet$ If $\mathfrak M=0$ and $m=1$, then there exists a constant $C\in \mathbb R$ such that
	\begin{equation} \label{vezi} \frac{u(r)}{V_{0}(r)}= \left( 1+\frac{C}{\log\, (1/r)}\right)^{-\frac{1}{\kappa}} \quad \mbox{for every } r>0\ \mbox{small}. 
	\end{equation}

$\bullet$ If $\mathfrak M\neq0$, then
\begin{equation} \label{azi10}
	\frac{u(r)}{V_{0}(r)}= 1+\frac{\mathfrak M}{\kappa^2}\,\frac{\log \log\, (1/r)}{\log \, (1/r)}(1+o(1))\quad \mbox{as } r\to 0^+.
\end{equation}

$\bullet$ If $\mathfrak M=0$ and $ m\neq 1$, then there exists a constant $C\in \mathbb R$ such that as $r\rightarrow 0^+$, 
\begin{equation} \label{sac1}
	\frac{u(r)}{V_{0}(r)}= 1+ \frac{C}{\log \,(1/r)}+
	\left[ \frac{m\left(m-1\right)}{2\kappa^3\,(\Theta_-)^2}+\frac{\left(\kappa+1\right) C^2}{2}\right] \frac{1}{\log^2 \,(1/r)}\,(1+o(1)) .
	\end{equation}
\end{theorem}

\begin{theorem}[Profile $W_0$] \label{noma} 
Let \eqref{cond1} hold, $\rho \in \mathbb{R}\setminus\{0\}$, $\lambda=\rho^2$ and $\Theta=\Theta_-=-\rho$. For every $R>0$, equation \eqref{eq1} in $B_R(0)\setminus \{0\}$ has infinitely many positive radial solutions satisfying 
\begin{equation} \label{wzer}  
\lim_{|x|\to 0}\frac{u(x)}{W_0(|x|)}=1. \end{equation}
If $u$ is any positive radial solution of \eqref{eq1} in $B_R(0)\setminus \{0\}$ for $R>0$ such that \eqref{wzer} holds, then 
	\begin{equation} \label{samiii}
	\frac{u(r)}{W_0(r)}=1-\frac{4m\left(2+\kappa\right)}{\rho\kappa^2\left(3\kappa+4\right)}\,\frac{\log \log \, (1/r)}{\log \, (1/r) }(1+o(1)) \quad \mbox{as } r\to 0^+.
	\end{equation}
\end{theorem}

\begin{theorem}[Profile $Y_{0}$] \label{ytam}
Let \eqref{cond1} hold, $\rho \in \mathbb{R}\setminus\{0\}$, $\lambda=\Theta= 0$ and $\rho\left(m-1\right)>0$. 
For every $R>0$, equation \eqref{eq1} in $B_R(0)\setminus \{0\}$ has infinitely many positive radial solutions satisfying 
\begin{equation} \label{yoo} \lim_{|x|\to 0}\frac{u(x)}{Y_0(|x|)}=1. \end{equation}
Let $u$ be any positive radial solution of \eqref{eq1} in $B_R(0)\setminus \{0\}$ for $R>0$ satisfying \eqref{yoo}. 

$\bullet$ If $q\not=0$, then as $r\to 0^+$, we have
\begin{equation} \label{more1}
\frac{u(r)}{Y_0(r)}= 1-\frac{q}{2\rho\,\kappa^2}\,\frac{ \log \log \,(1/r) } {\log \,(1/r)}(1+o(1)).			
\end{equation}

$\bullet$ If $q=0$, then necessarily $m>1$ and there exists a constant $C\in \mathbb R$ such that
\begin{equation} \label{cec01} u(r)=(2\rho)^{\frac{1}{m-1}} \log \,(1/r) +C\quad \mbox{for every } r>0\ \mbox{small}.  
\end{equation}
\end{theorem}

\begin{theorem}[Profile $Z_0$] \label{zzz0} Let \eqref{cond1} hold, $\lambda=\Theta=\rho=0$ and $m \in (0,2)$. 
Let $u$ be any positive radial solution of \eqref{eq1} in $B_R(0)\setminus \{0\}$ for $R>0$
such that 
\begin{equation} \label{zoo} \lim_{|x|\to 0}\frac{u(x)}{Z_0(|x|)}=1. \end{equation}
Then, there exists a constant $\eta\geq 0$ such that 
			\begin{equation} \label{lama}
			u(r)=u_\eta(r)=\left[ \frac{(2-m)^{1-m}(q+1)}{\kappa^{2-m}}\right]^\frac{1}{\kappa}
			\left( \log \frac{R}{r}+\eta\right)^{-\frac{2-m}{\kappa}}
			\quad \mbox{for all } r\in (0,R).
			\end{equation} 
		Hence, for every $R>0$, equation \eqref{eq1} in $B_R(0)\setminus \{0\}$ has infinitely many positive radial solutions satisfying \eqref{zoo}. 
\end{theorem}

%\subsection{Comparison with the literature}

%\subsection{Applications of our main results}

\subsection{Discussion on our results} 
We discuss (1) the motivation for our study of radial solutions, (2) the novelty and strategy of our proofs and (3) applications of our methods and results.

%For a generalisation of the results in \cite{Cmem,MF1} to quasilinear elliptic equations with Hardy potential and an absorption term of the form $|x|^\theta u^q$, we refer to \cite{BiCh}. 

\vspace{0.2cm}
{\em 1. Motivation for the study of radial solutions.} Our focus on the positive and radially symmetric solutions of \eqref{eq1} is motivated by recent findings in \cite{MF1} and \cite{CCM}.
In \cite{MF1} the authors have fully classified all positive solutions for \eqref{eq1} in $\mathbb R^N\setminus \{0\}$ when $m=0$ and $\rho=(2-N)/2$, namely
\begin{equation} \label{mubb1} \Delta u+\frac{\lambda }{|x|^2} u=|x|^\theta u^q\quad \mbox{in }\mathbb R^N\setminus \{0\}, 
\end{equation} where $\lambda,\theta\in \mathbb R$, $N\geq 3$ and $q>1$. 
As shown in \cite{MF1}, \eqref{mubb1} has positive solutions if and only if  $$\lambda>\lambda^*=\frac{\theta+2}{q-1}\left(N-2-\frac{\theta+2}{q-1}\right).$$ 
Under this sharp condition, 	
	every positive solution of \eqref{mubb1} is {\em radially symmetric}.
For a recent generalisation of the classification results in \cites{Cmem,MF1} to quasilinear elliptic equations with Hardy potential and an absorption term of the form $|x|^\theta u^q$, we refer to \cite{BiCh}.	
	
	The above phenomenon of radial symmetry for all positive solutions of \eqref{mubb1} in $\mathbb R^N\setminus \{0\}$ applies to related but different problems recently considered in \cite{CCM} such as 
	\begin{equation} \label{unu1}\mathbb L_{\rho,\lambda,\tau}[u]:=  \Delta u-(N-2+2\rho) \frac{x\cdot \nabla u}{|x|^2} +\lambda \frac{u^\tau |\nabla u|^{1-\tau}}{|x|^{1+\tau}}=|x|^\theta u^q\quad \mbox{in }
\mathbb R^N\setminus \{0\}, 
\end{equation}
where $\rho,\lambda, \theta\in \mathbb R$ are arbitrary, $N\geq 2$, $q>1$ and $\tau\in [0,1)$. 
It is proved in \cite{CCM} that \eqref{unu1} has positive $C^1(\mathbb R^N\setminus \{0\})$ distributional solutions if and only if 
$$ \frac{\theta+2}{q-1}\left( \frac{\theta+2}{q-1} +2\rho\right) +\lambda |t|^{1-\tau}>0.$$ 
Under this condition, all such solutions are   
{\em radially symmetric};  the existence and 
exact asymptotic behaviour near zero and at infinity for these solutions are also provided. 

\vspace{0.2cm}
{\em 2. Novelty and strategy of the proofs.}
Our results establish the existence and refined asymptotics of the positive radial solutions of \eqref{eq1} in $B_{R}(0)\setminus \{0\}$ with
a gradient-dependent absorption $|x|^\theta u^q |\nabla u|^m$ for arbitrary $m>0$. Our approach here is totally different from that in previous works  on particular cases of \eqref{eq1}
when $\rho=(2-N)/2$ and either $m=0$ (\cites{Cmem,MF1,LD2017}) or 
$m\in (0,2)$ and $\lambda=\theta=0$ (\cite{CC2015}). 
The methods for proving the existence results in \cites{Cmem,MF1,CC2015} could be adapted to cover only the range $m\in (0,2)$ for problem \eqref{eq1} when subject to a Dirichlet boundary condition on $\partial \Omega$ and $\Omega$ a smooth bounded domain in $\mathbb R^N$. The restriction $m<2$ appears when obtaining  
gradient estimates and the solution as the limit of a sequence of positive solutions for approximating boundary value problems.

In proving our main results, we use a different strategy 
as we do not impose any 
boundary condition on $\partial B_R(0)$. 
%which allows us to often obtain multiplicity results for suitable small $R>0$.    
For each possible asymptotic profile near zero (see Tables~\ref{tabel1}--\ref{tabel3}), 
we first refine the behaviour near zero for {\em all} positive radial solutions of \eqref{eq1}. Then, essentially use this information to reduce the
existence question for \eqref{eq1} to that for a
first order ODE of the form 
$X'(t)=f(t,X(t))$ on an interval $(T_0,\infty)$ with
$\lim_{t\to \infty} X(t)=0$.  
The key innovation and difficulty is to apply the {\em refined asymptotics} knowledge at zero for the radial solutions $u$ of \eqref{eq1} to design $t$ and $X(t)$ in terms of $u(r)$, $r=|x|$ and $u'(r)$ so that $t\to \infty$ and $X(t)\to 0$ as $r\to 0^+$.  
Then, the existence (resp,  multiplicity) of the radial solutions of \eqref{eq1} with the desired behaviour near zero follows from the
existence (resp., multiplicity) of solutions $X(t)$ for 
\begin{equation} \label{cumva} X'(t)=f(t,X(t))\quad \mbox{with } t\geq T_0>0.\end{equation}

In our case, if using only the dominant behaviour in the profile to create a dynamical system with three variables often leads to non-hyperbolic critical points when the stable and center manifold theory cannot be directly applied in their standard form. We illustrate this with the proof of Theorem~\ref{th81}, which is based on 
Theorem 7.1 in the Appendix of \cite{CRV}. Otherwise, we succeed in giving a  
different adaptation of the dynamical systems method as explained above leading to one first order ODE, see \eqref{cumva}. (This could be written as a 
system of two autonomous first order  ODEs in which one is explicitly solved.) 

The study of the positive radial solutions of \eqref{eq1} in $B_R(0)\setminus \{0\}$ reduces to that of the ODE: 
\begin{equation} \label{ez} u''(r)+(1-2\rho)\frac{u'(r)}{r}+\frac{\lambda}{r^2}\,u(r)=r^\theta u^q(r)|u'(r)|^m\quad \mbox{for all } r\in (0,R).  
	\end{equation}
	
In designing $t$ and $X(t)$, we often use that any positive solution of \eqref{ez} satisfies 
\begin{equation} \label{muzia} r \frac{d}{dr} \left(\frac{ru'(r)}{u(r)}\right) =
2\rho \, \frac{ru'(r)}{u(r)}-\lambda -\left(\frac{ru'(r)}{u(r)}\right)^2 +(r^\Theta u(r))^\kappa \left|\frac{ru'(r)}{u(r)}\right|^m 
\end{equation}	
for every $r\in (0,R)$. 

\vspace{0.2cm}
{\em 3. Applications of our methods and results.} 
We point out that our methods for obtaining local existence of radial solutions with a prescribed behaviour near zero have wider applicability.
Below, we mention distinct classes of problems. 

(A) For \eqref{unu1}, one can use 
our approach here, in conjunction with 
the complete classification
of the local behaviour near zero obtained in \cite{CCM}, to establish local existence and refined asymptotics of the positive radial solutions. 
This line of investigation will be pursued elsewhere.  

(B) Our methods could be adapted to provide new insights on the singular solutions of (ND)-type discovered in  \cites{CR,CRV}  for 
a different class of problems of the form
\begin{equation} \label{flfr} -\Delta u=\frac{u^{2^\star(s)-1}}{|x|^s}-\mu u^q\quad \mbox{in } B_R(0)\setminus \{0\}.
\end{equation} Here, $B_R(0)\subset \mathbb R^N$ with $N\geq 3$, $s\in (0,2)$, $2^\star(s)=2(N-s)/(N-2)$, $\mu>0$ and $q>1$. 
By Theorem 1.3 in \cite{CRV}, if $2^\star(s)-1<q<2^*-1$, then there exists $R>0$ such that \eqref{flfr} has infinitely many positive radial solutions of (ND) type (for "Non Differential"), i.e., % solutions 
satisfying 
$$ \lim_{|x|\to 0}  |x|^{\frac{s}{q-(2^\star(s)-1)}}  u(x)=\mu^{-\frac{1}{q-(2^\star(s)-1)}}.
$$
The (ND) profile, which can only appear in the range $2^\star(s)-1<q<2^*-1$ (see \cite{CR}), is generated by the competition between the terms in the right-hand side of \eqref{flfr}.
Like in the case of our first two  profiles outlined in Table~\ref{tabel1}, the Laplacian has no visible effect on the leading term in the asymptotics of the (ND) solutions.  Our strategy could be used to refine the asymptotics near zero of the radial solutions of (ND) type and then conclude the existence of such solutions. 

\vspace{0.2cm}
(C) Our main results %provide precise knowledge of the radial solutions
help us decipher the radial solutions of two other classes of nonlinear problems. 
For $N\geq 2$ and $\Omega=B_R(0)$ (or $\Omega=\mathbb R^N\setminus \{0\}$), we consider the radial solutions $w$ of 
 \begin{equation} \label{rigo1}
\Delta w+|\nabla w|^2 +\mu \frac{x\cdot \nabla w}{|x|^2} +\frac{\lambda}{|x|^{2}} =|x|^\theta e^{\kappa w}|\nabla w|^m\quad \mbox{in } 
\Omega\setminus \{0\},
\end{equation}
where $\mu,\lambda,\theta\in \mathbb R$ and $\kappa,m>0$, resp., the positive radial solutions $v$ of 
\begin{equation} \label{rigo2}
	\Delta (v^\alpha) +\mu \frac{x\cdot \nabla (v^\alpha)}{|x|^2} 
	+\lambda \frac{v^{\alpha}}{|x|^{2}} =|x|^\theta |\nabla v|^{m}\quad \mbox{in }\Omega\setminus \{0\},
\end{equation}
where $\mu,\lambda,\theta\in \mathbb R$, while $\alpha\in (0,1]$ and $m>\alpha$. For both problems, we let $\rho=(2-N-\mu)/2$. 

From our main results (Theorems~\ref{greu1}--\ref{zzz0}) with $u=e^w$ and $q=\kappa-m+1$ (resp., 
$u=\alpha^{m\alpha/(\alpha-m)} v^\alpha$ with $q=m(1-\alpha)/\alpha$), we obtain the existence and precise asymptotic behaviour near zero and at infinity of the radial solutions of \eqref{rigo1} (resp., positive radial solutions of \eqref{rigo2}).

Problems of the form \eqref{rigo2} are the stationary versions of parabolic equations with fast diffusion ($\alpha\in (0,1)$), featuring gradient-dependent absorption 
($|x|^\theta |\nabla v|^m$). The fast diffusion equation, $u_t=\Delta (u^\alpha)$ with $\alpha\in (0,1)$, is an important model for singular nonlinear diffusion phenomena, including gas-kinetics, thin liquid film dynamics and diffusion in plasmas (see \cites{das,Vazq}).  
Degenerate and singular parabolic equations with absorption have been intensively studied by many authors. The competition between
the diffusion $\Delta (v^\alpha)$ and the absorption leads to different dynamics and  large time behaviour of non-negative solutions for various regimes 
 of the exponents (see, e.g.,  \cites{Iag1,BIL,IL} and their references).

\vspace{0.2cm}
\noindent {\bf Structure of the paper.} In what follows, we prove the main results in separate sections.
% is dedicated to proving the main results. 

\section{The constant asymptotic profile when $\Theta<0<\lambda$ (Proof of Theorem~\ref{greu1})}

Let \eqref{cond1} hold, $\rho\in \mathbb{R}$, $\Theta<0$ and $\lambda>0$. Unless otherwise stated, $u$ is any positive radial solution of \eqref{eq1} in $B_R\{0\}\setminus \{0\}$ for $R>0$ satisfying \eqref{consta}, that is 
\begin{equation}\label{repet}
\lim_{r\rightarrow 0^+}u(r)=\gamma\in (0,\infty). 
\end{equation}

\begin{lemma}\label{nou1}
For every  $r>0$ small, it holds 
\begin{equation}\label{fezero} 
u'(r) \, \mathfrak B'(r)>0,\ \mbox{where } \mathfrak B(r)=\frac{ru'(r)}{u(r)}.
\end{equation}
Furthermore, 
we have $\lim_{r\rightarrow 0^+} \mathfrak B(r)=0$.
\end{lemma}

\begin{proof}
Observe that for $r_0>0$ small, we have $$u'(r)\neq 0\quad \mbox{for every }r\in (0, r_0).$$ 
Suppose the contrary, i.e., $u'(r_n)=0$ for a sequence $\{r_n\}_{n\geq 1}$ of positive numbers decreasing to $0$ as $n\to \infty$. 
Then, $ u''(r_n)<0$ for every $n\geq 1$. Hence, we get a contradiction as $u''(r_n)$ would have a constant sign for all $n\geq 1$.  

Recall from \eqref{muzia} that for every $r\in (0,R)$, 
 $$ r \mathfrak B'(r)=-\lambda+2\rho \,\mathfrak B(r)-\mathfrak B^2(r)+r^{\kappa \Theta}u^{\kappa}(r)\,|\mathfrak B(r)|^m.$$
Suppose by contradiction that there exists a sequence $\{r_n\}_{n\geq 1}$ in $(0,r_0)$ decreasing to $0$ as $n\rightarrow \infty$ such that $\mathfrak B'(r_n)=0$ for all $n\geq 1$. 
Then, for all $n\geq 1$, we find that 
$$r_n^2\mathfrak B''(r_n)=\kappa \left [\Theta+\mathfrak B(r_n)\right ]r_n^{\kappa \Theta}u^\kappa(r_n)|\mathfrak B(r_n)|^m.$$

Since $u'$ and, hence, $\mathfrak B$ has a constant sign on $(0,r_0)$, we have two situations:

{\bf (I)} Let  $\mathfrak B(r)>0$ for every $r\in (0,r_0)$. Then, we distinguish two cases:

(a) Let $m\geq 2$. Then, from $\mathfrak B'(r_n)=0$ for all $n\geq 1$, we infer that $\{\mathfrak B(r_n)\}_{n\geq 1}$ is bounded and $\lim_{n\to \infty} \mathfrak B(r_n)=0$. Thus,
 $\mathfrak B''(r_n)<0$ for all $n\geq 1$ large, which is impossible. 

(b) Let $m<2$. In this case, if $\{\mathfrak B(r_{n_j})\}_{j\geq 1}$ has a limit in $\mathbb R\cup\{+\infty\}$ for some subsequence $\{r_{n_j}\}_{j\geq 1}$ of $\{r_n\}$, then
$\lim_{j\to \infty} \mathfrak B(r_{n_j})$ is either $0$ or $\infty$. Therefore, one of the following applies:

(1) $\lim_{n\to \infty} \mathfrak B(r_n)=0$. Here, $\mathfrak B''(r_n)<0$ for all $n\geq 1$ large, which is impossible. 

(2) $\lim_{n\to \infty} \mathfrak B(r_n)=\infty$. Then, $\mathfrak B''(r_n)>0$ for all $n\geq 1$ large, which is impossible.

(3) There exist subsequences $\{r_{n_j}\}_{j\geq 1}$ and $\{\widetilde r_{n_j}\}_{j\geq1 }$ of $\{r_n\}$ such that 
$$ \mathfrak B(r_{n_j})\to 0\quad \mbox{and}\quad  
\mathfrak B(\widetilde r_{n_j})\to \infty\ \mbox{as } j\to \infty.$$
Then, for $j\geq 1$ large enough, we have $\mathfrak B''(r_{n_j})<0<\mathfrak B''(\widetilde r_{n_j})$, implying that  $r_{n_j}$ are local maxima points of $\mathfrak B(r)$, while $\widetilde r_{n_j}$ are local minima points of $\mathfrak B(r)$. This is impossible. 

{\bf (II)} Let $\mathfrak B(r)<0$ for every $r\in (0,r_0)$. Then, $\mathfrak B''(r_n)<0$ for all $n\geq 1$, which means that all critical points of $\mathfrak B$ on $(0,r_0)$ are local maxima points.
This is impossible.  

\vspace{0.2cm}
Thus, in either Case (I) or Case (II), we have ${\rm sgn}\,(\mathfrak B'(r))={\rm sgn} \,(u'(r))\not=0$ for all $r>0$ small. 
Hence, $\lim_{r\to 0^+}\mathfrak B(r)=0$ in view of \eqref{repet}. This ends the proof of Lemma~\ref{nou1}. 
\end{proof}

Let $\mathcal C=\pm \lambda^{1/m}$, where ${\rm sgn}\,(\mathcal C)={\rm sgn} \, (u'(r))$ for $r>0$ small. 
For each $r>0$ small, we define 
\begin{equation}\label{ecot}
t=r^{-\sigma}(u(r))^{\frac{\kappa}{m}}\ \mbox{with } \sigma=-\frac{\kappa \Theta}{m}>0
\quad\text{and}\quad X(t)=t\frac{ru'(r)}{u(r)}-\mathcal C.
\end{equation}

\begin{lemma} \label{gara1}
Then, $t\rightarrow \infty$ and $X(t)\rightarrow 0$ as $r\rightarrow 0^+$. Moreover, there exists $T_0>0$ large such that for all 
$t\geq T_0$, we have $X(t)+\mathcal C\not=0$ and 
\begin{equation}\label{1xeq}
\frac{dX}{dt}=\frac{\left (\sigma-2\rho\right )\frac{X(t)+\mathcal C}{t}+\frac{1-q}{m}\left (\frac{X(t)+\mathcal C}{t}\right )^2+\lambda\left [1-\left |1+\frac{X(t)}{\mathcal C}\right |^m\right ]}{\sigma-\frac{\kappa}{m}\frac{X(t)+\mathcal C}{t}}.
\end{equation}
Moreover, if $u'(r)>0$ (resp., $u'(r)<0$) 
for all $r>0$ small, then $u$ satisfies 
$(P_+)$ (resp., $(P_-)$) given by
\begin{equation} \label{frio} \tag{$P_{\pm}$} 
u(r)=\gamma \left( 1\pm \frac{\lambda^{\frac{1}{m}}}{\sigma \gamma^{\frac{\kappa}{m}}} r^\sigma (1+o(1)) \right) \ \mbox{as } r\to 0^+.
\end{equation} 
\end{lemma}

\begin{proof}
By \eqref{repet}, \eqref{ecot} and $\Theta<0$, we have $t\rightarrow \infty$ as $r\rightarrow 0^+$.  
Using Lemma~\ref{nou1}, we get 
\begin{equation} \label{main2} \frac{X(t)+\mathcal C}{t}\rightarrow 0\ \mbox{as } r\rightarrow 0^+\ \ \mbox{and } X(t)+\mathcal C\neq 0\ \mbox{for every } t>0 \ \mbox{large}. \end{equation}
Hence, there exists $T_0>0$ large such that for every $t\geq T_0$, we have 
$$ |X(t)+\mathcal C|>0\quad \mbox{and} \quad \sigma-\frac{\kappa}{m}\frac{X(t)+\mathcal C}{t}>0.
$$
Then, by \eqref{ecot}, for every $t\geq T_0$, we find that 
\begin{equation}\label{real2}
\frac{dr}{dt}=-\frac{r}{t\left (\sigma-\frac{\kappa}{m}\frac{X(t)+\mathcal C}{t}\right )}.
\end{equation}
Then, direct computations yield that \eqref{1xeq} holds. 

We assume one of the following situations:

$(i_1)$ Let $\sigma\not=2\rho$;

$(i_2)$ Let $\sigma=2\rho$ and $q\not=1$.

($i_3$) Let $\sigma=2\rho$ and $q=1$.  

\vspace{0.2cm}
We first show that if ($i_1$) or $(i_2)$ holds, then $X'(t)\neq 0$ for every $t>0$ large.
Suppose by contradiction that for an increasing sequence $
\{t_n\}_{n\geq 1}$ in $[T_0, \infty)$ with $\lim_{n\rightarrow \infty}t_n=\infty$, we have $X'(t_n)=0$ for all $n\geq 1$. By differentiating \eqref{1xeq} with respect to $t$, we obtain that 
$$t_n^2X''(t_n)=-\frac{X(t_n)+\mathcal C}{\sigma-\frac{\kappa}{m}\frac{X(t_n)+\mathcal C}{t_n}}\left [\sigma-2\rho+\frac{2\left(1-q\right)}{m}\frac{X(t_n)+\mathcal C}{t_n}\right ].$$
In light of \eqref{main2}, we get 
\begin{equation}
\begin{aligned}
&\lim_{n\rightarrow \infty} \frac{t_n^2X''(t_n)}{X(t_n)+\mathcal C}=-\frac{\sigma-2\rho}{\sigma}\neq 0&\quad&\text{if } \sigma\neq 2\rho\\
&\lim_{n\rightarrow \infty} \frac{t_n^3X''(t_n)}{(X(t_n)+\mathcal C)^2}=-\frac{2\left(1-q\right)}{m \,\sigma}\neq 0&\quad&\text{if } \sigma=2\rho\text{ and }q\neq 1.
\end{aligned}
\end{equation}

In each of these cases, we get a contradiction as $X''(t_n)$ has a constant sign for every $n\geq 1$. Hence, 
$X'(t)\not=0$ for every large $t>0$ and there exists $\lim_{t\to \infty} X(t)$, which we show to be $0$.   

\vspace{0.2cm}
Assume now that $(i_3)$ holds, i.e., $\sigma=2\rho$ and $q=1$. Note that if there exists an increasing sequence 
$\{t_n\}_{n\geq 1}$ in $[T_0,\infty)$ with $\lim_{n\to \infty} t_n=\infty$ and $X(t_n)=0$ for all $n\geq 1$, then 
$X\equiv 0$ on $[t_1,\infty)$. Otherwise, 
$X(t)\neq 0$ for any $t>0$ large enough, which implies that $X'(t)\neq 0$. 

Thus, in any of the situations $(i_1)$--$(i_3)$, there exists $\lim_{t\to \infty} X(t)=L$. We show that $L=0$. 
Indeed, if $L=\pm \infty$, then from \eqref{main2}, L'H\^opital's and \eqref{1xeq}, we see that
$$ 0=\lim_{t\to \infty} \frac{X(t)}{t} =\lim_{t\to \infty} X'(t)=-\infty.
$$
This is impossible. Hence, $L\in \mathbb R$ and again using \eqref{1xeq}, we have $L=0$ to avoid a contradiction. 
Moreover, in view of \eqref{ecot}, we arrive at 
\begin{equation} \label{simbi} \frac{m}{\kappa}\frac{\frac{d}{dr}\left ( u^{\frac{\kappa}{m}}(r)\right)}{\frac{d}{dr} (r^\sigma/\sigma)}=r^{1-\sigma}(u(r))^{\frac{\kappa}{m}-1}u'(r)=
X(t)+\mathcal C\to \mathcal C\quad \mbox{as } r\to 0^+.\end{equation}
Hence, using \eqref{repet}, we conclude that 
$$ \lim_{r\to 0^+} \frac{u^{\frac{\kappa}{m} } -\gamma^{\frac{\kappa}{m}}}{r^\sigma}=\frac{\kappa\,\mathcal C}{m\,\sigma},  
$$
from which we infer that 
$u$ satisfies \eqref{frio}, corresponding to $\mathcal C=\pm \lambda^{1/m}$. 
\end{proof}

\begin{lemma} \label{easi}
For every $\gamma\in \mathbb R_+$, there exists $R>0$ such that \eqref{eq1} in $B_R(0)\setminus \{0\}$ has a positive radial solution satisfying $(P_-)$ and infinitely many positive radial solutions satisfying $(P_+)$.
\end{lemma}

\begin{proof}   
We first consider the case $\mathcal C=-\lambda^{1/m}$ in \eqref{1xeq}. Since \eqref{eq1} is invariant under the scaling transformation $T_j[u]$ in \eqref{scale}, it is enough to prove the for some $\gamma\in \mathbb R_+$, there exists $R>0$ such that \eqref{eq1} in $B_{R}(0)\setminus \{0\}$ has a positive radial solution satisfying $(P_-)$. For $T_0>0$ large, we consider the ODE in \eqref{1xeq} for $t\geq T_0$. 
Then, there exists $\varepsilon>0$ small 
such that for every $T_0>1/\varepsilon$, 
the equation in \eqref{1xeq} for $t\geq T_0$ has a solution $X$ satisfying $\lim_{t\rightarrow \infty}X(t)=0$.

Fix $r_0>0$ arbitrary. Then, the differential equation for $r$ in \eqref{real2} for $t\in [T_0,\infty)$, subject to $r(T_0)=r_0$, 
has a unique solution given by
	\begin{equation} \label{billi} 
	r(t)=r_0 \exp\left( - \int_{T_0}^ t \frac{ds}{s\left (\sigma-\frac{\kappa}{m}\frac{X(s)+\mathcal C}{s}\right )}\right) \quad \mbox{for } t\geq T_0. 
	\end{equation}
	Since $[T_0,\infty)\ni t\longmapsto r(t)$ is decreasing, we can express $t$ as a function of $r$. 
	We define $u(r)$ by 
	\begin{equation} \label{ba012} u(r)=\left (t\,r^{\sigma}\right )^{\frac{m}{\kappa}}\quad \mbox{for every } r\in (0,r_0].  
	\end{equation} Then, using \eqref{real2}, jointly with \eqref{1xeq} and $\lim_{t\to \infty} X(t)=0$, we obtain that $u$ is a positive radial solution of \eqref{eq1} in $B_{r_0}(0)\setminus \{0\}$. Moreover, we regain \eqref{simbi}, which implies that there exists $\gamma\in (0,\infty) $ such that
	$\lim_{r\to 0^+} u(r)=\gamma$ and 
	$u$	satisfies $(P_-)$.

	\vspace{0.2cm}
	Second, we let $\mathcal C=\lambda^{1/m}$ in \eqref{1xeq}.  
	Here, there exists $\varepsilon>0$ small such that for every $T_0>1/\varepsilon$ and 
all $c\in (-\varepsilon,\varepsilon)$, the ODE in \eqref{1xeq} has a solution $X=X_c$ satisfying $X(T_0)=c$ and $\lim_{t\rightarrow \infty}X(t)=0$. (This statement can be obtained by using the stable manifold theorem for the first order autonomous differential system generated by $X_1(t)=X(t)$ and $X_2(t)=1/t$.)

Fix $T_0>1/\varepsilon$. Let $c\in (-\varepsilon,\varepsilon)$ and $X=X_c$ be as above. Fix $r_0>0$ arbitrary and define $r(t)$ and $u(r)$ as in \eqref{billi} and \eqref{ba012}, respectively. Then, $u$ is a positive radial solution of \eqref{eq1} in $B_{r_0}(0)\setminus \{0\}$ with $u'(r)>0$ for all $r\in (0,r_0]$.
Furthermore, there exists $\gamma=\gamma(X_c)\in (0,u(r_0))$ such that $\lim_{r\to 0^+} u(r)=\gamma$ and $u$ satisfies $(P_+)$. By varying $c\in (-\varepsilon,\varepsilon)$, we obtain infinitely many positive radial solutions $u=u_c$ of \eqref{eq1} in $B_{r_0}(0)\setminus \{0\}$ with 
$\lim_{r\to 0^+} u(r)=\gamma_c\in (0,u(r_0))$, where $u(r_0)=(T_0 r_0^\sigma)^{m/\kappa}$. 
Using the transformation $T_j[u_c]$ for $j>0$, we can easily conclude the assertion of Lemma~\ref{easi} corresponding to $(P_+)$ for arbitrary $\gamma\in (0,\infty)$.  
\end{proof}

%%%%
%\subsection{Refined asymptotics}
%{\bf Refined asymptotics.}
For $r>0$ small, we define 
\begin{equation}
\label{sfir}
s=\frac{u(r)}{r|u'(r)|},\quad 
Y(s)=s\left[ (r^\Theta u)^\kappa s^{-m}-\lambda\right] -\left(\sigma-2\rho\right)
{\rm sgn}\, (u'(r)).	
\end{equation}
By Lemmas~\ref{nou1} and \ref{gara1}, we have $\lim_{r\to 0^+}s= \infty$, as well as 
\begin{equation}
\label{sfor}
\lim_{r\to 0^+} sr^\sigma=\left( \frac{\gamma^\kappa}{\lambda}\right)^{\frac{1}{m}}\quad \mbox{and}\quad 
\lim_{s\to \infty} \frac{Y(s)}{s}=0. 		
\end{equation}
Moreover, $ds/dr<0$ for every $r>0$ small, which implies that there exists $s_0>0$ large so that 
\begin{equation} \label{milo1} \frac{ds}{dr}=-\frac{s \,{\rm sgn}\, (u'(r)) \left[Y(s)+\sigma \,{\rm sgn}\, (u'(r))-1/s\right]}{r} <0
\end{equation} for every $s\geq s_0$. 
A simple calculation reveals that
\begin{equation} \label{milos1} \frac{d}{dr} (sr^\sigma)=r^{\sigma-1} \left[1-sY(s)\right] {\rm sgn}\, (u'(r)). 
\end{equation}

We introduce the following constants:
\begin{equation} \label{dzep} \begin{aligned}
 &	d_0=\left(1-m\right) (\sigma-2\rho)+\sigma;\quad d_1=\sigma \left(\sigma-2\rho\right) +\lambda\left(1-q\right),\\
 & d_2=\left(2-m\right)d_1-2\lambda mq +\left(3\sigma-2\rho\right)\sigma,\\ 
 & \mathfrak b_1=\left( \frac{\lambda^{\frac{1}{m}}}{\sigma \gamma^{\frac{\kappa}{m}}} \right)^2  \frac{d_1}{2\lambda m},\quad \mathfrak b_2=\left( \frac{\lambda^{\frac{1}{m}}}{\sigma \gamma^{\frac{\kappa}{m}}} \right)^3\frac{d_1d_2+\lambda^2 mq \left(1-q\right)}{6\lambda^2m^2}.
 \end{aligned}
\end{equation}

\begin{lemma} \label{gara2} Let $s$ and $Y(s)$ be given by \eqref{sfir}. Then, for all $s\geq s_0$, 
\begin{equation}
\label{sfpo}	
\begin{aligned}
\left(Y(s)+\sigma\, {\rm sgn}\, (u'(r))-\frac{1}{s}\right) \frac{dY}{ds}=&-\lambda m Y(s)+\frac{d_1+d_0Y(s)\,{\rm sgn}\, (u'(r)) +\left(1-m\right) Y^2(s)}{s}\\
& -\frac{q\left[\left(\sigma-2\rho\right) {\rm sgn}\, (u'(r))+Y(s)\right]}{s^2}.
\end{aligned}
\end{equation}
Moreover, $Y$ satisfies 
\begin{equation} \label{safer} \lim_{s\to \infty} s Y(s)=\frac{d_1}{\lambda m}.\end{equation}

\end{lemma}

\begin{proof}
By direct calculation, $Y(s)$ satisfies \eqref{sfpo} for all $s\geq s_0$. 
To prove \eqref{safer}, we distinguish two cases. 

\vspace{0.2cm}
{\bf Case (I).} We first assume that $\mathfrak b_1^2+\mathfrak b_2^2\not=0$. This means that one of the following holds: 

(a) $d_1\not=0$;

(b) $d_1=0$ and $q\left(1-q\right)\not=0$. 

\vspace{0.2cm}
We split the proof of \eqref{safer} into three Steps. 

\vspace{0.2cm}
{\bf Step~1.} We show that $Y'(s)\not =0$ for every $s>0$ large and $\lim_{s\to \infty} Y(s)=0$. 
Indeed, we assume by contradiction that there exists an increasing sequence $\{s_n\}_{n\geq 1}$ in $(s_0,\infty)$, $\lim_{n\to \infty} s_n=\infty$ and $Y'(s_n)=0$ for all $n\geq 1$. By \eqref{sfpo} and 
$\lim_{s\to \infty} Y(s)/s=0$, we infer that 

\begin{equation} \label{sfou1} \begin{aligned}
 	& \lim_{n\to \infty} s_n Y(s_n)=\frac{d_1}{\lambda m}\not=0 && \mbox{if } d_1\not=0,&\\
 	& \lim_{n\to \infty} s_n^2 Y(s_n)=-\frac{q \left(\sigma-2\rho\right){\rm sgn}\, (u'(r))}{\lambda m}\not=0&& \mbox{if } d_1=0\ \mbox{and } q\left(1-q\right)\not=0.&
 	\end{aligned}
\end{equation}

Hence, in either of these situations, $Y(s_n)$ has the same sign for all $n\geq 1$. 
By differentiating \eqref{sfpo} with respect to $s$, we arrive at 
$$ \begin{aligned} \left(Y(s_n)+\sigma\, {\rm sgn}\, (u'(r))-\frac{1}{s_n}\right) Y''(s_n)&=-\frac{d_1+d_0Y(s_n)\,{\rm sgn}\, (u'(r)) +\left(1-m\right) Y^2(s_n)}{s_n^2} \\
  	&\ \ \  +\frac{2}{3} \frac{q\left[\left(\sigma-2\rho\right) {\rm sgn}\, (u'(r))+Y(s_n)\right]}{s_n^3} 
   \end{aligned}
$$
which, together with $Y'(s_n)=0$, implies that
  	$$ Y''(s_n)=-\frac{2\lambda m s_n Y(s_n) +d_1+d_0 Y(s_n) \,{\rm sgn}\, (u'(r)) +\left(1-m\right) Y^2(s_n)}{3 s_n^2 \left(Y(s_n)+\sigma\, {\rm sgn}\, (u'(r))-\frac{1}{s_n}\right)}.$$
Thus, in light of \eqref{sfou1}, we get
$$ \begin{aligned} 
&\lim_{n\to \infty} s_n^2 Y''(s_n)=-\frac{d_1 \,{\rm sgn}\, (u'(r))}{\sigma }\not=0 && \mbox{if } d_1\not=0,&\\
& \lim_{n\to \infty} s_n^3 Y''(s_n)=\frac{2 q\left(\sigma-2\rho\right)}{3\sigma} \not=0&& \mbox{if } d_1=0\ \mbox{and } q\left(1-q\right)\not=0.&
 \end{aligned}
$$ 
As ${\rm sgn}\, ( Y''(s_n))$ is the same for all critical points of $Y_n$, we reach a contradiction. Then, $Y'(s)\not=0$ for every $s>0$ large. This, jointly with \eqref{sfor} and  \eqref{milo1}, implies that 
$\lim_{s\to \infty} Y(s)=0$. 

\vspace{0.2cm}
{\bf Step~2.} We prove that $(sY(s))'\not=0$ for every $s>0$ large. 
%We use a similar argument to Step~1. 

For $s\geq s_0$, we let $Z(s)=s Y(s)$.  
In view of \eqref{sfpo}, we have
\begin{equation} \label{limb1} 
\begin{aligned}
\left(\frac{Z(s)}{s}+\sigma\, {\rm sgn}\, (u'(r))-\frac{1}{s}\right) Z'(s)=& d_1-\lambda m Z(s)
- \left[ q+1+\left(m-2\right) Z(s)\right] \frac{Z(s)}{s^2} \\
&- \left[q\left(\sigma-2\rho\right)-\left(d_0+\sigma\right) Z(s)\right]\frac{ {\rm sgn}\, (u'(r))}{s}.
\end{aligned}
\end{equation}

We assume by contradiction that $Z'(s_n)=0$ for all $n\geq 1$, where $\{s_n\}_{n\geq 1}$ is an increasing sequence in $(s_0,\infty)$ with $\lim_{n\to \infty} s_n=\infty$. It follows that 
$$ \lim_{n\to \infty}  Z(s_n)=\frac{d_1}{\lambda m}.$$ 
%and
%\begin{equation} \label{foul1} \begin{aligned}
% 	& \lim_{n\to \infty}  Z(s_n)=\frac{d_1}{\lambda m}\not=0 && \mbox{if } d_1\not=0,&\\
% 	& \lim_{n\to \infty} s_n Z (s_n)=-\frac{q \left(\sigma-2\rho\right){\rm sgn}\, (u'(r))}{\lambda m}\not=0&& \mbox{if } d_1=0\ \mbox{and } q\left(1-q\right)\not=0.&
% 	\end{aligned}
%\end{equation}

Moreover, we find that 
\begin{equation} \label{love}
\begin{aligned}
s_n^2\left(\frac{Z(s_n)}{s_n}+\sigma\, {\rm sgn}\, (u'(r))-\frac{1}{s_n}\right) Z''(s_n)=& 
2 \left[ 
q+1+\left(m-2\right) Z(s_n) \right] \frac{Z(s_n)}{s_n}\\
&+
\left[ q\left( \sigma-2\rho\right)-\left(d_0+\sigma\right) Z(s_n)\right] {\rm sgn}\, (u'(r)) .
\end{aligned}
\end{equation}

$\bullet$ If $ q\left( \sigma-2\rho\right)-\left(d_0+\sigma\right) \frac{d_1}{\lambda m}\not=0$, then 
using that $Z(s_n)/s_n\to 0$ as $n\to \infty$, we get
$$\sigma s_n^2 Z''(s_n) \sim q\left( \sigma-2\rho\right)-\left(d_0+\sigma\right) \frac{d_1}{\lambda m}\not=0\quad \mbox{as } n\to \infty, $$
which leads to a contradiction. 

$\bullet$ Let $ q\left( \sigma-2\rho\right)-\left(d_0+\sigma\right) \frac{d_1}{\lambda m}=0$. 
Then, necessarily $d_1\not=0$. (If $d_1=0$, then we would have $q=0$ or $q=1$, which cannot happen in Case (I).)
From $Z'(s_n)=0$ for all $n\geq 1$, we see that
\begin{equation} \label{love2}  \left[ 
q+1+\left(m-2\right) Z(s_n) \right] \frac{Z(s_n)}{s_n^2}
%&=d_1-\lambda m Z(s_n)- \left[q\left(\sigma-2\rho\right)-\left(d_0+\sigma\right) Z(s_n)\right]\frac{ {\rm sgn}\, (u'(r))}{s_n}\\
 =\left(d_1-\lambda m Z(s_n)\right) \left(1-\frac{\left(d_0+\sigma\right)}{\lambda m}\frac{ {\rm sgn}\, (u'(r))}{s_n} \right).
 \end{equation}

(i) Assume that $q+1+(m-2)\frac{d_1}{\lambda m}\not=0$. Consequently, from \eqref{love} and \eqref{love2}, we have
$$ \sigma  \,{\rm sgn}\, (u'(r))\, s_n^3 Z''(s_n)\to 2 \left[q+1+(m-2)\frac{d_1}{\lambda m}\right] \frac{d_1}{\lambda m}\not=0
\quad \mbox{as } n\to \infty,
$$
which leads to a contradiction.

(ii) Assume that $q+1+(m-2)\frac{d_1}{\lambda m}=0$. Then, since $\kappa=q+m-1>0$, we have
$$ d_1\not=\lambda m.$$
Now, \eqref{love2} implies that
$  d_1-\lambda m Z(s_n)=0$ for all $ n\geq 1$ large. Moreover, in view of \eqref{limb1}, we conclude that
$\lambda m Z(s)=d_1$ for every $s>0$ large. Using this fact into \eqref{milos1} and the definition of $s$ in \eqref{sfir}, by direct integration, we arrive at 
$$ u(r)=\gamma \left(1+\frac{\mathfrak c_1 {\rm sgn}\,(u'(r))}{\sigma} \left(\frac{\lambda}{\gamma^\kappa}\right)^{\frac{1}{m}} r^\sigma\right)^{\frac{1}{\mathfrak c_1}},\quad \mbox{where } \mathfrak c_1=1-\frac{d_1}{\lambda m}\not=0. 
$$
Since $u$ is a positive radial solution of \eqref{eq1}, we find that $\mathfrak c_1=1$, that is, $d_1=0$. This is again a contradiction.  
The proof of Step~2 is now complete. 
%$$ s_n^2  \left(d_1-\lambda m Z(s_n)\right) \to \left[q+1+(m-2)\frac{d_1}{\lambda m}\right] \frac{d_1}{\lambda m}\not=0\quad \mbox{as } n\to \infty.$$
%which used into \eqref{love} leads to 

\vspace{0.2cm}
{\bf Step 3.} {\em Proof of \eqref{safer} concluded.} 

By Step~2, we have $\lim_{s\to \infty} sY(s)=L_1\in \mathbb R$. We show that $L_1\not=\pm \infty$. 
Suppose the contrary. Then, $\lim_{s\to \infty} s |Y(s)|= \infty$. On the other hand, we have 
$$ \lim_{s\to \infty} \frac{\log |Y(s)|}{s}=\lim_{s\to \infty} \frac{Y'(s)}{Y(s)}=-\frac{\lambda m}{\sigma} {\rm sgn}\,(u'(r))\not=0. 
$$

If $u'(r)>0$ for all $r>0$ small, then we get a contradiction with $\lim_{s\to \infty} s |Y(s)|= \infty$. 

If $u'(r)<0$ for all $r>0$ small, then we get a contradiction with $\lim_{s\to \infty} Y(s)=0$.

Hence, we have $L_1\in \mathbb R$. We now show that $L_1=\frac{d_1}{\lambda m}$. Suppose by contradiction that this is not the case. Then, using 
\eqref{sfpo} and L'Hopital's rule, we find that
$$L_1=\lim_{s\to \infty} s Y(s)=
 \lim_{s\to \infty} -s^2 Y'(s)= \infty \cdot\,{\rm sgn}\left(\lambda m L_1-d_1\right) \,{\rm sgn}\,(u'(r)).
$$
This is a contradiction, proving that $L_1=\frac{d_1}{\lambda m}$. The proof of Step~3 is complete. 

\vspace{0.2cm}
{\bf Case (II).} Let $ \mathfrak b_1^2+\mathfrak b_2^2=0$, that is, $d_1=0$ and $q\left(1-q \right)=0$. So, one of the following holds:

(i) $q=0$ and $\sigma\left(\sigma-2\rho\right)+\lambda=0$;  

(ii) $q=1$ and $\sigma=2\rho$.

In either of these situations, from \eqref{sfpo}, we see that $Y'(s_1)=0$ for $s_1>0$ large leads to $Y(s_1)=0$. Hence, either $Y'(s)\not =0$ for every $s>0$ large or $Y(s)=0$ for every $s>0$ large. 

$\bullet$ If $Y(s)=0$ for every $s>0$ large, then \eqref{safer} is trivial. Moreover, using \eqref{milos1}, we find that
$$ u(r)=  \gamma \left( 1\pm \frac{\lambda^{\frac{1}{m}}} {\sigma \gamma^{\frac{\kappa}{m}}} r^\sigma\right) \quad \mbox{for every } r>0
\ \mbox{small}.$$

$\bullet $ Assume that $Y'(s)\not =0$ for every $s>0$ large. As for Case (I), we have $\lim_{s\to \infty} Y(s)=0$. 
Then, necessarily, $u'(r)>0$ for every small $r>0$, namely, $u$ satisfies $(P_+)$ in Lemma~\ref{gara1}. 
Indeed, we have ${\rm sgn}\,(Y'(s))=-{\rm sgn}\, (Y(s))$ for $s>0$ large. On the other hand, from \eqref{sfpo}, we get
$$  Y'(s)\, {\rm sgn}\, (u'(r)) \sim -\frac{\lambda m}{\sigma} Y(s)\quad \mbox{as } s\to \infty.  
$$ 
We show that $Z'(s)\not=0$ for every $s>0$ large, where $Z(s)=sY(s)$. Here, the proof is straightforward since \eqref{limb1} becomes
\begin{equation} \label{limb2} 
\left(\frac{Z(s)}{s}+\sigma-\frac{1}{s}\right) Z'(s)=Z(s) \left[ -\lambda m 
- \frac{q+1+\left(m-2\right) Z(s)}{s^2} 
 +\frac{d_0+\sigma}{s} \right]. 
\end{equation}
Since $Z(s)\not=0$ for all $s>0$ large and $\lim_{s\to \infty} Z(s)/s=0$, we get $Z'(s)\not=0$ for $s>0$ large. 
Hence, using the same reasoning as in Step~3 of Case (I), we arrive at 
$$\lim_{s\to \infty} sY(s)=0=\frac{d_1}{\lambda m}.$$ 
This concludes the proof of Lemma~\ref{gara2}. 
\end{proof}

Let $\mathfrak b_1$ and $\mathfrak b_2$ be as in \eqref{dzep}. 
By the proof of Lemma~\ref{gara2} in Case (II), we obtain the following. 

\begin{cor} If $\mathfrak b_1^2+\mathfrak b_2^2=0$ (that is, $d_1=0$ and $q\left(1-q\right)=0$) and $u$ satisfies $(P_-)$, then 
$$ u(r)= \gamma \left( 1- \frac{\lambda^{\frac{1}{m}}}{\sigma \gamma^{\frac{\kappa}{m}}} r^\sigma\right) \quad \mbox{for every } r>0\ \mbox{small}.$$

\end{cor}

\begin{lemma} \label{gara3} 
If $\mathfrak b_1^2+\mathfrak b_2^2\not=0$, then 
$u$ satisfies 
\begin{equation} \label{firumw} 
u(r)=\gamma \left( 1\pm \frac{\lambda^{\frac{1}{m}}}{\sigma \gamma^{\frac{\kappa}{m}}} r^\sigma 
+\mathfrak b_1 r^{2\sigma}\pm \mathfrak b_2 r^{3\sigma}
(1+o(1)) \right)\quad \mbox{as } r\to 0^+.
\end{equation}
\end{lemma}

\begin{proof} In light of \eqref{sfor}, \eqref{milos1} and \eqref{safer}, we obtain that
\begin{equation} \label{veci1} 
\begin{aligned}
& s\, r^\sigma= \left( \frac{\gamma^\kappa}{\lambda}\right)^{\frac{1}{m}}+\left(1-\frac{d_1}{\lambda m}\right)\frac{{\rm sgn}\,(u'(r))}{\sigma} \, r^\sigma  (1+o(1))\quad \mbox{as } r\to 0^+,\\
& \lim_{s\to \infty} r^{-\sigma} Y(s)= \lim_{s\to \infty} \frac{r^\sigma Y(s)}{r^{2\sigma}}=
 \left( \frac{\lambda}{\gamma^\kappa}\right)^{\frac{1}{m}}\frac{d_1}{\lambda m}. 
\end{aligned}
\end{equation}
For simplicity, we define 
\begin{equation} \label{frakn} \begin{aligned}
& \mathfrak N_1= \left( \frac{\lambda}{\gamma^\kappa}\right)^{\frac{1}{m}} \left[ \frac{d_1}{\lambda m} -\frac{\sigma-2\rho}{\sigma} \left(1-\frac{d_1}{\lambda m}\right)\right],\\
& \mathfrak N_2=\mathfrak N_1 +\frac{\left(1-m\right) \lambda^{\frac{1}{m}-1}}{2m \gamma^{\frac{\kappa}{m}}} (\sigma-2\rho)^2.
\end{aligned}
\end{equation}
Using \eqref{veci1} and the definition of $Y(s)$ in \eqref{sfir}, we arrive at 
\begin{equation} \label{saff1} \lim_{r\to 0^+} \frac{ \left( \frac{\gamma^\kappa}{\lambda}\right)^{\frac{1}{m}}\left[ (r^\Theta u(r))^\kappa s^{-m} -\lambda\right] -
\left(\sigma-2\rho\right) {\rm sgn}\, (u'(r)) \,r^\sigma } {r^{2\sigma} }=\mathfrak N_1.
\end{equation}
Recall that $\sigma m=-\kappa \Theta$ and $s=u(r)/[r |u'(r)|]$ so that 
\begin{equation} \label{saff2} (r^\Theta u(r))^\kappa s^{-m}=(u^{\frac{\kappa}{m}-1} r^{1-\sigma} |u'(r)|)^m.
\end{equation}
From \eqref{saff1} and \eqref{saff2}, we infer that as $r\to 0^+$
$$ \begin{aligned} {\rm sgn}\, (u'(r)) \,\frac{d}{dr} (u^{\frac{\kappa}{m}}(r))&= 
\frac{\kappa}{m} \lambda^{\frac{1}{m}} r^{\sigma-1} \left\{ 
1+\frac{\lambda^{\frac{1}{m}-1}}{\gamma^{\frac{\kappa}{m}}} \left[
\left(\sigma-2\rho\right) {\rm sgn}\, (u'(r))\, r^\sigma +\mathfrak N_1\, r^{2\sigma} (1+o(1))
\right]
\right\}^{\frac{1}{m}}\\
&=\frac{\kappa}{m} \lambda^{\frac{1}{m}} r^{\sigma-1} \left\{ 
1+\frac{\lambda^{\frac{1}{m}-1}}{m\gamma^{\frac{\kappa}{m}}} \left[ (\sigma-2\rho) \,{\rm sgn}\, (u'(r)) \,r^\sigma +
\mathfrak N_2\, r^{2\sigma} (1+o(1))\right]
\right\}.
\end{aligned} $$
Since $\lim_{r\to 0^+} u(r)=\gamma$, we obtain that as $r\to 0^+$
$$ \begin{aligned} u(r)&=\gamma \left\{ 
1+\frac{\kappa }{m}\frac{\lambda^{\frac{1}{m}}}{\sigma \gamma^{\frac{\kappa}{m}}}  {\rm sgn}\,(u'(r))\,r^\sigma
+\frac{\kappa}{m^2} \frac{\lambda^{\frac{2}{m}-1}}{\gamma^{\frac{2\kappa}{m}}} 
\left[ \frac{\sigma-2\rho}{2\sigma} r^{2\sigma} +\frac{\mathfrak N_2}{3\sigma}\, {\rm sgn}\, (u'(r))\, r^{3\sigma} (1+o(1))\right] 
\right\} ^{\frac{m}{\kappa}}\\
&=\gamma\left[1+\frac{\lambda^{\frac{1}{m}}}{\sigma \gamma^{\frac{\kappa}{m}}}  {\rm sgn}\,(u'(r))\,r^\sigma 
+ \mathfrak b_1 r^{2\sigma}+\mathfrak b_2 {\rm sgn}\, (u'(r)) r^{3\sigma} (1+o(1))
\right],
\end{aligned}$$
where $\mathfrak b_1$ and $\mathfrak b_2$ are given by
$$\begin{aligned} 
& \mathfrak b_1=\frac{\lambda^{\frac{2}{m}-1}}{2m\sigma^2 \gamma^{\frac{2\kappa}{m}}} \left[\sigma\left(\sigma-2\rho\right)+\lambda \left(1-q\right)\right],\\
& \mathfrak b_2=\frac{\lambda^{\frac{2}{m}-1}}{m\gamma^{\frac{2\kappa}{m}}} \frac{\mathfrak N_2}{3\sigma} +\frac{\lambda^{\frac{3}{m}}}{\sigma^3 \gamma^{\frac{3\kappa}{m}}} \frac{(1-q)}{6m^2 \lambda}\left[ 3\sigma (\sigma-2\rho) +\lambda \left(2-m-2q\right)\right].  
\end{aligned}$$
Since $\sigma(\sigma-2\rho)+\lambda(1-q)=d_1$, using the definition of $\mathfrak N_1$ and $\mathfrak N_2$ in \eqref{frakn}, we get that $\mathfrak b_1$ and $\mathfrak b_2$ above have the expressions in \eqref{dzep}. 
This concludes the proof of \eqref{firumw} and of Lemma~\ref{gara3}. 
\end{proof}

\begin{lemma} \label{gara4} Let $\mathfrak b_1^2+\mathfrak b_2^2=0$ and $u$ satisfy $(P_+)$. Assume that 
$d_1=0$ and $q=1$. Then, either  
\begin{equation} \label{mibba} u(r)=\gamma+ \frac{\lambda^{\frac{1}{m}}}{\sigma} r^\sigma\quad \mbox{for every } r>0\ \mbox{small} \end{equation} 
or there exists a constant $C\in \mathbb R\setminus \{0\}$ such that 
\begin{equation} \label{gara5}  u(r)=\gamma+ \frac{\lambda^{\frac{1}{m}}}{\sigma} r^\sigma +C r^{2\sigma+\frac{\lambda m}{\sigma}} \exp\left(-\frac{ m\gamma \lambda^{1-\frac{1}{m}}}{\sigma } r^{-\sigma}\right)(1+o(1))\quad \mbox{as } r\to 0^+.
\end{equation}
\end{lemma}

\begin{proof} In our framework, the differential equation \eqref{sfpo} satisfied by $Y(s)$ becomes
\begin{equation}
\label{sfpoo}	
\left(Y(s)+\sigma-\frac{1}{s}\right) \frac{dY}{ds}=Y(s) \left[ -\lambda m +\frac{\sigma  +\left(1-m\right) Y(s)}{s} -\frac{1}{s^2}\right].
\end{equation}
The proof of Lemma~\ref{gara2} in Case (II) shows that one of the following holds:

(a) $Y(s)=0$ for $s>0$ large (in which case, we have \eqref{mibba});

(b) $Y'(s)\not=0$ for every $s>0$ large. In this situation, $\lim_{s\to \infty} sY(s)=0$. 

In the rest of the proof we assume that (b) holds, which leads to ${\rm sgn}\, (Y(s))=-{\rm sgn}\, (Y'(s))$ for $s>0$ large.
 We show that there exists a constant $C_1\in \mathbb R\setminus \{0\}$ such that 
\begin{equation} \label{maza} \widetilde Y(s)=Y(s) s^{-1+\frac{\lambda m}{\sigma^2}} e^{\frac{\lambda ms}{\sigma}}\to C_1\ \mbox{as } s\to \infty. 
\end{equation}
By direct calculation, we obtain that 
$$ \widetilde Y'(s):=\frac{\widetilde Y(s)}{\left(Y(s)+\sigma-\frac{1}{s}\right) s^2} \left\{ -\frac{\lambda m}{\sigma^2} +ms^2 Y(s)\left[ \frac{\lambda}{\sigma}+\left(\frac{\lambda}{\sigma^2}-1\right)\frac{1}{s} \right]
\right\}.
$$
In view of \eqref{sfpoo} and $\lim_{s\to \infty} Y(s)=0$, we find that
$$ \lim_{s\to \infty} \frac{\log |Y(s)|}{s}=\lim_{s\to \infty} \frac{Y'(s)}{Y(s)}=-\frac{\lambda m}{\sigma}<0. 
$$
Hence, $\lim_{s\to \infty} s^2 Y(s)=0$ and 
$$ \lim_{s\to \infty} \frac{s^2 \widetilde Y'(s)}{\widetilde Y(s)}=-\frac{\lambda m}{\sigma^3}<0. $$
Fix a constant $a>\lambda m/\sigma^3$. Then, $\log |\widetilde Y(s)|+a/s$ is increasing for $s>0$ large. Since $\log |\widetilde Y(s)|$ is decreasing for $s>0$ large, we conclude the claim in \eqref{maza} for $C_1\in \mathbb R\setminus \{0\}$. 

Using \eqref{sfor} and \eqref{milos1},  jointly with \eqref{maza}, we infer that
$$ \lim_{r\to 0^+} \left(s-\frac{\gamma \,r^{-\sigma}}{\lambda^{1/m}}\right) =\frac{1}{\sigma} \quad  \mbox{and}\quad
\frac{Y(s)}{s}=C_2 r^{\frac{\lambda m}{\sigma}} e^{-\frac{\lambda m\gamma}{\sigma \lambda^{1/m} } r^{-\sigma}} (1+o(1))\ \  \mbox{as } 
r\to 0^+,$$
where $C_2=C_1 e^{1/\sigma} (\gamma/\lambda^{1/m})^{-\lambda m/\sigma^2}\in \mathbb R\setminus \{0\}$. 

By the definition of $Y(s)$ in \eqref{sfir} and $\sigma=2\rho$, we obtain that
$$ \begin{aligned} u'(r)&=\lambda^{\frac{1}{m}} r^{\sigma-1} \left(1+\frac{Y(s)}{\lambda s}\right)^{\frac{1}{m}} \\
&=
\lambda^{\frac{1}{m}} r^{\sigma-1} \left(1 +\frac{C_2}{\lambda m} r^{\frac{\lambda m}{\sigma}} e^{-\frac{\lambda m\gamma}{\sigma \lambda^{1/m} } r^{-\sigma}} (1+o(1))\right)\quad \mbox{as } r\to 0^+.
\end{aligned} $$
We now readily conclude \eqref{gara5}, where $C\in \mathbb R\setminus \{0\}$ is given by
$C=\lambda^{2/m-2} C_2/(m^2\gamma)$. 
\end{proof}

\begin{lemma} \label{gara6} Let $\mathfrak b_1^2+\mathfrak b_2^2=0$ and $u$ satisfy $(P_+)$. Assume that 
$d_1=0$ and $q=0$. Then, either 
\begin{equation} \label{mibb2} u(r)=\gamma+ \frac{\left(\lambda \gamma\right)^{\frac{1}{m}}}{\sigma} r^\sigma\quad \mbox{for every } r>0\ \mbox{small} \end{equation} 
or
there exists a constant $C\in \mathbb R\setminus\{0\}$ such that 
\begin{equation} \label{gara7}  u(r)=\gamma+ \frac{(\lambda \gamma)^{\frac{1}{m}}}{\sigma} r^\sigma +C r^{2\rho+\sigma} \exp\left(-\frac{m}{\sigma}\left( \lambda \gamma\right)^{\frac{m-1}{m}} r^{-\sigma}\right) (1+o(1))\quad \mbox{as } r\to 0^+.
\end{equation}
\end{lemma}

\begin{proof}
For all $r>0$ small, we define 
\begin{equation} \label{miza1} \mathfrak t=r^{\theta+1} (u'(r))^{m-1}\quad \mbox{and}\quad \mathcal Y(\mathfrak t)= \mathfrak t-\frac{\lambda u(r)}{r u'(r)} -(\sigma-2\rho).
\end{equation}	
Observe that $\mathcal Y(\mathfrak t)=Y(s)$, where $s$ and $Y(s)$ are defined in \eqref{sfir}. 
By Lemma~\ref{gara2}, we have 
\begin{equation} \label{micc0} \lim_{r\to 0^+} \mathfrak t r^{\sigma} =(\lambda \gamma)^{\frac{m-1}{m}} \quad \mbox{and}\quad  \lim_{\mathfrak t\to \infty} \mathcal Y(\mathfrak t)=0.\end{equation}

In what follows, we understand that $r>0$ is small and $\mathfrak t>0$ is large enough. 
By differentiating $\mathfrak t$ and 
 $\mathcal Y(\mathfrak t)$ with respect to $r$ and $\mathfrak t$, resp., we arrive at 
 \begin{equation}
 \label{miza2}
 \begin{aligned}
 	& \frac{d\mathfrak t}{dr}=\frac{\mathfrak t}{r} \left[-\sigma+\left(m-1\right)\mathcal Y(\mathfrak t)\right],\\
 	& \frac{d \mathcal Y}{d \mathfrak t}=\frac{\mathcal Y(\mathfrak t)}{-\sigma+\left(m-1\right)\mathcal Y(\mathfrak t)} \left[ m  -\frac{2(\sigma-\rho) +\mathcal Y(\mathfrak t)}{ \mathfrak t }
 	\right]. 
 \end{aligned}	
 \end{equation}

By the proof of Lemma~\ref{gara2} in Case (II), one of the following holds:

(a) $\mathcal Y(\mathfrak t)=0$ for $\mathfrak t>0$ large (in which case, we have \eqref{mibb2});

(b) $\mathcal Y'(\mathfrak t)\not=0$ for $\mathfrak t>0$ large. In this situation, $\lim_{\mathfrak t\to \infty} \mathfrak t \mathcal Y(\mathfrak t)=0$. 

In the rest of the proof we assume that (b) holds, which implies that 
$$ {\rm sgn}\, (\mathcal Y(\mathfrak t))=-{\rm sgn}\, (\mathcal Y'(\mathfrak t))\quad 
\mbox{for } t>0\ \mbox{large}.$$
 We show that there exists a constant $C_1\in \mathbb R\setminus \{0\}$ such that 
\begin{equation} \label{maza3} \widetilde{\mathcal Y}(\mathfrak t)=\mathcal Y(\mathfrak t) \mathfrak t^{-\frac{2\left(\sigma-\rho\right)}{\sigma}} e^{\frac{m\mathfrak t}{\sigma}}\to C_1\ \mbox{as } \mathfrak t\to \infty. 
\end{equation}

{\bf Proof of \eqref{maza3}.}
It is easy to see that $\widetilde{\mathcal Y}(\mathfrak t)$ satisfies the following ODE
$$ \begin{aligned} \frac{d \widetilde{\mathcal Y}}{d\mathfrak t} &=
\frac{\widetilde{\mathcal Y}(\mathfrak t) \mathcal Y(\mathfrak t)}{\sigma\left[-\sigma+\left(m-1\right)\mathcal Y(\mathfrak t)\right]} \left[ m\left(m-1\right)-
\frac{\sigma\left(2m-1\right)-2\rho\left(m-1\right)}{\mathfrak t}\right]
\\
&=\frac{\widetilde{\mathcal Y}^2(\mathfrak t) \mathfrak t^{\frac{2\left(\sigma-\rho\right)}{\sigma}} e^{-\frac{m\mathfrak t}{\sigma}}}{\sigma\left[-\sigma+\left(m-1\right)\mathcal Y(\mathfrak t)\right]} \left[ m\left(m-1\right)-
\frac{\sigma\left(2m-1\right)-2\rho\left(m-1\right)}{\mathfrak t}\right].
\end{aligned}
$$
This, jointly with $\lim_{\mathfrak t\to \infty} \mathfrak t \,\mathcal Y(\mathfrak t)=0$, gives that 
\begin{equation} \label{micc1} \lim_{\mathfrak t\to \infty}\frac{\mathfrak t\, \widetilde{\mathcal Y}'(\mathfrak t)}{\mathcal Y(\mathfrak t)}=0\quad \mbox{and}\quad 
-\frac{\widetilde{\mathcal Y}'(\mathfrak t)}{\widetilde{\mathcal Y}^2(\mathfrak t) }\sim \frac{m\left(m-1\right)}{\sigma^2} \mathfrak t^{\frac{2\left(\sigma-\rho\right)}{\sigma}} e^{-\frac{m\mathfrak t}{\sigma}}\quad \mbox{as } \mathfrak t\to \infty.
\end{equation}
Since $q=0$, we have $\kappa=m-1>0$ and, hence, $\widetilde{\mathcal Y}'(\mathfrak t)<0$ for $\mathfrak t>0$ large. Therefore, there exists $\lim_{\mathfrak t\to \infty} \widetilde{\mathcal Y}(\mathfrak t)=C_1\in [-\infty,\infty)$. We show that $C_1\in \mathbb R\setminus \{0\}$ by distinguishing two situations:

$\bullet$ Assume that $\mathcal Y(\mathfrak t)>0$ for $\mathfrak t>0$ large. Then, for $\mathfrak t>0$ large, we define
$$ f_a(\mathfrak t)=\frac{1}{\widetilde{\mathcal Y}(\mathfrak t) } +a \mathfrak t^{\frac{2\left(\sigma-\rho\right)}{\sigma}} e^{-\frac{m\mathfrak t}{\sigma}}.
$$ 
By taking $a>(m-1)/\sigma$, we see that $f_a(\mathfrak t)$ is decreasing for large $\mathfrak t>0$. Thus, $C_1\in (0,\infty)$.  

$\bullet$ Assume that $\mathcal Y(\mathfrak t)<0$ for $\mathfrak t>0$ large. By taking $a<(m-1)/\sigma$, we get that $f_a(\mathfrak t)$ is increasing for large $\mathfrak t>0$. Hence, we have $C_1\in [-\infty,0)$. We next rule out $C_1=-\infty$. Assume by contradiction that $C_1=-\infty$. Then, by L'H\^opital's rule and \eqref{micc1}, we find that
$$  \lim_{\mathfrak t\to \infty} \frac{1/\widetilde{\mathcal Y}(\mathfrak t)}{ \mathfrak t^{\frac{2\left(\sigma-\rho\right)}{\sigma}} e^{-\frac{m\mathfrak t}{\sigma}}}=-\frac{m-1}{\sigma}<0.
$$
This is a contradiction with the first limit in \eqref{micc1}. 

In either of the above situations, we get $C_1\in \mathbb R\setminus \{0\}$.
This completes the proof of \eqref{maza3}.  

Let $C_2= C_1 (\lambda \gamma)^{\frac{(m-1)(2\sigma-2\rho)}{m\sigma}} (m-1)/m\in \mathbb R\setminus \{0\}$. We next prove that 
\begin{equation}
\label{micc2} 
\mathfrak t=(\lambda \gamma)^{\frac{m-1}{m}} r^{-\sigma}+C_2 r^{2\rho-2\sigma} e^{-\frac{m}{\sigma} (\lambda \gamma)^{\frac{m-1}{m}} r^{-\sigma}} (1+o(1))\quad \mbox{as } r\to 0^+. 
\end{equation}
Indeed, from \eqref{micc0} and \eqref{miza2}, we infer that
$$ \frac{d (\mathfrak t \,r^\sigma)}{d\mathfrak t}=\frac{\left(m-1\right) r^\sigma \mathcal Y(\mathfrak t)}{-\sigma+\left(m-1\right)\mathcal Y(\mathfrak t)}\sim -\frac{\left(m-1\right)(\lambda \gamma)^{\frac{m-1}{m}} }{\sigma} C_1  \mathfrak t^{\frac{\sigma-2\rho}{\sigma}} e^{-\frac{m\mathfrak t}{\sigma}}\quad \mbox{as }\mathfrak t\to \infty. 
$$
It follows that 
$$ \begin{aligned}
 &\mathfrak t \,r^\sigma=(\lambda \gamma)^{\frac{m-1}{m}}+\frac{m-1}{m}(\lambda \gamma)^{\frac{m-1}{m}}C_1 \mathfrak t^{\frac{\sigma-2\rho}{\sigma}} e^{-\frac{m\mathfrak t}{\sigma}}(1+o(1))\quad \mbox{as }\mathfrak t\to \infty,\\
& e^{-\frac{m}{\sigma} \mathfrak t}=e^{-\frac{m}{\sigma} (\lambda \gamma)^{\frac{m-1}{m}} r^{-\sigma}} (1+o(1))\quad \mbox{as } r\to 0^+,
 \end{aligned}
$$ which lead to \eqref{micc2}. 

Using that $(\sigma+\theta+1)/(1-m)=\sigma-1$, from \eqref{micc2} and the definition of $\mathfrak t$ in \eqref{miza1}, we get 
$$ u'(r)= (\lambda \gamma)^{\frac{1}{m}} r^{\sigma-1} 
+\frac{C_2 (\lambda \gamma)^{\frac{2-m}{m}}}{m-1} r^{2\rho-1}e^{-\frac{m}{\sigma} (\lambda \gamma)^{\frac{m-1}{m}} r^{-\sigma}} (1+o(1))
 \quad \mbox{as } r\to 0^+.
$$
Hence,  we conclude the proof of \eqref{gara7}, where
$C=C_2(\lambda \gamma)^{(3-2m)/m} /[m(m-1)]\in \mathbb R\setminus \{0\}$.
\end{proof}

%%%%
\section{The asymptotic profile $E_0$ when $\Theta=0<\lambda$ (Proof of Theorem~\ref{tieo10})}

Throughout this section, we assume that \eqref{cond1} hold, $\Theta=0$, $\lambda>0$ and $\rho\in \mathbb{R}$. 
The proof of Theorem~\ref{tieo10} follows by combining Lemmas~\ref{step1}--\ref{step4}. Unless otherwise stated, $u$ is a positive radial solution of \eqref{eq1} in $B_R(0)\setminus \{0\}$ for $R>0$ such that 
 \begin{equation} \label{tuna1}
 u(r) \sim \left[ \lambda \left( \frac{\kappa}{m}\right)^m\right]^{\frac{1}{\kappa}} [\log \,(1/r)]^{\frac{m}{\kappa}}
 :={E_0(r)} \quad \mbox{as } r\to 0^+.
 \end{equation} 
Observe that $u'(r)<0$ for every $r\in (0,R)$ 
since $\lim_{r\to 0^+} u(r)=\infty$. Indeed, if $u'(r_1)=0$ for some $r_1\in (0,R)$, then $ r_1^2 u''(r_1)=-\lambda u(r_1) <0$, which means that $r=r_1$ is a 
local maximum point for $u$. This is impossible. It follows that $\lim_{r\nearrow R} u(r):=u(R)\in [0,\infty)$. 

\begin{lemma} \label{step1} If $\rho=0$ and $q=1$, then 
\begin{equation} \label{luxi3} u(r)=\lambda^{\frac{1}{m}}\,\log \,(R/r)+u(R)\quad \mbox{for all }r\in (0,R). \end{equation}  
\end{lemma}

\begin{proof} 
Indeed, we see that $u$ satisfies the following equation
$$  r\, \frac{d}{dr} \left( ru'(r)\right)=u(r) \left[ (r|u'(r)|)^m -\lambda\right]\quad \mbox{for every } r\in (0,R). 
$$
We define $w(s)=u(r)$ with $s=\log \,(R/r)$. Then, $w'(s)>0$ for all $s>0$. Moreover, 
\begin{equation} \label{weq}
w''(s)=w(s) \left[ (w'(s))^m-\lambda\right]\quad \mbox{for all } s>0.
\end{equation}
The first and second order derivatives of $w$ are to be understood with respect to $s$. 
Since $u$ satisfies \eqref{tuna1} and $\kappa=m$, it follows that 
\begin{equation} \label{oil} \lim_{s\to \infty} \frac{w(s)}{s}=\lambda^{\frac{1}{\kappa}}.  
\end{equation}
We now observe that $w'(s)=\lambda^{1/\kappa}$ for all $s>0$. Otherwise, there exists $s_1>0$ such that $w'(s_1)\not=\lambda^{1/\kappa}$. Suppose, for example, that 
$w'(s_1)>\lambda^{1/\kappa}$. Then, from \eqref{weq}, we find that $w''(s)>0$ for every $s\in [s_1,\infty)$. 
Hence, there exists $\lim_{s\to \infty} w'(s)>\lambda^{1/\kappa}$. This is a contradiction with \eqref{oil}. Similarly, we obtain a contradiction when $w'(s_1)<\lambda^{1/\kappa}$, which would lead to 
$w''(s)<0$ for every $s\in [s_1,\infty)$ and $\lim_{s\to \infty} w'(s)<\lambda^{1/\kappa}$. 

The above reasoning shows that $w'(s)=\lambda^{1/\kappa}$ for all $s>0$, which proves the claim in \eqref{luxi3}. 
\end{proof}

For every $r\in (0,R)$, we define $\Psi_0(r)$ and $\Psi_1(r)$ as follows
\begin{equation} \label{dipo} \Psi_0(r)=\frac{u(r)}{ru'(r)}\quad \mbox{and}\quad \Psi_1(r)=r\Psi_0'(r).  
\end{equation}

\begin{lemma} \label{step2}  We assume that either $\rho\not=0$ or $q\not=1$. Then, there exists $r_0\in (0,R)$ small such that
$\Psi_0'(r)>0$ and $\Psi_1'(r)\not=0$ for all $r\in (0,r_0)$. Moreover, we have 
$$ \lim_{r\to 0^+}\Psi_0(r)= -\infty \quad \mbox{and} \quad \lim_{r\to 0^+} \Psi_1(r)= \lim_{r\to 0^+} \frac{\Psi_0(r)}{\log r}=
\frac{\kappa}{m} .$$ 
\end{lemma}

\begin{proof}
Assume by contradiction that for a sequence $\{r_n\}_{n\geq 1}$ in $(0,r_0)$ decreasing to $0$ as $n\to \infty$, we have $\Psi_0'(r_n)=0$ for each $n\geq 1$. 
For each $r\in (0,R)$, we get that 
\begin{equation} \label{opa} \Psi_1(r)=r\Psi_0'(r)=1-2\rho \Psi_0(r)+\Psi_0^2(r)\left[ \lambda- u^\kappa(r)\, |\Psi_0(r)|^{-m} \right] .
\end{equation}
By differentiating \eqref{opa} and letting $r=r_n$, we arrive at 
$$ r_n^2 \Psi_0''(r_n)= \kappa  \,u^\kappa(r_n)\, |\Psi_0(r_n)|^{1-m}>0\quad \mbox{for all } n\geq 1.
$$
This leads to a contradiction, proving that $\Psi_0'(r)$ has a constant sign for each $r>0$ small enough. Therefore, there exists $\lim_{r\to 0^+} \Psi_0(r)=-\infty$ in view of \eqref{tuna1} and 
L'H\^opital's rule. Then, there exists $r_0\in (0,R)$ such that $\Psi_0'(r)>0$ for all $r\in (0,r_0)$. 

We next show that $\Psi_1'(r)\not=0$ for every $r>0$ small. 
For every $r\in (0,r_0)$, we define
$$ \mathfrak T_1(r):=2\rho \left[ \left(1-m\right) \Psi_1(r)+\kappa\right],\quad \mathfrak T_2(r):=
\kappa- \left(\kappa+m-2\right) \Psi_1(r)+\left(m-2\right) \Psi_1^2(r).
$$
In view of \eqref{opa}, for every $r\in (0,r_0)$, we find that 
\begin{equation} \label{luci} 
r \Psi_0(r)\Psi_1'(r)=\Psi_0(r) \,\mathfrak T_1(r)-\mathfrak T_2(r)
 +\lambda \Psi_0^2(r) \left[m\Psi_1(r)-\kappa\right] .
\end{equation}
Suppose by contradiction that there exists a sequence 
$\{r_n\}_{n\geq 1}$ in $(0,r_0)$ decreasing to $0$ as $n\to \infty$ such that $\Psi_1'(r_n)=0$ for all $n\geq 1$. By differentiating \eqref{luci} and using that $\Psi_1'(r_n)=0$ for all $n\geq 1$, we obtain that 
\begin{equation} \label{luci2} 
\frac{r_n^2 \Psi_0(r_n) \Psi_1''(r_n)}{\Psi_1(r_n)}=  
 \mathfrak T_1(r_n)+2\lambda \Psi_0(r_n) \left[m\Psi_1(r_n) -\kappa\right].
\end{equation}
Note that $\{\Psi_1(r_n)\}_{n\geq 1}$ is a bounded sequence. Indeed, assume by contradiction that for a subsequence $\{r_{n_j}\}_{j\geq 1}$ of $\{r_n\}$, we have 
$\Psi_1(r_{n_j})\to \infty$ as $j\to \infty$. Then, since $\lim_{r\to 0^+} \Psi_0(r)=-\infty$, $\lambda>0$ and $\Psi_1'(r_{n_j})=0$ for all $j\geq 1$, we obtain from \eqref{luci} that $m>2$ and 
$$ \lim_{j\to \infty} \frac{\Psi_1(r_{n_j})}{\Psi_0^2(r_{n_j})}=\frac{\lambda m}{m-2}. 
$$
Using this fact into \eqref{luci2}, we arrive at 
$$ \lim_{j\to \infty} \frac{r^2_{n_j}\, \Psi_1''(r_{n_j})}{\Psi_1^2(r_{n_j})}=2m\lambda>0.
$$ Hence, there exists $j_*>0$ large such that $\Psi_1''(r_{n_j})>0$ for every $j\geq j_*$. Then, $r_{n_j}$ is a local minimum for $\Psi_1$ for every $j\geq j_*$, which is a contradiction. Hence, $\{\Psi_1(r_n)\}_{n\geq 1}$ is bounded. 

\vspace{0.2cm}
From \eqref{luci}, $\lim_{r\to 0^+} \Psi_0(r)=-\infty$ and $\Psi_1'(r_n)=0$ for all $n\geq 1$, we see that
$$ 
\lambda \Psi_0(r_n) \left[m\Psi_1(r_n)-\kappa\right]=-\mathfrak T_1(r_n)
+\frac{ \mathfrak T_2(r_n)}{\Psi_0(r_n)} \quad \mbox{and} \quad \lim_{n\to \infty} \Psi_1(r_n)= \frac{\kappa}{m}.
$$
Using these facts into \eqref{luci2}, we infer that 
$$  \frac{r_n^2 \Psi_0(r_n) \Psi_1''(r_n)}{\Psi_1(r_n)}=
-\mathfrak T_1(r_n) +\frac{ 2 \mathfrak T_2(r_n)}{\Psi_0(r_n)} 
\quad \mbox{for all } n\geq 1,
$$
which, jointly with $\lim_{n\to \infty} \mathfrak T_1(r_n)=2\rho \kappa/m$, leads to 
\begin{eqnarray}
&\displaystyle  \lim_{n\to \infty} r_n^2 \Psi_1''(r_n) \Psi_0(r_n)=-\frac{2\rho \kappa^2}{m^2}\quad & \mbox{if } \rho\not=0, \label{luci3}\\
& \displaystyle \lim_{n\to \infty} r_n^2 \Psi_1''(r_n) \Psi_0^2(r_n)=\frac{4\kappa^2 \left(1-q\right)}{m^3} \quad & \mbox{if } \rho=0\ \mbox{and } q\not=1.  \label{luci4}
\end{eqnarray}

So, when $\rho\not=0$, then for every $n\geq 1$ large, the sign of $\Psi_1''(r_n)$ is the same as that of $\rho$. On the other hand, when $\rho=0$ and $q\not=1$, then 
for every $n\geq 1$ large, we obtain that the sign of $\Psi_1''(r_n)$ is the same as that of $(1-q)$. In either of these cases, we reach a contradiction.  

Thus, $\Psi_1'(r)\not=0$ for every $r>0$ small. Using \eqref{tuna1} and L'H\^opital's rule, we conclude that 
\begin{equation} \label{luci5} \lim_{r\to 0^+} \Psi_1(r)= \lim_{r\to 0^+} \frac{\Psi_0(r)}{\log r}=\lim_{r\to 0} \frac{u(r)}{u'(r) \,r\log r}=\lim_{r\to 0^+} \frac{\log \log 
\,(1/r)}{\log u(r)}=\frac{\kappa}{m}.
\end{equation}
The proof of Lemma~\ref{step2} is now complete. 
\end{proof}

\begin{lemma}\label{stepint}
Assume that either $\rho\neq 0$ or $q\neq1$. Using \eqref{dipo}, 
for each $r\in (0,r_0)$, we define
 \begin{equation}\label{xor}
 t=\Psi_0^2(r)\ \ \mbox{and }\    \frac{\kappa}{m}-X(t):=\Psi_1(r)= -\frac{r}{2\sqrt{t}}\,\frac{dt}{dr}.
  \end{equation} 
  Then, $t\rightarrow \infty$ and $X(t)\rightarrow 0$ as $r\rightarrow 0^+$. Moreover, for $t>0$ large, $|X(t)|<\kappa/(2m)$ and
  \begin{equation} \label{xeq}
\frac{dX}{dt}=
\frac{\lambda\,m\,X(t)}{2\left(\frac{\kappa}{m}-X(t)\right)}+\left( 1+\frac{m\,X(t)}{\frac{\kappa}{m}-X(t)}\right)\frac{\rho}{\sqrt{t}}+
\left( 1+\frac{m\,X(t)}{2\left(\frac{\kappa}{m}-X(t)\right)}\right)\frac{\frac{1-q}{m}+X(t)}{t}.
\end{equation} 
In addition, we have
\begin{equation}\label{som}
 \begin{aligned}
& \lim_{t\to \infty} \sqrt{t} \,X(t)=-\frac{2\rho\kappa}{m^2\,\lambda}&& \mbox{if } \rho\not=0,&\\
& \lim_{t\to \infty} t\,X(t)= \frac{2\left(q-1\right)\kappa}{m^3\lambda}&& \mbox{if } \rho=0\ \mbox{and } q\not=1.
\end{aligned} 
\end{equation}
\end{lemma}

\begin{proof}
From Lemma~\ref{step2}, we obtain that 
\begin{equation} \label{luci6} \frac{\sqrt{t}}{\log \, (1/r)}\to \frac{\kappa}{m} \quad  \mbox{and} \quad X(t)\to 0\ \ \mbox{as }r\to 0^+. 
\end{equation}
This implies that $t\rightarrow \infty$ as $r\rightarrow 0^+$.
For $t>0$ large, we find that  
\begin{eqnarray}
&\displaystyle \frac{dr}{dt}=-\frac{r}{2\sqrt{t}\left( \frac{\kappa}{m}-X(t)\right)}<0, \label{mi1}\\
& \displaystyle \frac{dX}{dt}=\frac{r\Psi_1'(r)}{2\sqrt{t} \left(\frac{\kappa}{m}-X(t)\right)}. \label{mi2} 
\end{eqnarray} 
A direct computation yields that $X$ satisfies \eqref{xeq}. Multiplying \eqref{xeq} by $e^{-\frac{\lambda m^2}{2\kappa}t}$, we get 
\begin{equation}\label{ecee}
\begin{aligned}
\frac{d}{dt}\left ( X(t)\,e^{-\frac{\lambda m^2}{2\kappa}t}\right )=& \left [\frac{\lambda\,m\,X^2(t)}{2\frac{\kappa}{m}\left(\frac{\kappa}{m}-X(t)\right)}+\left( 1+\frac{m\,X(t)}{\frac{\kappa}{m}-X(t)}\right)\frac{\rho}{\sqrt{t}}\right ]e^{-\frac{\lambda m^2}{2\kappa}t}\\ 
&+
\left( 1+\frac{m\,X(t)}{2\left(\frac{\kappa}{m}-X(t)\right)}\right)\frac{\frac{1-q}{m}+X(t)}{t}e^{-\frac{\lambda m^2}{2\kappa}t}.
\end{aligned}
\end{equation}

\vspace{0.2cm}
$\bullet$ First, we assume that $\rho\not=0$. 
Lemma~\ref{step2} and $\lim_{t\to \infty} X(t)=0$ imply that for $t>0$ large, $X(t) X'(t)< 0$ and ${\rm sgn}\, (\rho)={\rm sgn} \,(X'(t))=-{\rm sgn}\,(X(t))$. 
Hence, using \eqref{xeq}, we can find a constant $C>0$ such that 
$\sqrt{t} \,|X(t)|< C$ for every $t>0$ large. Thus, we get
\begin{equation}\label{xxt}
\lim_{t\rightarrow \infty}\sqrt{t}\,X^2(t)=0.
\end{equation}
Using L'H\^opital's rule, together with \eqref{ecee} and \eqref{xxt},  
we obtain that
$$ \lim_{t\rightarrow \infty} \sqrt{t}\,X(t)= 
\lim_{t\rightarrow \infty}\frac{\frac{d}{dt}\left ( X(t)\,e^{-\frac{\lambda m^2}{2\kappa}t}\right )}{\frac{d}{dt}\left (\frac{e^{-\frac{\lambda m^2}{2\kappa}t}}{\sqrt{t}}\right )}=-\frac{2\kappa\rho}{m^2\, \lambda}. $$

$\bullet$ Second, we assume that $\rho=0$ and $q\not=1$. 
Similarly, by Lemma~\ref{step2}, we have $X(t) X'(t)< 0$ for $t>0$ large. Moreover, there exists a constant $C>0$ such that  
$t \,|X(t)|< C$ for every $t>0$ large.
Hence,  
$\lim_{t\rightarrow \infty} t\,X^2(t)=0$, which together with \eqref{ecee} and L'H\^opital's rule, yields that 
$$
\lim_{t\rightarrow \infty} t\,X(t)= 
\lim_{t\rightarrow \infty}\frac{\frac{d}{dt}\left ( X(t)\,e^{-\frac{\lambda m^2}{2\kappa}t}\right )}{\frac{d}{dt}\left (\frac{e^{-\frac{\lambda m^2}{2\kappa}t}}{t}\right )}=\frac{2\kappa \left(q-1\right)}{m^3\, \lambda}.
$$
This completes the proof of \eqref{som} and of Lemma~\ref{stepint}. 
\end{proof}

\begin{lemma} \label{step3}  We assume that either $\rho\not=0$ or $q\not=1$. 

$\bullet$ If $\rho\not=0$, then 
\begin{equation} \label{simp3}
	\frac{u(r)}{E_0(r)}=1+\frac{2\rho\,m}{\lambda \kappa^2}\,\frac{\log \log \, (1/r)}{\log \, (1/r)} (1+o(1)) \quad \mbox{as } r\to 0^+. 
\end{equation}

$\bullet$ If $\rho=0$ and $q\not=1$, then there exists a constant $C\in \mathbb R$ such that as $r\to 0^+$
 \begin{equation} \label{luxi1} 
\frac{u(r)}{E_0(r)}=
1+\frac{C}{\log \,(1/r)} + \left( 
 \frac{m}{\lambda \kappa^3}-\frac{C^2}{2m} 
 \right) \frac{q-1}{\log^2\,(1/r) }(1+o(1)).
\end{equation}
\end{lemma}

\begin{proof}
For every $r\in (0,r_0)$, we define $\Psi_2(r)$ by 
$$ \Psi_2(r):= \frac{m}{\kappa}\sqrt{t}-\log\,(1/r).
$$
Using \eqref{mi1}, we obtain that 
\begin{equation} \label{daf1}
\Psi_2'(r)=\frac{m}{\kappa} \frac{X(t)}{r}. 	
\end{equation}

$\bullet$ First, we assume that $\rho\not=0$. 
Then, using \eqref{som} and \eqref{luci6}, we find that
$$ \frac{\frac{d}{dr}\left (\Psi_2(r)\right)}{\frac{d}{dr}[\log \log \, (1/r)]}
=-\frac{m}{\kappa} X(t) \log (1/r) \to\frac{2\rho}{\kappa\lambda} \quad \mbox{as } r\to 0^+.$$
This implies that $\lim_{r\to 0^+} \Psi_2(r)={\rm sgn}\, (\rho) \,\infty$ and by L'H\^opital's rule,
\begin{equation} \label{zeam} \lim_{r\rightarrow 0^+}\frac{\Psi_2(r)}{\log \log \, (1/r)}=\frac{2\rho}{\kappa\lambda}.
\end{equation}
Therefore, from \eqref{luci6} and \eqref{zeam}, we get
$$ \lim_{r\rightarrow 0^+} \frac{\frac{d}{dr} \left (\frac{u(r)}{E_0(r)}-1\right )}{ \frac{d}{dr}
\left[ \frac{\log \log \, (1/r)}{\log \, (1/r)}\right]}=\lim_{r\rightarrow 0^+} \frac{ \log\,(1/r)}{\sqrt{t}}\,
\frac{ \Psi_2(r)}{\log \log \, (1/r)}=\frac{2\rho m}{\lambda \kappa^2}, 
$$
which, by L'H\^opital's rule, proves the claim in \eqref{simp3}. 

$\bullet$ Second, we assume that $\rho=0$ and $q\not=1$. 
Then, using \eqref{som} and \eqref{luci6}, we obtain that
$$\lim_{r\rightarrow 0^+} \frac{\frac{d}{dr}\left (\Psi_2(r)\right )}{\frac{d}{dr}[\frac{1}{\log\,(1/r)}]}=\frac{m^3}{\kappa^3}\lim_{t\to \infty} t\,X(t)=\frac{2\left(q-1\right)}{\lambda \kappa^2}. $$
Thus, ${\rm sgn}\, (\Psi_2'(r))={\rm sgn} \,(q-1)$ and                            
there exists $ \lim_{r\to 0^+}\Psi_2(r)=L_1\in \mathbb R $. Using \eqref{luci6}, we get
$$\lim_{r\rightarrow 0^+}\frac{\frac{d}{dr}\left (\frac{u(r)}{E_0(r)}-1\right )}{\frac{d}{dr}
[\frac{1}{\log\,(1/r)}]}=\lim_{r\rightarrow 0^+} \frac{\log\,(1/r)}{\sqrt{t}}\,\Psi_2(r)=L_1 \frac{m}{\kappa}. 
$$
Using \eqref{tuna1}, \eqref{som}, \eqref{luci6},  and \eqref{daf1}, jointly with L'H\^opital's rule, we see that
$$ \lim_{r\to 0^+} \left( \Psi_2(r)-\frac{m L_1}{\kappa} \frac{E_0(r)}{u(r)} \frac{\sqrt{t}}{\log \,(1/r)} \right) \log\, (1/r) =
\lim_{r\to 0^+}  \frac{\frac{d}{dr} \left( \Psi_2(r)-\frac{m L_1}{\kappa} \frac{E_0(r)}{u(r)} \frac{\sqrt{t}}{\log \,(1/r)} \right)}{\frac{d}{dr}
\left[\frac{1}{\log\, (1/r)}\right]}=2L_2, $$
where $L_2=\frac{(q-1)}{\lambda\kappa^2} \left(1-\frac{\lambda \kappa}{2} L_1^2\right)$. 
Hence, it follows that
$$  \lim_{r\rightarrow 0^+} \left (\frac{u(r)}{E_0(r)}-1 -\frac{mL_1}{\kappa} \frac{1}{\log \,(1/r)}\right ) \log^2 (1/r)=
 \lim_{r\rightarrow 0^+}\frac{\frac{d}{dr}\left (\frac{u(r)}{E_0(r)}-1 -\frac{mL_1}{\kappa} \frac{1}{\log \,(1/r)}\right )}{\frac{d}{dr}
[\frac{1}{\log^2\,(1/r)}]}=\frac{m L_2}{\kappa}.
$$
This completes the proof of \eqref{luxi1}, where $C=mL_1/\kappa$. 
\end{proof}

\begin{lemma} \label{step4}  
Let $R>0$ be arbitrary. Then, equation \eqref{eq1} in $B_R(0)\setminus \{0\}$ has infinitely many positive radial solutions $u$ satisfying \eqref{tuna1}.
\end{lemma}

\begin{proof} 
Let $R>0$ be arbitrary. In light of Lemma~\ref{step1}, we need to consider only $\rho\not=0$ or $q\not=1$ and prove the existence of infinitely many positive radial solutions of \eqref{eq1} 
in $B_R(0)\setminus 
\{0\}$ satisfying \eqref{tuna1}. 
Note that for $T_0>0$ large, the differential equation in \eqref{xeq} on $[T_0,\infty)$ admits a solution $X(t)$ satisfying $\lim_{t\to \infty} X(t)=0$. 
The differential equation in \eqref{mi1} for $t\geq T_0$, subject to $r(T_0)=R$, has a unique solution 
\begin{equation} \label{rubc} r(t)= R\,\exp\left( -\frac{1}{2} \int_{T_0}^t \frac{ds}{\sqrt{s} \left(\frac{\kappa}{m}-X(s)\right)}\right)\quad \mbox{for all } t\geq T_0. 
\end{equation}
From  \eqref{mi1} and $X(t)\to 0$ as $t\to \infty$, we have 
\begin{equation} \label{vim} \frac{\sqrt{t}}{\log\, (1/r)}\to \frac{\kappa}{m}\quad \mbox{as } r\to 0^+.\end{equation}
For every $r\in (0,R]$, we can express $t$ as a function of $r$ using \eqref{rubc} and define $u(r)$ as follows
$$ u(r)= t^{\frac{m}{2\kappa}} \left[ \lambda+\frac{2\rho}{\sqrt{t}} +\left( X(t)-\frac{q-1}{m}\right)\frac{1}{t}\right]^{\frac{1}{\kappa}}.
$$
Since $X(t)$ solves \eqref{xeq} for all $t\geq T_0$ and $\lim_{t\to \infty} X(t)=0$, we obtain  
that $u$ is a positive radial solution of \eqref{eq1} in $B_{R}(0)\setminus \{0\}$ satisfying \eqref{tuna1} by virtue of \eqref{vim}. 

For each $j\in (0,1)$, by applying the transformation $T_j[u](r)=u(jr)$ for $r\in (0,R/j)$, we get that $T_j [u]$ is another positive radial solution of \eqref{eq1} in $B_{R/j} (0)\setminus \{0\}$ satisfying \eqref{tuna1}. 
As $R>0$ is arbitrary, the proof of Lemma~\ref{step4} is complete.
\end{proof}

%%%%%

\section{The constant asymptotic profile for $\lambda=0$ and $\beta\not=0$ (Proof of Theorem~\ref{profff1})} \label{nupo}

In this section, besides \eqref{cond1}, we assume throughout that 
\begin{equation} \label{bita} \lambda=0\quad\mbox{and}\quad  \beta:=\kappa \Theta+2\rho \left(m-1\right)\not=0.
\end{equation}
Moreover, $u$ is a positive radial solution of  
\eqref{eq1} in $B_{R}(0)\setminus \{0\}$ for $R>0$, namely, 
\begin{equation} \label{faz}
u''(r)+(1-2\rho)\,\frac{u'(r)}{r}=r^\theta \, u^q(r)|u'(r)|^m\quad \mbox{for } r\in (0,R)
\end{equation} 
with the property that 
	\begin{equation}\label{z2}
\lim_{r\to 0^+} u(r)=\gamma\in (0,\infty)\quad \mbox{and}\quad r^*(u)=0.
\end{equation}
Multiplying \eqref{faz} by $r^{1-2\rho}$, for every $r\in (0,R)$, we obtain that 
\begin{equation} \label{faz1}
(r^{1-2\rho}\,u'(r))'=(r^{1-2\rho} |u'(r)|)^m r^{\beta-1} u^q(r)\geq 0.	
\end{equation}
Since $r\longmapsto r^{1-2\rho}\,u'(r)$ is non-decreasing on $(0,R)$, there exists 
\begin{equation}\label{coo} \lim_{r\to 0} r^{1-2\rho} u'(r):=L_0 
\in [-\infty,\infty).\end{equation} 
Then, for $r_0>0$ small, we have $u'(r)\not=0$ for every $r\in (0,r_0)$. This is clear if $L_0\not=0$, while if $L_0=0$, we use that $r^*=0$. 
The conclusions depend on ${\rm sgn}\, (\beta)$ and whether $m=1$ or $m\not=1$. 

If $\rho\leq 0$, then necessarily $L_0=0$ and $u'(r)>0$ for all $r\in (0,R)$. Indeed, if 
$L_0\not=0$, then $\lim_{r\to 0^+} 
ru'(r)\not=0$, which would contradict \eqref{z2}. Thus, $L_0\not=0$ implies that $\rho>0$. 

%%%%
\begin{lemma}[Case (I), Theorem~\ref{profff1}] \label{nelem}
Let $m=1$ and $\beta<0$, where $\beta$ is defined in \eqref{bita}.   
Then, $\kappa\Theta=\theta+1<0$ and $u'>0$ on $(0,R)$. Moreover, there exists a constant $C>0$ such that \eqref{refina} is satisfied, that is,
 \begin{equation}
 \label{roma1}
 u(r)=\gamma+ C r^{2\rho-1-\theta} \exp\left( \frac{\gamma^q}{\theta+1} r^{\theta+1}\right)  (1+o(1))\quad \mbox{as } r\to 0^+. 	
 \end{equation}
\end{lemma}

\begin{proof}
Our assumptions give that $m=1$ and $\beta=\kappa\Theta=\theta+1<0$. For every $r\in (0,R)$, we set
\begin{equation}\label{pox1}
t:=r^{\theta+1}u^q(r)\quad \text{and}\quad X(t):=t\frac{ru'(r)}{u(r)} . 
\end{equation}
Clearly, $t\to \infty$ as $r\to 0^+$. We show that $L_0=0$ (see \eqref{coo} for the definition of $L_0$) and $u'(r)>0$ for every $r\in (0,R)$. 
To this end, we fix $0<C_*<-\gamma^q/\beta$. Then, $H(r)$ is non-decreasing for every $r>0$ small, where we define 
\begin{equation}
H(r):={\rm sgn}\, (u'(r))\log \,(r^{1-2\rho}|u'(r)|)+C_* r^{\theta+1}. 
\end{equation}
Thus, necessarily, $\lim_{r\rightarrow 0^+}{\rm sgn}\, (u'(r))\log\, (r^{1-2\rho}|u'(r)|)=-\infty$. From \eqref{faz1} and L'H\^opital's rule,  
$$ \lim_{r\rightarrow 0^+}\frac{\log \,(r|u'(r)|)}{r^{\theta+1}}=
\lim_{r\rightarrow 0^+}\frac{\log \,(r^{1-2\rho}|u'(r)|)}{r^{\theta+1}}=\frac{\gamma^q \, {\rm sgn}\, (u'(r)) }{\theta+1}. $$
Hence, $\lim_{r\to 0^+} r u'(r)=0$, $L_0=0$ and $u'(r)>0$ for every $r\in (0,R)$. (Otherwise, we get $\lim_{r\to 0} (r |u'(r)|)=\infty$, which 
 is a contradiction with \eqref{z2}.) 

\vspace{0.2cm} 
Consequently, using also \eqref{z2} and \eqref{pox1}, we have  
\begin{equation}\label{pox2}
\frac{X(t)}{t}=\frac{ru'(r)}{u(r)}\rightarrow 0 \quad\text{as }r\rightarrow 0^+\quad\text{and}\quad X(t)>0\quad\text{for every }r\in (0,R).
\end{equation}
Let $T_0>0$ be large such that  $\theta+1+q\frac{X(t)}{t}<0$ for every 
$t\geq T_0$.   
By \eqref{pox1}, we find that 
\begin{equation}\label{pox5}
\frac{dr}{dt}=\frac{r}{t\left (\theta+1+q\frac{X(t)}{t}\right )}<0\quad \mbox{for every } t\geq T_0. 
\end{equation}
Moreover, by direct calculations, we obtain that 
\begin{equation}\label{nue}
\frac{dX}{dt}=\frac{X(t)+\left (\theta+1+2\rho\right )\frac{X(t)}{t} +(q-1)\left (\frac{X(t)}{t}\right )^2}{\theta+1+q\frac{X(t)}{t}}\quad  \mbox{for every } t\geq T_0.
\end{equation}
In light of \eqref{pox2} and $\theta+1<0$,  we observe that $ X'(t)/X(t) \to 1/(\theta+1)<0$ as $t\to \infty$. 
Hence $X'(t)< 0$ for every large $t>T_0$ and $\lim_{t\rightarrow \infty} X(t)=0$. For $t\geq T_0$, we now define 
\begin{equation} \label{gun1}   \widetilde{X}(t):=t^{-1-\frac{2\rho}{\theta+1}}\, X(t) \exp\left(-\frac{t}{\theta+1}\right) >0.
\end{equation}

We show that the behaviour in \eqref{roma1} holds. 
Using \eqref{nue}, we see that
\begin{equation}\label{pox3}
t\,\frac{d}{dt}\left(\widetilde{X}(t)\right)=-\frac{q+\left (2\rho q+\theta+1\right )\frac{1}{t} } 
{ \left(\theta+1\right) \left(\theta+1+q\frac{X(t)}{t}\right) } \widetilde X(t)\, X(t) \quad \mbox{for every } t\geq T_0 . 
\end{equation}
Using \eqref{pox2} and $q=\kappa>0$, we infer that $\widetilde{X}'(t)<0$ for every $t>0$ large, which implies that $\lim_{t\rightarrow \infty} \widetilde{X}(t)=M\in [0,\infty).$  We show that $M\neq 0$. We fix $D>q/|\theta+1|$ and define 
$$G(t):=\frac{1}{\widetilde{X}(t)}+D t^{\frac{2\rho}{\theta+1}} \exp\left(\frac{t}{\theta+1}\right) \quad \mbox{for all } t\geq T_0.$$ 
From \eqref{gun1} and \eqref{pox3}, we get that $G'(t)<0$ for every $t>T_0$ large. Hence, we have $M\in (0,\infty)$. This, jointly with 
 \eqref{pox1}, yields that
$$r^{1-2\rho}u'(r) \left(u(r)\right)^{-\frac{2\rho q}{\theta+1}-1} \exp\left( -\frac{r^{\theta+1}u^q(r)}{\theta+1}\right) =\widetilde{X}(t)\rightarrow M\quad\text{as }r\rightarrow 0^+.$$
Using \eqref{z2}, we conclude that there exists a constant $C>0$ such that
$$\lim_{r\rightarrow 0^+}\frac{u(r)-\gamma}{r^{2\rho-1-\theta} \exp\left({ \frac{r^{\theta+1}}{\theta+1}u^q(r)}\right)}=C,$$
which implies that \eqref{roma1} holds. This completes the proof of Lemma~\ref{nelem}. 
\end{proof}

%%%%

\begin{lemma}[Existence, Case (I) in Theorem~\ref{profff1}]
Let $m=1$ and $\beta<0$. Then, for every $\gamma\in (0,\infty)$, there exists $R>0$ such that equation \eqref{faz} has infinitely many positive solutions satisfying \eqref{z2}.
\end{lemma}

\begin{proof}
For $T_0>0$ large, we consider the ODE in \eqref{nue}. Then, there exists $\varepsilon>0$ small such that for every $T_0>1/\varepsilon$ and all $c\in (0, \varepsilon)$, 
the equation in \eqref{nue} has a positive solution $X=X_c$ on $[T_0,\infty)$ satisfying $X(T_0)=c$ and $\lim_{t\rightarrow \infty}X(t)=0$.

Fix $T_0>1/\varepsilon$. Let $c\in (0, \varepsilon)$ and $X=X_c$. Fix $r_0>0$ arbitrary. Then, the differential equation for $r$ in \eqref{pox5}, subject to $r(T_0)=r_0$, 
has a unique solution given by
	\begin{equation*} 
	r(t)=r_0 \exp\left( \int_{T_0}^ t \frac{ds}{s\left (\theta+1+q\frac{X(s)}{s}\right )}\right) \quad \mbox{for } t\geq T_0. 
	\end{equation*}
	Since $[T_0,\infty)\ni t\longmapsto r(t)$ is decreasing, we can express $t$ as a function of $r$. 
	We define $u_c(r)$ by 
	\begin{equation} \label{ba12} 
	u_c(r)=\left (t\,r^{-\theta-1}\right )^{\frac{1}{q}}\quad \mbox{for every } r\in (0,r_0].  
	\end{equation}
Then, using \eqref{pox5}, \eqref{nue} and $\lim_{t\to \infty} X(t)=0$, we obtain that $u_c$ is a positive increasing solution of \eqref{faz} in $(0,r_0)$
with $\lim_{r\rightarrow 0^+}u_c(r)=\gamma_c\in (0,u_c(r_0))$ and $u_c(r_0)=\left (T_0\,r_0^{-\theta-1}\right )^{\frac{1}{q}}$. 

For arbitrary $\gamma\in (0,\infty)$, 
by varying $c\in(0,\varepsilon)$ and using the scaling transformation $T_j[u_c]$ in \eqref{scale} with $j=(\gamma/\gamma_c)^{1/\Theta}$, we 
obtain $R>0$ and infinitely many positive increasing solutions $u$ of \eqref{faz} in $(0,R)$ satisfying  \eqref{z2}.
\end{proof}

%%%%

\begin{lemma} \label{maria} We assume  that $m\not=1$ or $\beta>0$.  Let $L_0$ be as in \eqref{coo}. %The following hold: 
\begin{itemize}
\item[(i)] If $L_0=0$ (possible only for ${\rm sgn}\, (\Theta)={\rm sgn}\, (1-m)={\rm sgn}\, (\beta)$) or $L_0=-\infty$ (possible only for ${\rm sgn}\, (\Theta)=
{\rm sgn}\, (1-m)=-{\rm sgn}\, (\beta)$ and $\rho>0$), then \eqref{uite1} holds. 

\item[(ii)] If $L_0\in \mathbb R\setminus \{0\}$ (possible only for $\beta>0$ and $\rho>0$), then \eqref{refina} holds. 
\end{itemize}
\end{lemma}

\begin{rem}
Note that Case (II) (resp, Case (III)) in Theorem~\ref{profff1}	corresponds to Lemma~\ref{maria} (i) (resp., (ii)). 
\end{rem}

\begin{proof}
If $m\not=1$, then 
from \eqref{faz1}, we find that 
\begin{equation} \label{mida} \frac{ {\rm sgn} \, (u'(r) ) }{1-m} \frac{d}{dr} \left[ \left( r^{1-2\rho} |u'(r)| \right)^{1-m}\right] \sim 
\frac{\gamma^q}{\beta} (r^\beta)'
\quad  \mbox{as } r\to 0^+.  
\end{equation}

(A) Suppose that $\beta<0$ and $m\not=1$. Since $r^{\beta}\to \infty$ as $r\to 0$, we get $L_0=- \infty$ (resp., $0$) if $m<1$ (resp., $m>1$) 
and
$$ {\rm sgn}\,(u'(r))={\rm sgn}\, (m-1)={\rm sgn}\, \left(\frac{\beta}{1-m}\right)={\rm sgn}\, (\zeta-2\rho)\quad \mbox{for all } r\in (0,r_0).$$ 

(B) Suppose that $\beta>0$. 
Letting $0<C<\gamma^q/\beta$ (resp., $C>\gamma^q/\beta$), we get that $H(r)$ is increasing (resp., decreasing) for every $r>0$ small, where we define 
$$ H(r)= \left\{ \begin{aligned} 
&  \frac{ {\rm sgn} \, (u'(r) ) }{1-m}  \left( r^{1-2\rho} |u'(r)| \right)^{1-m}-C r^\beta && \mbox{if } m\not=1,&\\
&  {\rm sgn} \, (u'(r) ) \log \left(r^{1-2\rho} |u'(r)|\right ) -C r^\beta && \mbox{if } m=1. &
\end{aligned} \right. $$
Thus, necessarily, $\lim_{r\to 0^+} H(r)\in \mathbb R$. Hence, if $m<1$ (resp., $m=1$), then $L_0\in \mathbb R$ (resp., $L_0\in \mathbb R\setminus \{0\}$). Moreover, if $m> 1$, then $L_0\in [-\infty,\infty)\setminus \{0\}$. 

\vspace{0.2cm}
Whether Case (A) or Case (B) holds, we distinguish the following situations: 

\vspace{0.2cm}
$\bullet$ Let $L_0\in \mathbb R\setminus \{0\}$. This can only happen when $\rho>0$ and $\beta>0$ (see Case (A)).  
Then, by L'H\^opital's rule, we arrive at
	\begin{equation} \label{faz3}		
	\lim_{r\to 0^+} \frac{r^{1-2\rho}u'(r)-L_0}{r^\beta}
	=\frac{\gamma^q |L_0|^m}{\beta}.	
	\end{equation}	
	Hence, we readily conclude  \eqref{refina} with $C_1=L_0/(2\rho)$. 

\vspace{0.2cm}
$\bullet$ Let $L_0=0$ (possible only for ${\rm sgn}\,(\beta)={\rm sgn}\, (1-m)$) or $L_0=-\infty$ (possible only for ${\rm sgn}\,(\beta)=-{\rm sgn}\, (1-m)$ and $\rho>0$).  
From \eqref{mida} and L'H\^opital's rule, we obtain that 
\begin{equation} \label{sifo}  r |u'(r)|\sim \left( \frac{\gamma^q}{|\zeta-2\rho|} \right)^{\frac{1}{1-m}}
r^\zeta\quad \mbox{as }r\to 0^+. 
\end{equation}
In view of \eqref{z2}, we have $\zeta>0$, which implies that ${\rm sgn}\, (\Theta)={\rm sgn}\,(1-m)$. Hence, from \eqref{sifo}, we conclude  
\eqref{uite1}. This ends the proof of Lemma~\ref{maria}. 
\end{proof}

\begin{lemma}[Existence, Case (II) in Theorem~\ref{profff1}] \label{viz2}
Let $\beta\neq 0$, $m\not=1$, $\Theta\not=0$ and ${\rm sgn}\, (\Theta)={\rm sgn}\, (1-m)$. 
If $\beta>0$ (resp., $\beta<0$), 
then for every $\gamma\in (0,\infty)$, there exists $R>0$ small such that 
\eqref{faz} has a
positive solution (resp., infinitely many positive solutions) satisfying \eqref{uite1}, i.e.,
 \begin{equation}
 \label{uite10}
 u(r)=\gamma+  \frac{{\rm sgn}\,(\zeta-2\rho)}{\zeta} \left( \frac{\gamma^q}{|\zeta-2\rho|} \right)^\frac{1}{1-m}
r^{\zeta} \,(1+o(1)),\ \mbox{where } \zeta:=\frac{\kappa \,\Theta}{1-m}>0. 
 \end{equation} 
\end{lemma}

\begin{proof} Suppose that $R>0$ and $u$ is a positive solution of \eqref{faz} satisfying \eqref{uite10}. 
Let
 $$ \vartheta_1=\zeta,\quad  L_{\mathfrak G}= {\rm sgn}\,(\zeta-2\rho) 
\left( \frac{\gamma^q}{|\zeta-2\rho|} \right)^\frac{1}{1-m}\quad\text{and }  \mathfrak G(r)=r^{1-\zeta}u'(r)\quad\text{for all }r\in (0,R).$$ 
We fix $\eta\in (0, \vartheta_1/2)$.  
Let $r_0\in (0,R)$ be as small as needed and define $ t=\log \,(r_0/r)$, as well as 
\begin{equation} \label{trgg}
X_1(t)=r^{-\eta}\left( u(r)-\gamma \right),\quad X_2(t)=\mathfrak{G}(r) -L_{\mathfrak G}\quad \mbox{and}\quad X_3(t)=r^\eta.
\end{equation}

By Lemma~\ref{maria}, we have $\mathbf X(t)=(X_1(t),X_2(t), X_3(t))\to (0,0,0)$ as $r\to 0^+$. Hence, for every $\varepsilon>0$, 
we can take $r_0=r_0(\varepsilon)>0$ small such that 
$|X_j(t)|<\varepsilon$ for every $r\in (0,r_0)$ and $j=1,2,3$.  
We fix $\varepsilon>0$ small so that $g(\xi)=|\xi_2+L_\mathfrak G|^m$  is well-defined and smooth on the 
open ball $B_\varepsilon(\mathbf 0)$ in $\mathbb R^3$, where 
$\mathbf 0=(0,0,0)$. 
%For every $\xi=(\xi_1,\xi_2,\xi_3)\in B_\varepsilon(\mathbf 0)$, we define 
%\begin{equation} \label{ceva1} g(\xi)= |\Theta_-|^m (\xi_1\xi_3+\gamma)^m 
% \left ( 1- \frac{(\xi_2+L_{\mathfrak G})|\xi_3|^{\frac{\vartheta}{\eta}}}{ \left(\xi_1\xi_3+\gamma\right) \Theta_-} \right)^m  
%\end{equation}
%
%Let $a_2>0$ be defined by
%\begin{equation} \label{aare} a_2= \vartheta+ 2\sqrt{\rho^2-\lambda}.
%\end{equation}
Hence, for every $t\geq 0$, we arrive at 
\begin{equation*}
\left\{\begin{aligned}
& X_1'(t)=\eta \,X_1(t) -\left(X_2(t)+L_{\mathfrak G}\right) \left(X_3(t)\right)^{\frac{\vartheta_1}{\eta} -1},\\
& X_2'(t)=\beta \left(X_2(t)+L_{\mathfrak G}\right)
-\left( X_1(t) X_3(t)+\gamma \right)^q  g(\mathbf X(t)),\\
& X_3'(t)=-\eta \,X_3(t).
\end{aligned}
\right.
\end{equation*}

We find that 
$\mathbf{X}(t)$ solves for $t\geq 0$ 
the following system
\begin{equation}\label{shss}
\mathbf{X}'(t)=A\mathbf{X}+\mathbf{F}(\mathbf{X}),
\end{equation}
where $A$ is a $3\times 3$ diagonal matrix given by $A={\rm diag}\, [\eta,\beta,-\eta]$ and $\mathbf F$ is a $C^1$-function on the open ball $B_\varepsilon(\mathbf 0)\subset \mathbb R^3$ satisfying 
$\mathbf F(\mathbf 0)=\mathbf 0$ and $D\mathbf F(\mathbf 0)=0$ since $0<\eta<\vartheta_1/2$. Thus, $\mathbf 0$ is a saddle critical point for \eqref{shss} since $A$ has two positive eigenvalues and one negative eigenvalue when $\beta>0$ (resp., two negative eigenvalues and one positive eigenvalue when $\beta<0$). In both situations, we can apply the Stable Manifold Theorem to \eqref{shss} so that the behaviour of \eqref{shss}   
near $\mathbf{X}=\mathbf 0$ is approximated by the behaviour of its linearization $\mathbf{X}'=A\mathbf X$ at $\mathbf X=\mathbf 0$. 
Let $\phi_t$ be the flow of the system \eqref{shss} in relation to which we obtain the existence of a one-dimensional (local) stable manifold $S$ when $\beta>0$, and a two-dimensional stable  manifold $S$ when $\beta<0$.

So, when $\beta>0$ there exist $\varepsilon_0>0$ small and two differentiable functions $w_1,w_2:(-\varepsilon_0,\varepsilon_0)\rightarrow (-\varepsilon_0,\varepsilon_0)$ such that the stable manifold $S$ is 
defined by  
$X_1=w_1(X_3)$ and $X_2=w_2(X_3)$ 
such that for all $t\geq 0$, we have $\phi_t(S)\subseteq S$ and for all $\mathbf{x}_0 \in S$, $\lim_{t\rightarrow \infty} \phi_t(\mathbf{x}_0)=0$. 
Thus,  
for every $x_{03}\in (-\varepsilon_0, \varepsilon_0)$, the system \eqref{shss} subject to the initial condition
\begin{equation}\label{mia22}
\mathbf{X}(0)=(w_1(x_{03}), w_2(x_{03}),x_{03})
\end{equation}
has a solution $\mathbf{X}$ on $[0,\infty)$ satisfying 
\begin{equation} \label{imcc} \lim_{t\to \infty} \mathbf X(t)=\mathbf 0.  
\end{equation}

On the other hand, when $\beta<0$, there exist $\varepsilon_0>0$ and $w:(-\varepsilon_0,\varepsilon_0)\times (-\varepsilon_0,\varepsilon_0)\rightarrow (-\varepsilon_0,\varepsilon_0)$ such that for all $(x_{02},x_{03})\in (-\varepsilon_0,\varepsilon_0)\times (-\varepsilon_0,\varepsilon_0)$, the system \eqref{shss} subject to the initial condition
\begin{equation}\label{mia234}
\mathbf{X}(0)=(w(x_{02},x_{03}), x_{02},x_{03})
\end{equation}
has a solution $\mathbf{X}$ on $[0,\infty)$ satisfying \eqref{imcc}. 

In particular, in both cases, we find that 
$X_3(t)=x_{03}\,e^{-\eta t}$ for all $t\geq 0$. 

Let $0<r_0<\varepsilon_0^{1/\eta}$ be arbitrary and take $x_{03}=r_0^\eta\in (0,\varepsilon_0)$. In addition, for  $\beta<0$, we consider $x_{02}\in (-\varepsilon_0,\varepsilon_0)$ be arbitrary. 
We construct a positive solution $u$ of \eqref{faz} in $(0,r_0)$ satisfying \eqref{uite10}. 	
Let $\mathbf{X}$ be the above solution of \eqref{shss}, subject to \eqref{mia22} (resp., \eqref{mia234}), when $\beta>0$ (resp., when $\beta<0$). 
By setting $r=r_0\,e^{-t}$ for $t\geq 0$, we get $X_3(t)=r^\eta$. We define 
\begin{equation} \label{mibb}
u(r)= \left(X_1(t)+\gamma\right)X_3(t)\quad \mbox{for every } r\in (0,r_0].
\end{equation}
In view of \eqref{shss}, by differentiating \eqref{mibb} with respect to $t$, we regain the expression of $X_2$ in \eqref{trgg}. By differentiating such an expression with respect to $t$ and using the differential equation for $X_2$ in \eqref{shss}, we obtain that $u$ is a positive solution of \eqref{faz} in $(0,r_0)$ satisfying \eqref{uite10} since \eqref{imcc} holds. In the case $\beta<0$, we get infinitely many positive solutions of \eqref{faz} in $(0,r_0)$ by varying $x_{02}\in (-\varepsilon_0,\varepsilon_0)$. 
This concludes the proof of the lemma.  
\end{proof}

%Suppose that $R>0$ and $u$ is a positive solution of \eqref{faz} satisfying \eqref{uite10}. The proof follows the same steps as in Lemma \ref{viz1} with the following modifications 
%$$ \mathfrak G(r)=r^{1-\zeta}u'(r),\quad  \vartheta_1=\zeta,\quad  L_{\mathfrak G}= {\rm sgn}\,(\zeta-2\rho) 
%\left( \frac{\gamma^q}{|\zeta-2\rho|} \right)^\frac{1}{1-m}.$$ 
%We fix $\eta\in (0,\zeta/2)$.  
%We let $r_0\in (0,R)$ small as needed and define $t=\log\, (r_0/r)$, as well as 
%$$ X_1(t)=r^{-\eta} (u(r)-\gamma),\quad X_2(t)=r^{1-\zeta}u'(r)-L_{\mathfrak G},\quad X_3(t)=r^\eta.
%$$
%Here, $g(\xi)$ in \eqref{ceva1}  and $a_2$ in \eqref{aare} are given by 
%$g(\xi)=\left |\xi_2+L_{\mathfrak G}\right |^m$ and $a_2=\beta$, respectively. 
%Hence, 
%$\mathbf X(t)$ solves \eqref{shs}, where $A={\rm diag}\, [\eta,\beta,-\eta]$. 
%Then, the arguments from Lemma~\ref{viz1} carry over to yield the existence of $R>0$ small such that \eqref{faz} has  
%a positive solution when $\beta>0$ (respectively, infinitely many positive solutions when $\beta<0$) 
%satisfying \eqref{uite10}. When $\beta<0$ the local stable manifold $S$ is two-dimensional, which for $R>0$ small yields the existence of infinitely many 
%positive solutions of \eqref{faz} satisfying \eqref{uite10}.

%%%%

%		
	\begin{lemma}[Existence, Case (III) in Theorem~\ref{profff1}] \label{probc}
		Let  $\beta>0$ and $\rho>0$. Then, 
		for every $\gamma\in (0,\infty)$ and $C_1\in \mathbb R\setminus \{0\}$, there exists $R>0$ such that \eqref{faz} 
	has a positive solution $u$  satisfying \eqref{refina}, i.e.,
	\begin{equation} \label{saii}
	u(r)=\gamma+ C_1\,r^{2\rho}+\frac{\gamma^q (2\rho\,|C_1|)^m}{\beta \left(2\rho+\beta\right)}r^{2\rho+\beta}(1+o(1))\quad \mbox{as } r\to 0^+.
	\end{equation} 
	\end{lemma}
	
	\begin{proof} Let $\gamma\in (0,\infty)$ and $C_1\in \mathbb R\setminus \{0\}$ be arbitrary. 
		Suppose that $u$ is a positive solution of \eqref{faz} satisfying \eqref{saii}. 
	Fix $\eta\in (0,\min\,\{\beta,2\rho\})$. Let $\varepsilon>0$ be small such that 
		$$ 
		(\varepsilon+|C_1|)\,\varepsilon^{\frac{2\rho}{\eta}}<\gamma/2 
		\quad \mbox{and}\quad 
		(\varepsilon+C_2)\,\varepsilon^{\frac{\beta}{\eta}}<\rho\,|C_1|,\quad \mbox{where } C_2:=\frac{\gamma^q\left(2\rho \,|C_1|\right)^m}{\beta}>0. 
		$$
		If needed, we diminish $R>0$ to ensure that $R<\varepsilon^{1/\eta}$, as well as 
		\begin{equation} \label{gup} 
		\left| r^{-2\rho}(u(r)-\gamma)-C_1\right|<\varepsilon\quad \mbox{and}\quad \left| r^{-\beta}\left (r^{1-2\rho}u'(r)-2\rho\, C_1\right ) -C_2\right|<\varepsilon		\end{equation}
		 for every $r\in (0,R]$. 
	We set $r=R\,e^{-t}$ and for $t\geq 0$, we  define $(X_1(t), X_2(t), X_3(t))$ as follows
	\begin{equation}\label{s91}
	X_1(t)=r^{-2\rho}(u(r)-\gamma)-C_1, \quad X_2(t)=r^{-\beta}\left (r^{1-2\rho}u'(r)-2\rho\, C_1\right ) -C_2\quad \text{and}\quad X_3(t)=r^\eta.
	\end{equation} 
	Observe that $X_3(t)>0$ for all $t\geq 0$ and 
	\begin{equation} \label{faz4}
	(X_1(t), X_2(t), X_3(t))\rightarrow (0,0,0)\quad \mbox{as } t\rightarrow \infty.
	\end{equation}
	
	Let $g:\mathbb R^3\to \mathbb R$ be a $C^1$-function on $\mathbb R^3$ satisfying
	\begin{equation} \label{gif}  g(t_1,t_2,t_3)=\left[ \left(t_1+C_1\right)|t_3|^{2\rho/\eta}+\gamma\right]^q \left( \frac{\left(t_2+C_2\right)|t_3|^{\beta/\eta}}{C_1}
	+2\rho \right)^m |C_1|^m
	\end{equation} for all $|t_j|<\varepsilon$ with $j\in \{1,2,3\}$. Our choice of $\eta>0$ gives $\beta/\eta>1$ and $2\rho/\eta>1$.

	Note that $ g(\mathbf 0)=\gamma^q (2\rho\,|C_1|)^m=C_2\beta$ and \eqref{gup} gives that $|X_j(t)|<\varepsilon$ for $j\in \{1,2,3\}$ and all $t\geq 0$. 
	Hence, $\mathbf{X}(t)=(X_1(t), X_2(t), X_3(t))$ satisfies for $t\geq 0$ the nonlinear differential system 
	\begin{equation}\label{h2}
	\left\{ \begin{aligned} 
	& X_1'(t)=2\rho\, X_1(t)-(X_2(t)+C_2)\,|X_3(t)|^{\beta/\eta}
	:=H_1(\mathbf X), \\
	& X_2'(t)=\beta\left(X_2(t)+C_2\right)-
	g(\mathbf X)
	:=H_2(\mathbf X),\\
	& X_3'(t)=-\eta X_3(t)
	:=H_3(\mathbf X).
	\end{aligned}
	\right.	
	\end{equation}

	We see that $\mathbf 0=(0,0,0)$ is a hyperbolic critical point for \eqref{h2}, which  
	we rewrite as
	\begin{equation}\label{b2}
	\mathbf{X}'(t)=A\mathbf{X}+\mathbf{F}(\mathbf{X}),
	\end{equation}
	where $A=D\mathbf{H}(\mathbf 0)={\rm diag}\,[2\rho,\beta,-\eta]$ and 
	$\mathbf{F}(\mathbf{X})=\mathbf{H}(\mathbf{X})-A\mathbf{X}$, $\mathbf{F}\in C^1(\mathbb R^3)$ with $\mathbf{F}(\mathbf 0)=\mathbf 0$ and $D\mathbf{F}(\mathbf 0)=0$. 
The eigenvalues $\mu_j$ ($j=1,2,3$) of the matrix $A$ are all non-zero and, more precisely, 
	\begin{equation*}
	\mu_1=2\rho>0,\quad \mu_2=\beta>0, \quad \text{and } \quad \mu_3=-\eta<0.
	\end{equation*}

The critical point $\mathbf 0$ for \eqref{b2} is a saddle point and we can apply the Stable Manifold theory.
Thus, the behavior of the nonlinear system \eqref{b2} near $\mathbf{X}=\mathbf 0$ is approximated by the behavior of its linearization $\mathbf{X'}=A\mathbf{X}$ at $\mathbf{X}=\mathbf 0$. 
Let $\phi_t$ be the flow of the nonlinear system \eqref{b2}. The Stable Manifold Theorem applied to  \eqref{b2} gives the existence of a (local) stable manifold $S$ that is one-dimensional.
More precisely, there exist $\varepsilon_0>0$ small and two differentiable functions $w_1,w_2:(-\varepsilon_0,\varepsilon_0)\rightarrow (-\varepsilon_0,\varepsilon_0)$ such that the stable manifold $S$ is 
defined by  
$X_1=w_1(X_3)$ and $X_2=w_2(X_3)$ 
such that for all $t\geq 0$, we have $\phi_t(S)\subseteq S$ and for all $\mathbf{x}_0 \in S$, $\lim_{t\rightarrow \infty} \phi_t(\mathbf{x}_0)=0$. 
Therefore,  
	for every $x_{03}\in (-\varepsilon_0, \varepsilon_0)$, the system \eqref{b2} subject to the initial condition
	\begin{equation}\label{m2}
	\mathbf{X}(0)=(w_1(x_{03}), w_2(x_{03}),x_{03})
	\end{equation}
	has a solution $\mathbf{X}$ on $[0,\infty)$ satisfying \eqref{faz4}. In particular, we have 
	$X_3(t)=x_{03}\,e^{-\eta t}$ for all $t\geq 0$. As $\varepsilon_0>0$ is small, using \eqref{gif}, for all $t\geq 0$, we get 
	\begin{equation} \label{yop} g(\mathbf{X}(t))=
	\left[ \left(X_1(t)+C_1\right)|X_3(t)|^{2\rho/\eta}+\gamma\right]^q \left|\left(X_2(t)+C_2\right)|X_3(t)|^{\beta/\eta}+2\rho \,C_1\right|^m.
	\end{equation} 
		
	Let $0<R<\varepsilon_0^{1/\eta}$ be arbitrary. 
	We now construct a positive solution $u$ of \eqref{faz} satisfying \eqref{saii}. 	
	To this end, let $\mathbf{X}_{R}$ be the above solution $\mathbf{X}$ of \eqref{h2}, subject to \eqref{m2}, corresponding to $x_{03}=R^\eta\in (0,\varepsilon_0)$. 
	By setting $r=R e^{-t}$ for $t\geq 0$, we get $X_3(t)=r^\eta$. We define
	\begin{equation} \label{mia}
	u(r)= \left(X_1(t)+C_1\right)(X_3(t))^{2\rho/\eta}
	+\gamma=
	r^{2\rho}\left [X_1(\log\,(R/r))+C_1\right]+\gamma 
	\end{equation}
for every $ r\in (0, R]$. 
	Using the differential equation for $X_1$ in \eqref{h2} and differentiating \eqref{mia} with respect to $t$, we arrive at the expression of $X_2$ in \eqref{s91}. By differentiating $X_2$ with respect to $t$ and using the differential equation for $X_2$ in \eqref{h2}, jointly with \eqref{yop}, we conclude that $u$ is a positive solution of \eqref{faz} satisfying \eqref{saii} since $\lim_{t\to  \infty}\mathbf{X}_{R}(t)=\mathbf 0$. %This ends the proof.   
\end{proof}

%%%%%

%%%%%%%%%
\section{The constant asymptotic profile for $\lambda=\beta=0$ (Proof of Theorem \ref{equa1})}

Let \eqref{cond1} hold, $\lambda=0$ and $\beta=\kappa \Theta+2\rho\left(m-1\right)=0$. Unless otherwise stated, we always assume that $u$ is a 
positive solution of \eqref{faz} for $R>0$ such that
\begin{equation}\label{co1}
\lim_{r\rightarrow 0^+}u(r)=\gamma\in (0,\infty)\quad \mbox{and}\quad r^*(u)=0.
\end{equation}

\begin{lemma} \label{lok}
Let $m=1$ and $\Theta=0$. Then, $u'(r)\not=0$ for all $r\in (0,R)$ and, moreover,  
$2\rho+\gamma^q {\rm sgn}\, (u'(r))\geq 0$. By letting $u(R^-)=\lim_{r\nearrow R} u(r)$, we have  for all $r\in (0,R)$ % , we have
\begin{equation} \label{fip}
\int_{u(r)}^{u(R^-)} \frac{ds}{ 2\rho \left(s-\gamma\right)+ \frac{{\rm sgn}\, (u'(r))}{q+1}\left( s^{q+1}-\gamma^{q+1}\right)}=\log \left(\frac{R}{r}\right). 
\end{equation} 		
$\bullet$ If $2\rho+\gamma^q \,{\rm sgn}\, (u'(r))>0$, then there exists a constant $C_1>0$ such that as $r\to 0^+$
\begin{equation} \label{sof}
u(r)=\gamma+C_1 {\rm sgn}\, (u'(r)) \,r^{2\rho+\gamma^q \,{\rm sgn}\, (u'(r))}+
q\,\gamma^{q-1}\,(C_1)^2 \,
\frac{r^{2\left(2\rho+\gamma^q\, {\rm sgn}\,(u'(r)) \right)} }
{2\left(2\rho+\gamma^q\,{\rm sgn}\,(u'(r))   \right)} \,  (1+o(1)). 
\end{equation} 
$\bullet$ If $2\rho+\gamma^q \,{\rm sgn}\, (u'(r))=0$, then as $r\to 0^+$, 
\begin{equation} \label{sof1}
u(r)=\gamma+\frac{2\, \gamma^{1-q}}{q}\frac{{\rm sgn}\, (u'(r))}{\log\, (1/r)}\left[
1-\frac{2\,(q-1)\, {\rm sgn}\, (u'(r))}{3\,q\,\gamma^q} \frac{\log \,(\log \,(1/r))}{\log \,(1/r)}(1+o(1)) \right]. 
\end{equation}
Furthermore, if also $q=1$, then it holds 
\begin{equation} \label{sof2} u(r)=\gamma+\frac{2\, {\rm sgn}\, (u'(r))}{\log \left(\frac{R}{r}\right)+
	\frac{2\, {\rm sgn}\, (u'(r))}{u(R^-)-\gamma}}\quad \mbox{for all } r\in (0,R).
\end{equation}
\end{lemma}

\begin{proof}
Let $r_0\in (0,R]$ be such that $u'(r)\not=0$ for every $r\in (0,r_0)$. 
Since $m=1$ and $\beta=\Theta=0$, from \eqref{faz1},  we have
\begin{equation} \label{erb}
\frac{d}{dr} (ru'(r))=u'(r) \left[2\rho +u^q(r)\,{\rm sgn}\, (u'(r))\right]
\quad \mbox{for every } r\in (0,r_0).
\end{equation} 
Using that $\lim_{r\to 0^+} u(r)=\gamma$, we have 
$$ 2\rho+\gamma^q \,{\rm sgn}\, (u'(r)) \geq 0$$ and 
$\lim_{r\to 0^+} ru'(r)=0$. 
Indeed, if $2\rho+\gamma^q \,{\rm sgn}\, (u'(r)) < 0$, then ${\rm sgn}\, ((ru'(r))')=-{\rm sgn}\,(u'(r))$ for $r\in (0,r_0)$, 
which means that $\lim_{r\to 0^+} ru'(r)=0$ and we obtain a contradiction. 

From \eqref{erb}, it follows that 
$$ {\rm sgn}\, (u'(r)) \frac{d}{dr} (\log \,(r^{1-2\rho} |u'(r)|))=u^q(r) \frac{d}{dr} (\log r)\quad \mbox{for all } r\in (0,r_0).  
$$
Then, we infer that $ \lim_{r\to 0^+} {\rm sgn }\, (u'(r)) \log\, (r^{1-2\rho} |u'(r)|)=-\infty$. 
We have two possibilities:

(1) Let $L_0= \lim_{r\to 0^+} r^{1-2\rho} u'(r)=0$ and $u'(r)>0$ for every $r\in (0,R)$;

(2) Let $L_0=-\infty$ and $u'(r)<0$ for every $r\in (0,r_0)$. We  show that $u'(r)<0$ for all $r\in (0,R)$. 
Indeed, let $ r_1=\sup \{ r \in (0,R):\ u'(t)<0\ \mbox{ for all } t\in (0, r) \}$. Then,  
from \eqref{erb}, we have $ (ru'(r))'<u'(r) (2\rho-\gamma^q )\leq 0$ for every $r\in (0,r_1)$. Hence, 
$ru'(r)$ is decreasing on $(0,r_1)$.  Hence, $r_1=R$ and $u'<0$ on $(0,R)$. 

Suppose either Case (1) or Case (2). By integrating \eqref{erb}, we arrive at
\begin{equation} \label{iomm1} ru'(r)=2\rho\left(u(r)-\gamma\right)+\frac{{\rm sgn}\, (u'(r))}{q+1}
\left( u^{q+1}(r)-\gamma^{q+1}\right)\quad \mbox{for all } r\in (0,R). 
\end{equation}
This proves the claim in \eqref{fip}. 

With the change of variable $v(t)=u(r)-\gamma$
with $t=\log \,(1/r)$ for $r\in (0,\min\{1,R\})$,   
we have 
\begin{equation} \label{fom} -\frac{dv}{dt}=2\rho\, v(t)+\frac{\gamma^{q+1} \,{\rm sgn}\, (u'(r))}{q+1}\left[ 
\left(1+ \frac{v(t)}{\gamma}
\right)^{q+1} -1
\right]\quad \mbox{for every } t>0\ \mbox{large}.
\end{equation}
Remark that ${\rm sgn}\, v(t)={\rm sgn}\, (u'(r))$. 

$\bullet$ When $2\rho +\gamma^q \,{\rm sgn}\,(u'(r)) >0$, we show that there exists a positive constant $C_1$ such that 
 	\begin{equation} \label{nim1}
 	 \lim_{t\to \infty} 
 	 e^{\left[2\rho+\gamma^q\, {\rm sgn}\, (u'(r)) \right] t}\left(
 	 e^{\left[2\rho+\gamma^q\, {\rm sgn}\, (u'(r))\right]\,t}\,|v(t)|-C_1\right)=\frac{q\,\gamma^{q-1}(C_1)^2}{2\left[2\rho+\gamma^q\, {\rm sgn}\, (u'(r))\right]}. 
 \end{equation} 
From \eqref{fom}, we have  
\begin{equation} \label{emu} v'(t)+\left[2\rho+\gamma^q \,{\rm sgn}\,(u'(r))\right] v(t)=-
\frac{q\,\gamma^{q-1}}{2} {\rm sgn}\, (u'(r))\,v^2(t) \left(1+o(1)\right)\quad \mbox{as } t\to \infty. 
\end{equation} Hence, $t\longmapsto e^{\left[ 2\rho+\gamma^q  {\rm sgn}\, (u'(r)) \right] \,t} |v(t)|$ is decreasing on $(T,\infty)$ for $T>0$ large and there exists
a non-negative constant $C_1$ such that 
\begin{equation} \label{nill1} \lim_{t\to \infty} 
 e^{\left[ 2\rho+\gamma^q  {\rm sgn}\, (u'(r)) \right] \,t} |v(t)|
=C_1. \end{equation}
We show that $C_1>0$. Indeed, fix a constant $c>\frac{q\,\gamma^{q-1}  }
{ 2 \left[ 2\rho+\gamma^q  \, {\rm sgn}\, (u'(r)) \right]}$.
Then, \eqref{emu} yields that
$$\mathcal E(t)=\frac{1}{e^{(2\rho+\gamma^q\, {\rm sgn}\, (u'(r)))\,t}\,|v(t)|}+ c\,e^{-(2\rho+\gamma^q \, {\rm sgn}\, (u'(r))) t}$$ 
 	is decreasing on $(T,\infty)$ for $T>0$ large and $\lim_{t\to \infty} \mathcal E(t)=1/C_1$. So, we have $C_1>0$. 
 	
 	Returning to \eqref{emu}, we see that $v'(t)/v(t)\to -(2\rho+\gamma^q \,{\rm sgn}\, (u'(r)))$ as $t\to \infty$ and by L'H\^opital's rule, we find that there exists 
 	\begin{equation} \label{nill2} \lim_{t\to \infty} \frac{e^{\left[2\rho+\gamma^q \,{\rm sgn}\, (u'(r)) \right] t}\,|v(t)|-C_1}{v(t)}
 	 	=\frac{q\,\gamma^{q-1}C_1 {\rm sgn}\, (u'(r))}{2\left[2\rho+\gamma^q\, {\rm sgn}\, (u'(r))\right]}.
 	\end{equation}
 	From \eqref{nill1} and \eqref{nill2},  
 	we conclude the proof of \eqref{nim1}, which is equivalent to \eqref{sof}.  

\vspace{0.2cm}
$\bullet$ When $2\rho+\gamma^q \,{\rm sgn}\,(u'(r))=0$, we distinguish two situations: 

(i) If $q=1$, then from \eqref{fip} we readily obtain \eqref{sof2}.

(ii) If $q\not=1$, then from \eqref{fom}, we get
\begin{equation} 
\label{nona1} -v'(t)= \frac{q\,\gamma^{q-1}}{2}\,v^2(t)\, {\rm sgn}\, (u'(r)) \left[1 +\frac{q-1}{3\gamma}\,v(t) \left(1+o(1)\right) 
\right]
\quad \mbox{as } t\to \infty.  
\end{equation} This implies that $t\,|v(t)|\to 2\gamma^{1-q}/q$ as $t\to \infty$. 
Fro $t>0$ large, we define
$$ \mathcal N(t)=\frac{1}{|v(t)|}-\frac{q\,\gamma^{q-1}}{2}\,t. $$
Using \eqref{nona1}, we see that 
$$\lim_{t\to \infty} t\,\mathcal N'(t)=\frac{q\left(q-1\right)\gamma^{q-2}}{6}\, \lim_{t\to \infty} t\,v(t)=
\frac{q-1}{3\gamma} \,{\rm sgn}\, (u'(r))\not=0.$$ 
This gives that $\mathcal N(t)\to +\infty\cdot  \,{\rm sgn}\,( \left(q-1\right) u'(r) )$. 
Then, by L'H\^opital's rule, we infer that 
$$ \lim_{t\to \infty}\frac{\mathcal N(t)}{\log t}=
\lim_{t\to \infty} t\,\mathcal N'(t) =\frac{q-1}{3\gamma} \,{\rm sgn}\, (u'(r)).
 $$
 Hence, using the definition of $\mathcal N$, we arrive at \eqref{sof1} since $v(t)=u(r)-\gamma$
 and $t=\log \,(1/r)$.
\end{proof}

%%%%%%%%

\begin{lemma} \label{manu}
Assume that $\Theta=\rho=0$. Then, necessarily 
$m\in [1, 2)$, $u'>0$ on $(0,R)$ and 
\begin{equation} \label{van0}
\int_{u(r)}^{u(R)}\frac{ds}{\left(s^{q+1}-\gamma^{q+1}\right)^{\frac{1}{2-m}}}=\left( \frac{2-m}{q+1}\right)^{\frac{1}{2-m}}\log \left(\frac{R}{r}\right)
\quad \mbox{for all } r\in (0,R). 
\end{equation}
If $m=1$, then there exists a constant $C_1>0$ such that \eqref{1sof} holds, namely, 
\begin{equation} \label{1sifa}
u(r)=\gamma+C_1 \,r^{\gamma^q}+
q\,\gamma^{q-1}\,(C_1)^2 \,
\frac{r^{2\gamma^q } }
{2\gamma^q } \,  (1+o(1)) \quad \mbox{as } r\to 0^+.. 
\end{equation} 
Furthermore, if $m\in (1,2)$, then we have 
\begin{equation}
\label{mosaa}
u(r)=\gamma+ \frac{(m-1)^{\frac{2-m}{1-m}}}{2-m} \gamma^{\frac{q}{1-m}} \left(\log \,(1/r)\right)^{\frac{2-m}{1-m}} (1+o(1))\quad \mbox{as } r\to 0^+.
\end{equation}

\end{lemma}

\begin{proof}
From \eqref{faz1} with $\rho=0$, we obtain that 
\begin{equation} \label{0sbn} 
\left(ru'(r) \right)'=r^{\theta+1} u^q(r)|u'(r)|^m\geq 0\quad \mbox{for all }r\in (0,R).\end{equation}
Hence, using \eqref{co1}, we get that $r\rightarrow ru'(r)$ is non-decreasing on $(0,R)$, $\lim_{r\to 0^+} ru'(r)=0$ and $u'(r)>0$ for every $r\in (0,R)$.  
For every $r\in (0,R)$, we define
$\mathfrak T_m(r)$ by 
\begin{equation} \label{tamd}
\mathfrak{T}_m(r)=
\left\{ \begin{aligned} 
& \frac{(r u'(r))^{2-m}}{2-m}-\frac{u^{q+1}}{q+1}\,
 && \mbox{if } m\not=2, &\\
& \log\, (ru'(r))-\frac{u^{q+1}}{q+1} && \mbox{if } m=2.&
\end{aligned} \right.
\end{equation}
Since $\mathfrak T_m'(r)=0$ for all $r\in (0,R)$,  
there exists a constant $C\in \mathbb R$ such that $\mathfrak T_m(r)=C$ for all $ r\in (0,R)$. This necessarily 
implies that $m<2$ and 
\begin{equation} \label{van2} 
\frac{u'(r)}{\left(u^{q+1}-\gamma^{q+1}\right)^\frac{1}{2-m}}=\left(\frac{2-m}{q+1}\right)^\frac{1}{2-m}\frac{1}{r} \quad \mbox{for all } r\in (0,R). 
\end{equation} 
We also show that $m\geq 1$. 
For every $\tau>\gamma$, we define 
$$ \mathcal P_\gamma(\tau)= \int_\tau^\infty \frac{ds}{(s^{q+1}-\gamma^{q+1})^{\frac{1}{2-m}}}. 
$$ Since $(q+1)/(2-m)>1$, we see that $\mathcal P(\tau)<\infty$ for every $\tau>\gamma$. Note that $\lim_{\tau\searrow \gamma} \mathcal P_\gamma(\tau)=\infty$ when $1\leq m<2$, in which case $\mathcal P_\gamma:(\gamma,\infty)\to (0,\infty)$ is a decreasing and bijective function. 

By integrating \eqref{van2}, we get
\begin{equation} \label{van3} \mathcal P_\gamma(u(r))-\mathcal P_\gamma(u(R))= 
\left(\frac{2-m}{q+1}\right)^\frac{1}{2-m}\log \left(\frac{R}{r}\right)\quad \mbox{for every } 
r\in (0,R).
\end{equation}

If $m\in (0,1)$, then $\lim_{\tau\searrow \gamma} \mathcal P_\gamma(\tau)<\infty$ and letting $r\searrow 0$ in \eqref{van3}, we 
get a contradiction since the left-hand side is finite, while the right-hand side is infinite. 
Consequently, we have $m\in [1,2)$.  

From \eqref{van3}, we conclude the proof of \eqref{van0}. 

$\bullet$ If $m=1$, then the conclusion of \eqref{1sifa} follows from Lemma~\ref{lok} with $\rho=0$. 

$\bullet$ If $m\in (1,2)$, then as in the proof of Lemma~\ref{lok}, we define 
$v(t)=u(r)-\gamma$
with $t=\log \,(1/r)$ for $r\in (0,\min\{1,R\})$. In light of \eqref{van2}, we find that
$$ -v'(t) [v(t)]^{-\frac{1}{2-m}} =\left[ \left(2-m\right) \gamma^q\right]^{\frac{1}{2-m}} \left[
1+\frac{q}{2\gamma (2-m)}\, v(t) (1+o(1))\right]\quad \mbox{as } t\to \infty. 
$$
In particular, using that $\lim_{t\to \infty} v(t)=0$, we get 
$$  \lim_{t\to \infty} t^{\frac{2-m}{m-1}} \, v(t)  =\left( \frac{m-1}{2-m}\right)^{\frac{2-m}{1-m}} \left[ \left(2-m\right) \gamma^q\right]^{\frac{1}{1-m}} , $$
which proves the assertion of \eqref{mosaa}. This completes the proof of Lemma~\ref{manu}. 
\end{proof}

%%%%

\begin{lemma} \label{nap}
	Let $\kappa \Theta=2\rho \left(1-m\right)\not=0$. Then, necessarily $\rho>0$ and ${\rm sgn}\, (\Theta)={\rm sgn}\, (1-m)=-{\rm sgn}\,(u'(r))$ for all $r>0$ small. Moreover, we have
	\begin{equation} \label{fiv}
		\lim_{r\to 0^+} \frac{u(r)-\gamma}{r^{2\rho}\left(\log \frac{1}{r}\right)^{\frac{1}{1-m}}}=
		\lim_{r\to 0^+}\frac{u'(r)}{2\rho\, r^{2\rho-1}\left(\log \frac{1}{r}\right)^{\frac{1}{1-m}}}=-\frac{{\rm sgn}\, (1-m)}{2\rho}\left[ \left|1-m\right|\gamma^q\right]^
		{\frac{1}{1-m}}.
		\end{equation}
\end{lemma}

\begin{proof} Assume by contradiction that $\rho<0$.  
Then, $L_0=\lim_{r\to 0^+} r^{1-2\rho} u'(r)=0$ and $u'(r)>0$ for all $r\in (0,R)$ (see the discussion in Section~\ref{nupo} before Lemma~\ref{nelem}). 
Fix $C\in (0,\gamma^q)$. For every $r>0$ small, we define 
$$ H_1(r)= \frac{(r^{1-2\rho} u'(r))^{1-m}}{1-m} -C\log r.
$$
In light of \eqref{faz1}, we obtain that $H_1'(r)>0$ for every $r>0$ small, which implies that $m>1$ and
$$ \lim_{r\to 0^+} \frac{(r^{1-2\rho} u'(r))^{1-m}}{\log \,(1/r)} =\left(m-1\right) \gamma^q. 
$$
Consequently, $u'(r)\sim [\left(m-1\right) \gamma^q]^{\frac{1}{1-m}} r^{2\rho-1} [\log\, (1/r)]^{\frac{1}{1-m}}$ as $r\to 0^+$. 
This is a contradiction with our assumption in \eqref{co1} that $\lim_{r\to 0^+} u(r)=\gamma\in (0,\infty)$.  

The above argument shows that $\rho>0$, $m\not=1$ and ${\rm sgn}\, (\Theta)={\rm sgn}\, (1-m)=-{\rm sgn}\, (u'(r))$. 

For small $r_0\in (0,R)$, we have $u'(r)\not=0$ for all $r\in (0,r_0)$ (see the explanation after \eqref{coo}).  
Then, for every $r\in (0,r_0)$, we find that
\begin{equation} \label{blam}
\frac{ {\rm sgn}\,(u'(r))}{1-m}\,\frac{d}{dr}\left[ \left( r^{1-2\rho}\,|u'(r)|\right)^{1-m}\right] =\frac{u^q(r)}{r},
\end{equation}

We fix $C\in (0,\gamma^q)$. Using \eqref{blam}, it follows that the function $$  \mathfrak{F}(r):=\frac{{\rm sgn}\,(u'(r))}{1-m}\, \left( r^{1-2\rho}\,|u'(r)|\right)^{1-m}-
C \,\log r
	$$ is increasing for every $r>0$ small enough. Therefore, $ \lim_{r\to 0^+} \left( r^{1-2\rho}\,|u'(r)|\right)^{1-m}=\infty$ and 
	$$ {\rm sgn}\,(\Theta)={\rm sgn}\,(1-m)=-{\rm sgn}\, (u'(r))\quad \mbox{for every } r\in (0,r_0).  
	$$
From \eqref{blam} and L'H\^opital's rule, we obtain that there exists 
$$ \lim_{r\to 0^+}  \frac{(r^{1-2\rho}\,|u'(r)|)^{1-m}}{\log \,(1/r)}=
|1-m|\,\gamma^q.
$$ This proves the claim of \eqref{fiv}. The proof of Lemma~\ref{nap} is complete. 
\end{proof}

\begin{lemma}  \label{qa2}
	Assume that 
$\kappa \Theta=2\rho \left(1-m\right)\not=0$. Then, for every $\gamma\in (0,\infty)$, there exists
		$R>0$ small such that \eqref{faz} has 
		infinitely many positive solutions $u$ satisfying \eqref{fiv}. 
\end{lemma}

\begin{proof}
For every $r>0$ small such that $u'(r)\not=0$, we define
\begin{equation}\label{mig1}
t:=\left (r^\Theta u(r)\right )^{-\kappa}\left | \frac{ru'(r)}{u(r)}\right |^{1-m}\quad\text{and}\quad X(t):=t\frac{ru'(r)}{u(r)}.
\end{equation}
(If $m>1$, then $u'(r)>0$ for every $r\in (0,R)$.) 
Since \eqref{fiv} holds, we see that 
$$ \lim_{r\to 0^+} \frac{t }{\log \,(1/r)} = |1-m|\quad \mbox{and}\quad
\lim_{r\to 0^+} \frac{X(t)}{r^{2\rho} \left[\log \,(1/r)\right]^{\frac{2-m}{1-m}}} =\gamma^{\kappa} (m-1) |1-m|^{\frac{1}{1-m}}.$$  
Hence, $t\to \infty$ and $X(t)\rightarrow 0$ as $r\rightarrow 0^+$. Then, there exists $T_0>0$ such that $|1-m|+ q\, X(t) >0$ for every $t\geq T_0$. By virtue of \eqref{mig1}, we find that 
\begin{equation}\label{mig2}
\frac{dr}{dt}=-\frac{r}{|1-m|+q\,X(t)}\quad \mbox{for all } t\geq T_0.
\end{equation}
By differentiating $X(t)$ in \eqref{mig1} with respect to $t$, we get 
\begin{equation}\label{mig3}
\frac{dX}{dt}=-\frac{X(t)}{|1-m|+q\,X(t)}\left [2\rho-\frac{X(t) +{\rm sgn}\, (1-m)}{t} \right ] +\frac{X(t)}{t} \quad\text{for all }t\geq T_0.
\end{equation}
Since $\lim_{t\to \infty} X'(t)/X(t) =  -2\rho /|1-m|<0$, we see that for $t>0$ large,   
$$ {\rm sgn} \,(X(t))=-{\rm sgn} \, (X'(t)) ={\rm sgn}\, (u'(r))=-{\rm sgn}\, (1-m).$$
We point out that \eqref{mig1} and \eqref{mig2} imply that 
\begin{equation} \label{miga} \frac{d}{dt} (\log \,(u(r)))=-\frac{r u'(r)}{u(r) \left[ |1-m| +q\, X(t)\right]} =-\frac{X(t)}{t \left[ |1-m| +q\, X(t)\right]}\quad \mbox{for all } t\geq T_0.  \end{equation}
Thus, by writing $u(r)=\widetilde u(t)$ and integrating \eqref{miga} with respect to $t$ over $[T_0,\infty)$, we arrive at 
$$ u(r)=\widetilde u(t)= \gamma\exp \left (\int_{t}^\infty \frac{X(s)}{s\left (|1-m|+q\,X(s)\right )}\,ds\right )\quad \mbox{for all } t\geq T_0.
$$

Let $\gamma\in (0,\infty)$ be fixed. For suitably small $R>0$, we next construct infinitely many positive solutions of \eqref{faz} satisfying \eqref{fiv}. 
To this end, we consider the ODE in \eqref{mig3} for $t\geq T_0$ and $T_0>0$ large. Note that there exists $\varepsilon>0$ small such that for every $T_0>1/\varepsilon$ and all $c\in (-\varepsilon, \varepsilon)\setminus \{0\}$, the ODE in \eqref{mig3} has a strictly monotone solution 
$X=X_c$ satisfying \begin{equation} \label{sens} X_c(T_0)=c\quad\mbox{and}\quad  \lim_{t\rightarrow \infty}X_c(t)=0. \end{equation}
We assume that $\varepsilon>0$ is small such that $q\,\varepsilon<|1-m|$. (We can construct explicit solutions in the case $q=0$, see Remark~\ref{mipp}.)

As $m\not=1$, we distinguish two cases. If $m>1$, then we fix $c\in (0,\varepsilon)$, whereas if $m<1$, then we let $c\in (-\varepsilon,0)$. 
Let $X_c$ be the above solution of \eqref{mig3} satisfying \eqref{sens}.  
We define
$$\begin{aligned} 
& \widetilde{u}_c(t)=\gamma\exp \left (\int_{t}^\infty \frac{X_c(s)}{s\left (|1-m|+q\,X_c(s)\right )}\,ds\right )\quad \mbox{for all } t\geq T_0,\\
&  R_c=\left [T_0^{2-m}|c|^{m-1}(\widetilde{u}_c(T_0))^\kappa \right ]^{-\frac{1}{\kappa\Theta}}.
\end{aligned} $$
The differential equation in \eqref{mig2}, subject to $r(T_0)=R_c$, has a unique positive solution 
$$r(t)=R_c \exp\left (-\int_{T_0}^t\frac{ds}{|1-m|+q\,X_c(s)}\right )\quad\text{for all }t\geq T_0.$$ 
Since $t\longmapsto r(t)$ is invertible on $[T_0,\infty)$, we can define $$ u_c(r)=\widetilde{u}_c(t)\quad \mbox{for every  } r\in (0, R_c].$$ 
Using \eqref{mig2}, \eqref{mig3} and \eqref{sens}, we obtain that $u_c$ is a positive solution of \eqref{faz} in $(0, R_c]$ satisfying \eqref{fiv}. 
Let $c_0\in (-\varepsilon,\varepsilon)\setminus \{0\}$ be such that ${\rm sgn}\,(c_0)={\rm sgn}\, (c)={\rm sgn}\,(m-1)$. We vary $c_0$ such that $|c_0|\in (|c|,\varepsilon)$ increases. Then, $R_{c_0}$ increases and 
the solutions $X_{c_0}$ (and, hence, $u_{c_0}$) are distinct, leading to infinitely many positive solutions $u=u_{c_0}$ of \eqref{faz} in $(0,R_{c})$ satisfying \eqref{fiv}. 
\end{proof}

\begin{rem} \label{mipp} Let $R\in (0,\infty)$. If $q=0$ in the framework of Lemma~\ref{qa2}, then for arbitrary $\gamma\in (0,\infty)$ 
and every constant $C\geq 0$, we obtain a
positive solution $u=u_C$ of \eqref{faz} in $(0,R)$ satisfying \eqref{fiv}, where $u_C$ is defined by
\begin{equation} \label{myo} u(r)=u_C(r)=\gamma+\int_0^r s^{2\rho-1} \left[C+\left(m-1\right) \log\, (R/s) \right]^{\frac{1}{1-m}}\,ds\quad \mbox{for every } r\in (0,R). 
\end{equation} 
Indeed, $\Theta=-2\rho$ and $\kappa=m-1>0$. We have $u'(r)>0$ for all  $r\in (0,R)$ and
\begin{equation} \label{myo2} \left( r^{1+\Theta} u'(r) \right)^{1-m}=t=C+\left(m-1\right)\log \,(R/r), \end{equation} 
where $C\geq 0$ is a constant. Then, \eqref{myo} follows readily from \eqref{myo2}.
\end{rem}

%%%%%%
\section{The asymptotic profile $U_0$ when $\Theta\not=0$ and $\ell>0$ (Proof of Theorem~\ref{raf})} \label{uzpro}

Throughout this section, we assume that
\eqref{cond1} holds,   
$\Theta\not=0$ and $\ell>0$.
The proof of Theorem~\ref{raf} follows from Lemmas~\ref{wisp4} and \ref{wisp5}. 
Unless otherwise stated, $u$ is a positive radial solution of \eqref{eq1} in $B_R(0)\setminus \{0\}$ for $R>0$ satisfying \eqref{uoo}, namely, 
\begin{equation}\label{uoo1}
 u(r)\sim M_0 \,r^{-\Theta}:= U_0(r)\quad \mbox{as }r\to 0^+, \ \mbox{where } M_0:=\left(|\Theta|^{-m}\ell \right)^{\frac{1}{\kappa}}. 
\end{equation}
Assume that $u\not\equiv U_0$ on
any interval $(0,r_*)$ with $r_*\in (0,R)$.

\begin{lemma} \label{wisp0}
For every $r>0$ small, we have $\Theta \, u'(r)<0$.	
\end{lemma}

\begin{proof}
We distinguish two situations. First, let $\lambda\not=0$.  From  \eqref{uoo1}, it is enough to show that $u'(r)\not =0$ for every $r>0$ small. 
Suppose the contrary, i.e., $u'(r_n)=0$ for a sequence $\{r_n\}_{n\geq 1}$ of positive numbers decreasing to $0$ as $n\to \infty$. 
Then, $\lambda u''(r_n)<0$ for every $n\geq 1$. Hence, we get a contradiction as $u''(r_n)$ would have a constant sign for all $n\geq 1$.  

Second, if $\lambda=0$, then the assumption $\ell>0$ implies that $\Theta\left(\Theta+2\rho\right)>0$. We see that 
\begin{equation} \label{azi1}  (r^{1-2\rho} u'(r))'=r^{\theta+1-2\rho} u^q(r) |u'(r)|^m\quad \mbox{for every } r\in (0,R).
\end{equation}
Hence, $r^{1-2\rho} u'(r)$ is non-decreasing on $(0,R)$ and $\lim_{r\to 0^+} r^{1-2\rho} u'(r)=L_0\in [-\infty,\infty)$. 

If $\Theta>0$ (resp., $\Theta<0$), then $\Theta>-2\rho$ (resp., $\Theta<-2\rho$), which implies that $L_0=-\infty$ (resp., $L_0=0$). 
(Otherwise, we would get that $\lim_{r\to 0^+} r^{1+\Theta} u'(r)\not=-M_0 \Theta$.) Hence, since $\Theta\not=0$,   
it follows that $\Theta \,u'(r)<0$ for every $r>0$ small.	
\end{proof}

For every $r\in (0,R)$, we define $\mathfrak B_1(r)$ and 
 $\mathfrak B_2(r)$ as follows
 \begin{equation} \label{jim11} \mathfrak B_1(r)=r^{\Theta}\,u(r)\quad \mbox{and}\quad
\mathfrak B_2(r)=\frac{ru'(r)}{u(r)}+\Theta=\frac{r\mathfrak B_1'(r)}{\mathfrak B_1(r)}. \end{equation}

\begin{lemma} \label{wisp1}
There exists $r_0>0$ small such that the functions $\mathfrak{B}_1'(r) $, $\mathfrak B_2(r)$ and $\mathfrak B_2'(r) $ don't vanish on $(0,r_0)$, having the same sign on that interval. Moreover, $\lim_{r\to 0^+} \mathfrak B_2(r)=0$.
\end{lemma}

\begin{proof}
Suppose by contradiction that there exists a sequence $\{r_n\}_{n\geq 1}$ in $(0,R)$ decreasing to $0$ as $n\to \infty$ such that $\mathfrak B_1'(r_n)=0$ for all $n\geq 1$. We see that
\begin{equation} \label{azi2} r_n^2  \mathfrak B_1''(r_n)= |\Theta|^m \mathfrak{B}_1(r_n) 
\left(  \mathfrak B_1^\kappa (r_n)-M_0^\kappa\right)
\quad \mbox{for all } n\geq 1.\end{equation}
Assuming that for some $n\geq 1$, we have $\mathfrak B_1(r_n)<M_0$ (resp., $\mathfrak B_1(r_n)>M_0$), we obtain that $\mathfrak B_1''(r_n)<0$
(resp., $\mathfrak B_1''(r_n)>0$), that is, $r_n$ is a local maximum point for $\mathfrak B_1$ (resp., a local minimum point for $\mathfrak B_1$). 
This is not possible. Hence, $\mathfrak B_1(r_n)$ would need to be $M_0$ for any $n\geq 1$ large and, moreover, $\mathfrak B_1\equiv M_0$ on some interval $(0,r_*)$. This is a contradiction with our assumption that $u\not\equiv U_0$ in 
any interval $(0,r_*)\subset (0,R)$. 

Hence, there exists 
$r_0\in (0,R)$ such that we have: 

(i) Either $\mathfrak B_1'(r)>0$ for every $r\in (0,r_0)$ so that $\mathfrak B_1>M_0$ and $\mathfrak B_2>0$ on $(0,r_0)$. 

(ii) Or $\mathfrak B_1'(r)<0$ for every $r\in (0,r_0)$, which implies that 
$\mathfrak B_1<M_0$ and $\mathfrak B_2<0$  on $(0,r_0)$. 

In the situation of (i) (resp., (ii)), we prove \eqref{lov} for a {\em positive} constant $\mu_0$ (resp., for a {\em negative} constant $\mu_0$). 

Similarly, we show that $\mathfrak B_2'(r)\not =0$ for $r>0$ small enough. For all $r\in (0,R)$, we have
\begin{equation} \label{oras} 
r \mathfrak B_2'(r)
=-\mathfrak B_2(r) \left( \mathfrak B_2(r)-2\rho-2\Theta\right) 
+\mathfrak B_1^\kappa(r) |\Theta-\mathfrak B_2(r)|^m-\ell. 
\end{equation}
Suppose by contradiction that there exists a sequence $\{r_n\}_{n\geq 1}$ in $(0,R)$ decreasing to $0$ as $n\to \infty$ such that $\mathfrak B_2'(r_n)=0$ for all $n\geq 1$. Then, for all $n\geq 1$, we find that
$$ r_n^2 \mathfrak B_2''(r_n)=\kappa r_n^{\theta+2} u^{q-1}(r_n) 
\,|u'(r_n)|^m \,\mathfrak B_2(r_n)=\kappa  \,\mathfrak B_2(r_n) \,\mathfrak B_1^\kappa(r_n) |\Theta-\mathfrak B_2(r_n)|^m.
$$
Using Lemma~\ref{wisp0} and that  $\mathfrak B_1'(r)\not=0$ on  $(0,r_0)$, we infer that 
$\mathfrak B_2''(r_n)$ has the same sign as 
$\mathfrak B_2(r_n)$ and $\mathfrak B_1'(r_n)$ on $(0,r_0)$. This is a contradiction. Hence, $\mathfrak B_2'(r)$ does not vanish for any $r>0$ small. By diminishing $r_0>0$ if necessary, we can assume that $\mathfrak B_2'(r)\not=0$ for every $r\in (0,r_0)$. 
So, there exists  $\lim_{r\to 0^+}\mathfrak B_2(r)=0$ in view of 
\eqref{uoo1} and L'H\^opital's rule.
\end{proof}

\begin{lemma} \label{wisp2}
For every $r\in (0,r_0)$, we define 
\begin{equation} \label{vin}
t=-\log  \mathfrak |\mathfrak B_2(r)| \quad \mbox{and}\quad 
X(t)+\xi_0 =-r\,\frac{dt}{dr}=\frac{r\mathfrak B_2'(r)}{\mathfrak B_2(r)}.
\end{equation} 
Then, $\frac{dX}{dt}\not=0$ for $t>0$ large enough, $t\to \infty$ as $r\to 0^+$ and $\lim_{t\to \infty} X(t)=0$.
\end{lemma}

\begin{proof} 
From Lemma~\ref{wisp1}, we have that $X(t)+\xi_0>0$ for every $t>0$ large  and $t\to \infty$ as $r\to 0^+$. 
We show that if $\lim_{t\to \infty} X(t)=\mathfrak L$, then $\mathfrak L=0$. 
Indeed, \eqref{oras} and \eqref{vin} imply that 
\begin{equation} \label{bat1} \lim_{r\to 0^+} \frac{\mathfrak B_1^\kappa(r) |\Theta-\mathfrak B_2(r)|^m-\ell}{\mathfrak B_2(r)}=\mathfrak L+\xi_0-2\rho-2\Theta.
\end{equation}
On the other hand, we obtain that 
$$ \frac{\frac{d}{dr} \left( \mathfrak B_1^\kappa(r) |\Theta-\mathfrak B_2(r)|^m\right)}{\mathfrak B_2'(r)} =
\mathfrak B_1^\kappa(r) |\Theta-\mathfrak B_2(r)|^m
 \left( \frac{\kappa\mathfrak B_2(r)}{r\mathfrak B_2'(r)}+\frac{m}{\mathfrak B_2(r)-\Theta}\right)\to \ell\left( \frac{\kappa}{\mathfrak L+\xi_0} -\frac{m}{\Theta}\right)
$$ as $r\to 0^+$. Then, by L'H\^opital's rule and \eqref{bat1}, it follows that 
\begin{equation} \label{bat2}   
\mathfrak L+\xi_0-2\rho-2\Theta
=\ell\left( \frac{\kappa}{\mathfrak L+\xi_0} -\frac{m}{\Theta}\right).
\end{equation}
Since $\mathfrak L+\xi_0\geq 0$ and $\xi_0$ is the positive root of \eqref{xiz1}, from \eqref{bat2}, we get that $\mathfrak L=0$. 

\vspace{0.2cm} 
We write $X'(t)$ for the derivative of $X$ with respect to $t$. 
To conclude the proof of Lemma~\ref{wisp2}, we need to establish that $X'(t)\not=0$ for every $t>0$ large. 
For convenience, we define
\begin{equation} \label{nota1}
\begin{aligned}
&\Xi_0=\xi_0-2\rho-2\Theta,\quad \alpha_0=m\ell+\Theta\left(\xi_0+\Xi_0\right) , \\
& \alpha_1=\Theta \xi_0\left(2+\frac{m\ell}{\Theta^2}-\left( \frac{\kappa}{\xi_0}-\frac{m}{\Theta}\right) \Xi_0\right),\quad 
 \beta_1=\left(m-1\right)\left(\xi_0+\Xi_0\right) +\Theta\left(2-\kappa\right),\\
& \alpha_2:=\left(m-2\right) \xi_0+\kappa \left(\Xi_0-\Theta\right). 
\end{aligned}\end{equation} 

By differentiating $X$ in \eqref{vin} with respect to $t$, we arrive at 
\begin{equation} \label{fib}
\begin{aligned}
e^t \left( X(t)+\xi_0\right) \left(\Theta-e^{-t} {\rm sgn}\,(\mathfrak B_2)\right) X'(t)=&
 e^t X(t)   \left[ \alpha_0 +\Theta X(t)\right]\\
 &+ 
 \left[  \alpha_1+ \beta_1 X(t) +(m-1) X^2(t)  \right] {\rm sgn}\, (\mathfrak B_2)  \\
 &+e^{-t} \[\alpha_2 +(m-2+\kappa) X(t) \right] +\kappa e^{-2t} {\rm sgn}\,(\mathfrak B_2)
 \end{aligned}
\end{equation}
for every $t>0$ large. It is important to remark that $\alpha_0$ in \eqref{nota1} is different from zero. Indeed, since $\xi=\xi_0$ is the positive root of \eqref{xiz1}, we have 
\begin{equation} \label{sie}  \frac{\ell m}{\Theta}+\Xi_0=\frac{\ell\kappa}{\xi_0}>0 , \end{equation}
which implies that $\alpha_0/\Theta>0$ and thus $\alpha_0\not=0$. 

Assume by contradiction that there exists an increasing sequence $\{ t_n\}_{n\geq 1}$ of positive numbers such that $t_n\to \infty$ as $n\to \infty$ and $X'(t_n)=0$ for all $n\geq 1$. 
By differentiating \eqref{fib} with respect to $t$ and taking $t=t_n$, we obtain that  
$$ \begin{aligned} e^{t_n} \left(X(t_n)+\xi_0\right) \left(\Theta -e^{-t_n} {\rm sgn}\, (\mathfrak B_2) \right) X''(t_n)=
& e^{t_n} X(t_n) [\alpha_0+\Theta X(t_n)] \\
&-e^{-t_n} [\alpha_2 +(m-2+\kappa) X(t_n)] -2\kappa e^{-2t_n} {\rm sgn}\, (\mathfrak B_2).
\end{aligned} $$
Since $X'(t_n)=0$ for all $n\geq 1$, we infer the following according to three different cases:
\begin{enumerate}
\item If $\alpha_1\not=0$, then $\lim_{n\to \infty} e^{t_n}X(t_n) = -\alpha_1 \,{\rm sgn}\, (\mathfrak B_2)/\alpha_0$ and 
$$\lim_{n\to \infty} e^{t_n}X''(t_n)=-\frac{\alpha_1\, {\rm sgn}\, (\mathfrak B_2)}{\Theta\,\xi_0}\not=0 .$$ 
\item If $\alpha_1=0$ and $\alpha_2\not=0$, then $\lim_{n\to \infty}e^{2t_n} X(t_n)=-\alpha_2/\alpha_0$ and $$\lim_{n\to \infty}e^{2t_n} X''(t_n)= -\frac{2\alpha_2}{\Theta\,\xi_0}\not=0.$$ 
\item If $\alpha_1=\alpha_2=0$, then $\lim_{n\to \infty} e^{3t_n} X(t_n)= -\kappa \,  {\rm sgn}\, (\mathfrak B_2)  /\alpha_0$ and 
$$  \lim_{n\to \infty}e^{3t_n} X''(t_n)= -\frac{3\kappa \, {\rm sgn}\, (\mathfrak B_2) }{\Theta\,\xi_0}\not=0.
$$
\end{enumerate}
In each of these cases, we see that for every $n\geq 1$ sufficiently large, $X''(t_n)$ has a constant sign, which is the same irrespective of the sequence $\{t_n\}_{n\geq 1}$ satisfying $\lim_{n\to \infty} t_n=\infty$ and $X'(t_n)=0$ for all $n\geq 1$. This is a contradiction, showing that $X'(t)\not=0$ for all $t>0$ large enough. In particular, there exists $\lim_{t\to \infty} X(t)$, which is zero from the earlier argument.  
\end{proof}

\begin{lemma}  \label{wisp4}
Let $\xi_0$ be the {\em positive} root of the quadratic equation in $\xi$ given in \eqref{xiz}, that is 
\begin{equation} \label{xiz1} 
\xi^2+\left(\frac{\ell\,m}{\Theta}-2\left(\rho+\Theta\right) \right) \xi-\ell\kappa=0. 
	\end{equation} 
Then, the following limits exist
\begin{eqnarray} 
& \displaystyle \lim_{r\to0^+} r^{-\xi_0} \mathfrak B_2(r)=\mu_0\in \mathbb R\setminus \{0\},
\label{soad1} \\ 
& \displaystyle  \frac{u(r)}{U_0(r)}=1+\frac{\mu_0}{\xi_0}\,r^{\xi_0}
		(1+o(1)) \mbox{  as } r\to 0^+. \label{lov}
\end{eqnarray}
\end{lemma}

\begin{proof} 
We set $Y(t)=t+\xi_0 \log r$ for $t>0$ large enough. We prove that 
there exists 
\begin{equation} \label{ret2}
\lim_{t\to \infty} Y(t)=y_0\in \mathbb R.
\end{equation}
Then, \eqref{soad1} holds with $\mu_0=e^{-y_0} {\rm sgn}\, (\mathfrak B_2)$.

\vspace{0.2cm}
{\em Proof of \eqref{ret2}.}
From \eqref{vin}, we find that
$$Y'(t)=1+\frac{\xi_0}{r}\,\frac{dr}{dt}=\frac{X(t)}{X(t)+\xi_0}.
$$
From Lemma~\ref{wisp2}, we have $\lim_{t\to \infty} X(t)=0$ and ${\rm sgn} (X(t)) \,{\rm sgn}\, (X'(t)) <0$ for large $t>0$. 
Since $\xi_0>0$, it follows that 
for $t>0$ large, we have
\begin{equation} \label{soad2} \mbox{either (a) } Y'(t)>0 \ (\mbox{when } X(t)>0) \quad \mbox{or (b) } Y'(t)<0 \ (\mbox{when } X(t)<0).\end{equation} 

In either case, there exists $\lim_{t\to \infty} Y(t)=y_0$ in $(-\infty,\infty]$ for (a) and in $[-\infty,\infty)$ for (b). In both these situations, we show that $y_0\in \mathbb R$. 
Assume first that (a) occurs, that is, $X'(t)<0$ and $X(t)>0$ for every $t>0$ large. Then, using \eqref{fib} and $\lim_{t\to \infty} X(t)=0$, jointly with $\alpha_0/\Theta>0$, we obtain that $\alpha_1 \,{\rm sgn} (\mathfrak B_2)/\Theta<0$ and there exists a constant $C_1>0$
such that $e^t X(t)<C_1$ for all $t>0$ large. Then, by choosing a constant $C>0 $ large (for example, $C>C_1/\xi_0$), the function $t\longmapsto Y(t)+ C e^{-t}$ is decreasing for every $t>0$ large. Consequently,
$y_0=\lim_{t\to \infty} Y(t)\not=\infty$ and, hence, 
 $y_0\in \mathbb R$ in case (a) of \eqref{soad2}. 
 
 A similar reasoning applies for (b) in \eqref{soad2} when $X'(t)>0$ and $X(t)<0$ for every $t>0$ large. Here, $\alpha_1 \,{\rm sgn} (\mathfrak B_2)/\Theta>0$ and for a constant $C_1>0$ large, we have $0<-e^t X(t)<C_1$ for all $t>0$ large. Then, by letting $C> 2C_1/\xi_0$, the function $t\longmapsto Y(t)- C e^{-t}$ is increasing for every $t>0$ large. Thus, we have $y_0\not=-\infty$ so that $y_0\in \mathbb R$ also in case (b). 
 This completes the proof of \eqref{ret2}, which validates \eqref{soad1}.
 
 Note that \eqref{soad1} and L'H\^opital's rule imply the the limit in \eqref{lov} since  
$$ \lim_{r\to 0^+} \frac{\mathfrak B_1(r)-M_0}{r^{\xi_0}}=\lim_{r\to 0^+} \frac{ \mathfrak B_1(r)\, \mathfrak B_2(r)}{ \xi_0 \,r^{\xi_0}}= \frac{M_0\,\mu_0}{\xi_0} .
$$
 This ends the proof of Lemma~\ref{wisp4}. 
\end{proof}

\begin{lemma}
\label{wisp5}	
Let $R>0$ be arbitrary. Then, equation \eqref{eq1} in $B_R(0)\setminus \{0\}$ has infinitely many positive radial solutions $u$ satisfying \eqref{uoo1}. 
\end{lemma}

\begin{proof} We remark that for $T_0>0$ large enough, the differential equation in \eqref{fib} on $[T_0,\infty)$ admits a solution $X(t)$ satisfying $\lim_{t\to \infty} X(t)=0$, 
	where ${\rm sgn} \, ({\mathfrak B_2})$ is understood as $\pm$. 
	
	For $r_0>0$ fixed arbitrary, we consider the differential equation for $r$ in \eqref{vin} for $t\in [T_0,\infty)$, subject to $r(T_0)=r_0$, namely,  
		\begin{equation} \label{compis} 
 \left\{ 	\begin{aligned} 
 & \frac{dr}{dt}=-\frac{r}{X(t)+\xi_0}\quad \mbox{for } t\geq T_0, \\
 & r(T_0)=r_0. 
\end{aligned}
\right.
	\end{equation}
This problem has a unique solution given by
	\begin{equation} \label{bilii} r(t)=r_0 \exp\left( -\int_{T_0}^ t \frac{ds}{X(s)+\xi_0}\right) \quad \mbox{for } t\geq T_0. 
	\end{equation}
	Since $r=r(t)$ is invertible, from \eqref{bilii}, we obtain $t=t(r)$. 	For every $r\in (0, r_0]$, we define
	\begin{equation} \label{ba15}
	\mathfrak B_2(r)=\pm e^{-t}\quad \mbox{and} \quad
        u(r)=r^{-\Theta}\left [\frac{\left [X(t)+\xi_0-2\rho-2\Theta\right ] {\rm sgn}\, (\mathfrak B_2) \,e^{-t}+e^{-2t}+ \ell }{|\Theta-e^{-t}{\rm sgn}\, (\mathfrak B_2) |^m}\right ]^{\frac{1}{\kappa}}.
	\end{equation}
% u(r)=\widetilde{u}(t)=\exp \left (-\int_t^\infty \frac{e^{-s}-\Theta}{X(s)+\xi_0}\,ds \right )\quad \mbox{for every } t\geq T_0.  
Then, $u$ is a positive radial solution of \eqref{eq1} in $B_{r_0}(0)\setminus \{0\}$ satisfying \eqref{uoo1}. For each $j\in(0,1)$, by applying the transformation in \eqref{scale}, that is,  $T_j[u](r)=j^{-\Theta}u(jr)$ for all $r\in (0, r_0/j)$, we get that
$T_j(u)$ is another positive radial solution of \eqref{eq1} in $B_{r_0/j}(0)\setminus\{0\}$ satisfying \eqref{uoo1}. 
%We have $T_{j_1}[u](r_0)<T_{j_2}[u](r_0)<u(r_0)$ for every $0<j_1<j_2<1$ since $r\longmapsto r^{-\Theta} u(r) $ is increasing on $(0,r_0]$. 
This ends the proof of Lemma~\ref{wisp5}. 	
\end{proof}

%%%%

%%%
\section{The asymptotic profile $\Phi_{\rho,\lambda}^+$ for $\lambda<\rho^2$, $\Theta>\Theta_+\not=0$ (Proof of Theorem \ref{prof1})}

Throughout this section, we let \eqref{cond1} hold, $\lambda< \rho^2$ and $\Theta>\Theta_+\not=0$.  Unless otherwise stated, $u$ is any positive radial solution of \eqref{eq1} in $B_R(0)\setminus \{0\}$ with $R>0$ such that \eqref{minaaa} holds, that is
\begin{equation}\label{zin}
\lim_{r\rightarrow 0^+}r^{\Theta_+}u(r)=\gamma \in (0,\infty).
\end{equation} 
For every $r\in (0,R)$, we define 
\begin{equation}\label{reti}
\mathfrak H(r):=r^{1+\Theta_-}u'(r)+\Theta_+ r^{\Theta_-} u(r). 
\end{equation}
Then, $\mathfrak H(r)$ is non-decreasing on $(0,R)$ since 
$$\mathfrak H'(r)=r^{\Theta_-+1+\theta }u^q(r)|u'(r)|^m\geq 0\quad \text{for all } r\in (0,R).$$
Hence, we have  $\lim_{r\to 0^+} \mathfrak H(r)=\mathfrak{a}_0\in [-\infty,\infty)$. 
By L'H\^opital's rule and \eqref{zin}, we obtain that 
\begin{equation} \label{frak22}
\lim_{r\to 0^+} r^{\Theta_-}\left[u(r)-\gamma \,r^{-\Theta_+}\right]=\frac{1}{2\sqrt{\rho^2-\lambda}}\lim_{r\to 0^+}\mathfrak H(r)=\frac{\mathfrak{a}_0}{2\sqrt{\rho^2-\lambda}}:=
\mathfrak a_1\in [-\infty,\infty).
\end{equation}

We distinguish three cases: 
\begin{equation} \label{chid}
\mbox{(I)}\ \chi>0,\quad 
\mbox{(II)} \ \chi<0\quad \mbox{and}\quad
\mbox{(III)}\  \chi=0,\quad \mbox{where }\chi:=\kappa\left( \Theta-\Theta_+\right)-2\sqrt{\rho^2-\lambda} .
\end{equation}

The proof of Theorem~\ref{prof1} follows from Lemmas \ref{dimi1}, \ref{dimi2} and \ref{exo2}. 

\begin{lemma}\label{dimi1} 
Let $\mathfrak a_1$ and $\chi$ be as in \eqref{frak22} and \eqref{chid}, respectively. For $\gamma\in (0,\infty)$, we define 
 $$\mathfrak a_2:=\frac{\gamma^{q+m}\left |\Theta_+\right |^m}{\kappa\left (\Theta-\Theta_+\right )}>0.$$

\begin{itemize}
	\item[(I)] If $\chi>0$, then $\mathfrak a_1\in \mathbb R$ and 
		\begin{equation} \label{refi}
	u(r)=\gamma\, r^{-\Theta_+}+\mathfrak{a}_1\,r^{-\Theta-}+(\mathfrak{a}_2/\chi)\,r^{\chi-\Theta_-}(1+o(1))\ \mbox{as } r\to 0^+. 
		\end{equation} 
	\item[(II)] If $\chi<0$, then 
	$\mathfrak{a}_1=-\infty$ and 
	\begin{equation} \label{ros2}	
	u(r)=\gamma\,r^{-\Theta_+}+(\mathfrak{a}_2/\chi)\,r^{\chi-\Theta_-}(1+o(1))\  \mbox{ as } r\to 0^+. 
		\end{equation}
		\item[(III)] If $\chi=0$, then $\mathfrak{a}_1=-\infty$ and, moreover, we have
	\begin{equation} \label{ros3}
	u(r)=\gamma\,r^{-\Theta_+}+\mathfrak{a}_2\,r^{-\Theta_-}\log r\,(1+o(1))\quad \mbox{as } r\to 0^+.
	\end{equation}
	\end{itemize}	
\end{lemma}

\begin{proof}
The function 
$$ \mathfrak B(r)= r^{\Theta_++1} u'(r)+\Theta_-r^{\Theta_+} u(r)$$ is non-decreasing for $r\in (0,R)$ since $\mathfrak B'(r)=r^{\Theta_++1+\theta} u^q(r)|u'(r)|^m\geq 0$. Hence, there exists $\lim_{r\to 0^+} \mathfrak B(r)\in [-\infty,\infty)$. Since $\Theta_+\not=0$, by L'H\^opital's rule and \eqref{zin}, we infer that 
\begin{equation} \label{pumm1}\lim_{r\to 0^+} r^{\Theta_++1}\,u'(r)=-\gamma\,\Theta_+.
\end{equation} 
Using \eqref{zin} and \eqref{pumm1}, we obtain that 
\begin{equation} \label{ros1}   \mathfrak H'(r)=r^{\Theta_-+1+\theta} u^q(r)|u'(r)|^m=
 \gamma^{q+m}\,|\Theta_+|^m  r^{\chi-1}(1+o(1))\quad \mbox{as } r\to 0^+.
\end{equation} 
Remark that if $\chi\leq 0$, then \eqref{ros1} yields that $\lim_{r\to 0^+} \mathfrak H(r)=\mathfrak{a}_0=-\infty$ and, hence, $\mathfrak a_1=-\infty$. 
We next consider separately each of the cases in \eqref{chid}.  

\vspace{0.2cm}
{\bf Case (I):} Let $\chi>0$.  
Since $\Theta_+\not=0$, we have 
$\lim_{r\to 0^+} \mathfrak H(r)=\mathfrak a_0\in \mathbb R$ and $\mathfrak{a}_1=\frac{\mathfrak a_0}{2\sqrt{\rho^2-\lambda}}\in \mathbb R$. 
From \eqref{ros1}, jointly with $\Theta_+-\Theta_-=2\sqrt{\rho^2-\lambda}$ and L'H\^opital's rule, we get
\begin{equation}\label{woam}
\lim_{r\rightarrow 0^+}  \frac{r^{\Theta_+} u(r)-\gamma -\mathfrak a_1 r^{\Theta_+-\Theta_-}}{r^{\kappa \left(\Theta-\Theta_+\right)}}
=\lim_{r\rightarrow 0^+}  \frac{\mathfrak H(r)-\mathfrak{a}_0}{\kappa \left( \Theta-\Theta_+\right) r^\chi}=\frac{\mathfrak a_2}{\chi},
\end{equation}
which proves the behaviour  of $u$ in \eqref{refi}.   

{\bf Case (II):} Let $\chi<0$. 
Similar to \eqref{woam}, we derive that 
\begin{equation} \label{rosi2}
\lim_{r\to 0^+} \frac{r^{\Theta_+}u(r)-\gamma}{r^{\kappa(\Theta-\Theta_+)}}=
 \lim_{r\to 0^+} \frac{\mathfrak H(r)}{\kappa\left(\Theta-\Theta_+\right) r^\chi}
=\frac{1}{\kappa \left(\Theta-\Theta_+\right)} \lim_{r\to 0^+} \frac{\mathfrak H'(r)}{\chi\, r^{\chi-1}}=\frac{\mathfrak a_2}{\chi}, \end{equation} 
which proves \eqref{ros2}. 

{\bf Case (III):} Let $\chi=0$, that is, $\kappa(\Theta-\Theta_+)=\Theta_+-\Theta_-$. 
Using \eqref{ros1}, we derive that 
\begin{equation*} 
 \lim_{r\to 0^+} \frac{r^{\Theta_+}u(r)-\gamma}{r^{\kappa(\Theta-\Theta_+)}\,\log r}=
\frac{1}{\kappa \left(\Theta-\Theta_+\right)}
\lim_{r\to 0^+}\frac{\mathfrak H(r)}{\log r}=
\frac{1}{\kappa \left(\Theta-\Theta_+\right)} \lim_{r\to 0^+} r \mathfrak H'(r)=\mathfrak a_2.
\end{equation*}
This proves the assertion in \eqref{ros3}, completing the proof of Lemma~\ref{dimi1}.  
\end{proof}

\begin{lemma}\label{dimi2}
For every $\gamma\in (0,\infty)$, there exists $R>0$ small such that \eqref{eq1} in $B_R(0)\setminus\{0\}$ has infinitely many positive radial solutions satisfying \eqref{zin}.
\end{lemma}

\begin{proof}
Let $u$ be any positive radial solution of \eqref{eq1} in $B_R(0)\setminus \{0\}$ with $R>0$ such that \eqref{zin} holds.
Fix $\beta_2\in (0,2\sqrt{\rho^2-\lambda})$ close enough to $2\sqrt{\rho^2-\lambda}$ such that $2\sqrt{\rho^2-\lambda}-\beta_2<\kappa\left(\Theta-\Theta_+\right)$. Using $\chi$ defined in \eqref{chid}, these conditions on $\beta_2$ are equivalent to 
	\begin{equation} \label{zina}
\max\,\{0,-\chi\}<\beta_2<2\sqrt{\rho^2-\lambda}.
	\end{equation}
Then, we choose $\beta_1>0$ and $\beta_3>0$ small enough such that 
	\begin{equation} \label{zina2}
 0<\beta_1<2\sqrt{\rho^2-\lambda}-\beta_2\quad \mbox{and}\quad 
 0<\beta_3<\min\,\{\beta_1, 2\sqrt{\rho^2-\lambda}-\beta_2-\beta_1,\beta_2+\chi \}.
\end{equation}
We next fix $\varepsilon\in (0,1)$ small such that 
$$ 0<\varepsilon<\gamma/2\quad \mbox{and}\quad 
\left(1+|\Theta_+|\right)\varepsilon<|\Theta_+|\gamma.  $$
The assumption that $\Theta_+\not=0$ is used here essentially to make possible the choice of $\varepsilon$.  

	For every $\xi=(\xi_1,\xi_2,\xi_3)\in \mathbb R^3$ satisfying $|\xi_j|<\varepsilon$ for $j\in \{1,2,3\}$, we define 
$$ 
g(\xi)=\left(\gamma+\xi_1\,|\xi_3|^{\frac{\beta_1}{\beta_3}}\right)^q\,\left|\xi_2\,|\xi_3|^\frac{2\sqrt{\rho^2-\lambda}-\beta_2}{\beta_3}-\Theta_+\left(
\gamma+\xi_1\,|\xi_3|^{\frac{\beta_1}{\beta_3}}\right) \right|^m
$$ 
and introduce 
$\mathbf F(\xi)=(F_1(\xi),F_2(\xi),0)$ with $F_1(\xi)$ and $F_2(\xi)$ given by
\begin{equation} \label{f12def} F_1(\xi)=-\xi_2 \,|\xi_3|^\frac{2\sqrt{\rho^2-\lambda}-\beta_1-\beta_2}{\beta_3},
\quad 
F_2(\xi)=-|\xi_3|^{\frac{\beta_2+\chi}{\beta_3}} \,g(\xi).  
\end{equation}
Remark that $\mathbf{F}$ is a $C^1$-function on the open ball $B_\varepsilon (\mathbf 0)\subset \mathbb R^3$. Using our choice of $\beta_1,\beta_2$ and $\beta_3$ in \eqref{zina} and \eqref{zina2}, respectively, we find that 
\begin{equation} \label{zero1} \mathbf{F}(\mathbf 0)=\mathbf 0\quad \mbox{and}\quad  D\mathbf{F}(\mathbf 0)=0.\end{equation}

Since $\beta_2>-\chi$ and $2\sqrt{\rho^2-\lambda}-\beta_1>\max\,\{0,-\chi\}$ (see \eqref{zina}), by Lemma~\ref{dimi1}, we have $$
\lim_{r\to 0^+}  r^{-\beta_1}\left(r^{\Theta_+} u(r)-\gamma \right)=0\quad \mbox{and}\quad 
\lim_{r\to 0^+} r^{\beta_2} |\mathfrak{H}(r)|=0
$$
irrespective of the case for $\chi$ in \eqref{chid}, where $\mathfrak H$ is defined in \eqref{reti}.  
By diminishing $R>0$, we can assume that $R<\varepsilon^{1/\beta_3}$ and the following hold
\begin{equation} \label{ins}
\left|r^{-\beta_1}\left(r^{\Theta_+} u(r)-\gamma \right) \right|<\varepsilon\quad \mbox{and}\quad 
r^{\beta_2} \left|\mathfrak{H}(r)\right|<\varepsilon\quad \mbox{for every } r\in (0,R].
\end{equation}
 	For every $r\in (0, R]$, we set $t=\log \,(R/r)$ and define $\mathbf X(t)=(X_1(t), X_2(t), X_3(t))$ as follows
 \begin{equation}\label{soc}
 X_1(t)=r^{-\beta_1}\left( r^{\Theta_+}u(r)-\gamma\right),
 \quad X_2(t)=r^{\beta_2} \mathfrak{H}(r),
 \quad  \mbox{and}\quad X_3(t)=r^{\beta_3}.
 \end{equation}
  Let $A=\mbox{diag}\,[\beta_1, -\beta_2, -\beta_3]$ be the diagonal matrix with the 
elements on the main diagonal being $\beta_1$, $-\beta_2$ and $-\beta_3$. Then, $\mathbf X(t)$ solves for $t\geq 0$ the first order nonlinear differential system
 \begin{equation}
 \label{ram1}
 \mathbf X'(t)=A\,\mathbf X(t)+\mathbf F(\mathbf X(t)),
 \end{equation}
 where $\mathbf F(\mathbf X)=(F_1(\mathbf X),F_2(\mathbf X),0)$ with $F_1$ and $F_2$ given by \eqref{f12def}. 
  
 Remark that the eigenvalues of $A$ are $\mu_1=\beta_1>0$, $\mu_2=-\beta_2<0$ and $\mu_3=-\beta_3<0$. Hence, 
 in light of \eqref{zero1}, the origin $\mathbf 0\in \mathbb R^3$ is a saddle critical point of \eqref{ram1}, the behaviour of which near $\mathbf 0$ 
 is approximated by the behavior of its linearization $\mathbf{X'}=A\mathbf{X}$ at $\mathbf{X}=\mathbf 0$. 
 Let $\phi_t$ be the flow of the nonlinear system \eqref{ram1}. The Stable Manifold Theorem applied to  \eqref{ram1} gives the existence of a (local) stable manifold $S$ that is two-dimensional. 
 More precisely, 
 there exist $\varepsilon_0\in (0,\varepsilon)$ small and a differentiable function $w:(-\varepsilon_0,\varepsilon_0)\times (-\varepsilon_0,\varepsilon_0)\to (-\varepsilon_0,\varepsilon_0)$ such that
 the equation 
 $X_1=w(X_2,X_3)$ defines a two-dimensional stable differentiable manifold $S$.  
For all $t\geq 0$, we have $\phi_t(S)\subseteq S$ and for all $\mathbf{x}_0 \in S$, $\lim_{t\rightarrow \infty} \phi_t(\mathbf{x}_0)=0$. 
 Therefore,  
 for every $x_{02},x_{03} \in (-\varepsilon_0, \varepsilon_0)$, the system \eqref{ram1} subject to the initial condition
 \begin{equation}\label{mari2}
 \mathbf{X}(0)=(w(x_{02},x_{03}), x_{02},x_{03})
 \end{equation}
 has a solution $\mathbf{X}$ on $[0,\infty)$ satisfying 
 \begin{equation} \label{iot1} \lim_{t\to \infty} \mathbf X(t)=\mathbf 0. 
 \end{equation}
  Since $X_3'(t)=-\beta_3 X_3(t)$ for $t\geq 0$, we have 
 $X_3(t)=x_{03}\,e^{-\beta_3 t}$ for all $t\geq 0$. 
  As we aim to recover \eqref{soc}, we will choose $x_{03}\in (0,\varepsilon_0)$ to obtain that $X_3(t)>0$ for all $t\geq 0$. 
  
  \vspace{0.2cm}
  To conclude the proof of Lemma~\ref{dimi2},  we
  fix $R\in (0,\varepsilon_0^{1/\beta_3})$. We choose $x_{03}=R^{\beta_3}\in (0,\varepsilon_0)$
and let $x_{02}\in (-\varepsilon_0,\varepsilon_0)$ be arbitrary. We set $r=R\,e^{-t}$ for $t\geq 0$ so that $X_3(t)=r^{\beta_3}$ and define 
$$ u(r):=\left(\gamma+X_1(t)\,(X_3(t))^{\frac{\beta_1}{\beta_3}}\right)(X_3(t))^{-\frac{\Theta_+}{\beta_3}}=
\left(\gamma+r^{\beta_1} X_1(t)\right) r^{-\Theta_+}.  
$$
Using \eqref{iot1}, 
we recover \eqref{soc} and that $u$ is a positive radial solution of \eqref{eq1} in $B_R(0)\setminus \{0\}$ satisfying \eqref{zin}. 
We find infinitely many such positive solutions
by varying $x_{02}\in (-\varepsilon_0,\varepsilon_0)$. 
\end{proof}

	In view of Lemma~\ref{dimi1}, our Lemma~\ref{dimi2} gives the construction of infinitely many positive solutions of \eqref{eq1} satisfying, on the one hand, \eqref{ros2} when $\chi<0$ and, on the other hand, \eqref{ros3} when $\chi=0$. See also Remark~\ref{rain} when $\chi<0$.

When $\chi>0$, we refine the conclusion of Lemma~\ref{dimi2}.
More exactly,  
given $\gamma\in (0,\infty)$ and any $\mathfrak{a}_1\in \mathbb R$, it is natural to ask whether 
for some $R>0$, there exists a positive radial solution $u$ of \eqref{eq1} in $B_R(0)\setminus \{0\}$ satisfying \eqref{refi}. 
We next give an affirmative answer to this question. 

\begin{lemma} \label{exo2} 
	Let $\chi>0$. Then, for every $\gamma\in (0,\infty)$ and each $\mathfrak{a}_1\in \mathbb R$, there exists 
				$R>0$ such that \eqref{eq1} in $B_R(0)\setminus \{0\}$ has a positive radial solution $u$ satisfying \eqref{refi}. 	 
\end{lemma}

\begin{proof} 
Let $\gamma\in (0,\infty)$ and $\mathfrak a_1\in \mathbb R$ be arbitrary. First, we assume that $u$ is a positive radial solution of \eqref{eq1} in $B_R(0)\setminus \{0\}$ for $R>0$ such that \eqref{refi} holds. Then, we also have \eqref{woam}. Set
$$ \mathfrak a_3:=\frac{\gamma^{q+m} |\Theta_+|^m}{\chi}.$$

(a) Let $\mathfrak a_1\not=0$. 
Fix $\varepsilon\in (0,1)$ small such that 
$\varepsilon^2<\gamma$ and $
\varepsilon^2 (|\Theta_+|+\mathfrak a_3+\varepsilon)+\varepsilon |\mathfrak{a}_0| <\gamma\,|\Theta_+|$. 
We choose $\eta_1>0$ and 
$\eta_3>0$ small such that 
\begin{equation} \label{etar} 0<\eta_1<2\sqrt{\rho^2-\lambda}\quad \mbox{and}\quad
0<\eta_3<\min\left\{\chi, \eta_1,2\sqrt{\rho^2-\lambda}-\eta_1 \right\}. 
\end{equation}
Then, using \eqref{refi} and \eqref{woam}, we can find 
$r_0\in (0,R)$ small satisfying $r_0<\varepsilon^{1/\eta_3}$ such that 
\begin{equation} \label{iob}
\left|r^{-\eta_1} \left(r^{\Theta_+}u(r)-\gamma \right)\right|<\varepsilon\quad \mbox{and}\quad
\left| r^{-\chi}\left(\mathfrak{H}(r)-\mathfrak{a}_0\right) -\mathfrak{a}_3\right|<\varepsilon\quad \mbox{for all } r\in (0,r_0]. 
\end{equation}
For the definition of $\mathfrak{H}$, we refer to \eqref{reti}. 
For every $\xi=(\xi_1,\xi_2,\xi_3)\in \mathbb R^3$ with $\max_{1\leq j\leq 3} |\xi_j|<\varepsilon$, we define $F_1(\xi)$ and $F_2(\xi)$ as follows
\begin{equation} \label{key1}
\begin{aligned}
& F_1(\xi)=-\left[\left(\xi_2+\mathfrak{a}_3\right)|\xi_3|^{\frac{\chi}{\eta_3}}+\mathfrak{a}_0 \right]|\xi_3|^{\frac{2\sqrt{\rho^2-\lambda}-\eta_1}{\eta_3}},\\
& F_2(\xi)=\chi\mathfrak{a}_3-
\left( \gamma+\xi_1 |\xi_3|^\frac{\eta_1}{\eta_3}\right)^q
\left| -F_1(\xi)\,|\xi_3|^{\frac{\eta_1}{\eta_3}} 
-\Theta_+\left(\gamma+\xi_1 |\xi_3|^\frac{\eta_1}{\eta_3} \right) \right|^m
\end{aligned}
\end{equation}
and we write $\mathbf F(\xi)=(F_1(\xi),F_2(\xi),0)$.
Thus, 
using \eqref{etar}, we get that $\mathbf F$ is a $C^1$-function on the open ball $B_\varepsilon(\mathbf 0)\subset \mathbb R^3$ such that  
$\mathbf F(\mathbf 0)=\mathbf 0$ and $D\mathbf F(\mathbf 0)=0$.

For every
$r\in (0,r_0]$, we set 
$ t=\log \,(r_0/r)$
and define 
$\mathbf X(t)=(X_1(t),X_2(t), X_3(t))$ by 
\begin{equation} \label{trogi}
X_1(t)=r^{-\eta_1}\left( r^{\Theta_+}u(r)-\gamma \right),\quad X_2(t)=r^{-\chi}\left(\mathfrak{H}(r)-\mathfrak{a}_0\right) -\mathfrak{a}_3,\quad \mbox{and}\quad X_3(t)=r^{\eta_3}.
\end{equation}
It follows that $\mathbf{X}(t)$ solves for $t\geq 0$ 
the following system
\begin{equation}\label{sas1}
\mathbf{X}'(t)=A\mathbf{X}+\mathbf{F}(\mathbf{X}),
\end{equation}
where $A=\mbox{diag}\,[\eta_1, \chi,-\eta_3]$ is a diagonal matrix with the entries on the main diagonal given by $\eta_1$, $ \chi$ and $-\eta_3$. Since $\eta_1,\chi>0$ and $-\eta_3<0$, we obtain that $\mathbf 0$ is a saddle critical point for \eqref{sas1}, the behaviour of which near zero is approximated by that of 
its linearization $\mathbf{X'}=A\mathbf{X}$ at $\mathbf{X}=\mathbf 0$.

Let $\phi_t$ be the flow of the nonlinear system \eqref{sas1}. By the Stable Manifold Theorem, 
there exist $\varepsilon_0\in (0,\varepsilon)$ small and two differentiable functions $w_1,w_2:(-\varepsilon_0,\varepsilon_0)\rightarrow (-\varepsilon_0,\varepsilon_0)$ such that the stable manifold $S$ for \eqref{sas1} is 
defined by  
$X_1=w_1(X_3)$ and $X_2=w_2(X_3)$ 
such that for all $t\geq 0$, we have $\phi_t(S)\subseteq S$ and for all $\mathbf{x}_0 \in S$, $\lim_{t\rightarrow \infty} \phi_t(\mathbf{x}_0)=0$. 
Therefore,  
for every $x_{03}\in (-\varepsilon_0, \varepsilon_0)$, the system \eqref{sas1} subject to the initial condition
\begin{equation}\label{mib2}
\mathbf{X}(0)=(w_1(x_{03}), w_2(x_{03}),x_{03})
\end{equation}
has a solution $\mathbf{X}$ on $[0,\infty)$ satisfying \begin{equation} \label{mib3} \lim_{t\to \infty} \mathbf X(t)=\mathbf 0.  
\end{equation}
Since $X_3'(t)=-\eta_3 X_3(t)$, we get that $X_3(t)=x_{03}\,e^{-\eta_3 t}$ for all $t\geq 0$.

\vspace{0.2cm}
Let $0<r_0<\varepsilon_0^{1/\eta_3}$ be arbitrary.
We next construct a positive radial solution $u$ of \eqref{eq1} in $B_{r_0}(0)\setminus \{0\}$ such that \eqref{refi} and \eqref{woam} hold. To this aim, let $\mathbf{X}$ be the above solution of \eqref{sas1}, subject to \eqref{mib2}, corresponding to $x_{03}=r_0^{\eta_3}\in (0,\varepsilon_0)$.  
We set $r=r_0\,e^{-t}$ for $t\geq 0$ and define 
\begin{equation} \label{miv}
u(r)= \left(X_1(t)(X_3(t))^{\frac{\eta_1}{\eta_3}}+\gamma\right)(X_3(t))^{-
	\frac{\Theta_+}{\eta_3}}\quad \mbox{for every } r\in (0,r_0].
\end{equation}
Since $X_3(t)=r^{\eta_3}$, by differentiating \eqref{miv} with respect to $t$ and using \eqref{sas1}, we regain \eqref{trogi}. Then, 
$u$ is a positive radial solution of \eqref{eq1} in $B_{r_0}(0)\setminus \{0\}$ such that \eqref{refi} and \eqref{woam} hold.

\vspace{0.2cm}
(b) Let $\mathfrak a_1=0$.  
We fix $\varepsilon\in (0,1)$ small such that 
$ \varepsilon^2<\gamma/2$ and $
\varepsilon^2 (|\Theta_+|+\mathfrak a_3+\varepsilon) <\gamma\,|\Theta_+|$.   

By our assumptions, we can choose $\eta>0$ small satisfying
$ 0<\eta< \kappa(\Theta-\Theta_+)/2 $.  
Recall that $u$ satisfies \eqref{woam} with $\mathfrak a_0=\mathfrak a_1=0$. Then, let  
$r_0\in (0,R)$ be small satisfying $r_0<\varepsilon^{1/\eta}$ such that 
\begin{equation} \label{ineo}
\left|r^{-\eta} \left(r^{\Theta_+}u(r)-\gamma \right)\right|<\varepsilon\quad \mbox{and}\quad
\left| r^{-\chi}\,\mathfrak{H}(r) -\mathfrak{a}_3\right|<\varepsilon
\end{equation}
for every $r\in (0,r_0]$, where $\mathfrak a_3=\gamma^{q+m}\,|\Theta_+|^m/\chi$. 

For every $\xi=(\xi_1,\xi_2,\xi_3)\in \mathbb R^3$ with $\max_{1\leq j\leq 3} |\xi_j|<\varepsilon$, we define $\mathbf F(\xi)=(F_1(\xi),F_2(\xi),0)$ with $F_1(\xi)$ and $F_2(\xi)$ given by 
\begin{equation} \label{defx}
\begin{aligned} 
&F_1(\xi)=
-\left(\xi_2+\mathfrak a_3\right)|\xi_3|^{\frac{\kappa (\Theta-\Theta_+)}{\eta}-1},\\
& F_2(\xi)=\chi\mathfrak a_3-(\gamma+\xi_1\xi_3)^q \left|(\xi_2+\mathfrak a_3)|\xi_3|^\frac{\kappa(\Theta-\Theta_+)}{\eta}-\Theta_+(\gamma+\xi_1\xi_3)\right|^m.
\end{aligned} \end{equation}

For every
$r\in (0,r_0]$, we set 
$ t=\log \,(r_0/r)$
and define 
$\mathbf X(t)=(X_1(t),X_2(t), X_3(t))$ as follows 
\begin{equation} \label{trog}
X_1(t)=r^{-\eta}\left( r^{\Theta_+}u(r)-\gamma \right),\quad X_2(t)=r^{-\chi}\,\mathfrak{H}(r) -\mathfrak{a}_3,\quad \mbox{and}\quad X_3(t)=r^\eta.
\end{equation}
We have $ \lim_{t\to \infty}\mathbf X(t)=\mathbf 0$ using \eqref{ros2} and our choice of $\eta$. 
Then, $\mathbf{X}(t)$ solves  
\begin{equation}\label{sis}
\mathbf{X}'(t)=A\mathbf{X}+\mathbf{F}(\mathbf{X})\quad \mbox{for  } t\geq 0,
\end{equation}
where $A=\mbox{diag}\,[\eta, \chi,-\eta]$ is a diagonal matrix with the entries on the main diagonal given by $\eta$, $\chi$ and $-\eta$. 
Now, $\mathbf F$ is a $C^1$-function on the open ball $B_\varepsilon(\mathbf 0)\subset \mathbb R^3$ satisfying 
$\mathbf F(\mathbf 0)=\mathbf 0$ and $D\mathbf F(\mathbf 0)=0$. Since $A$ has two positive eigenvalues ($\eta$ and $\chi$) and one negative eigenvalue ($-\eta$), it follows that $\mathbf 0$ is a saddle critical point for the system \eqref{sis}. We can apply the Stable Manifold Theorem to \eqref{sis}.  
Let $\phi_t$ be the flow of the nonlinear system \eqref{sis} in relation to which we obtain the existence of a (local) stable manifold $S$ of dimension one. As in part (a) above,  
there exist $\varepsilon_0\in (0,\varepsilon)$ small and two differentiable functions $w_1,w_2:(-\varepsilon_0,\varepsilon_0)\rightarrow (-\varepsilon_0,\varepsilon_0)$ such that the stable manifold $S$ is 
defined by  
$X_1=w_1(X_3)$ and $X_2=w_2(X_3)$ 
such that for all $t\geq 0$, we have $\phi_t(S)\subseteq S$ and for all $\mathbf{x}_0 \in S$, $\lim_{t\rightarrow \infty} \phi_t(\mathbf{x}_0)=0$. 

Hence,  
for every $x_{03}\in (-\varepsilon_0, \varepsilon_0)$, the system \eqref{sis} subject to the initial 
condition \eqref{mib2} 
has a solution $\mathbf{X}$ on $[0,\infty)$ satisfying 
\eqref{mib3}. 
Moreover, 
$X_3(t)=x_{03}\,e^{-\eta t}$ for all $t\geq 0$.

Let $0<r_0<\varepsilon_0^{1/\eta}$ be arbitrary. 	
Let $\mathbf{X}$ be the above solution of \eqref{sis}, subject to \eqref{mib2}, corresponding to $x_{03}=r_0^\eta$. 
By setting $r=r_0\,e^{-t}$ for $t\geq 0$, we get $X_3(t)=r^\eta$. We define
\begin{equation} \label{micc}
u(r)= \left(X_1(t)X_3(t)+\gamma\right)(X_3(t))^{-
	\frac{\Theta_+}{\eta}}\quad \mbox{for every } r\in (0,r_0].
\end{equation}
By differentiating \eqref{micc} with respect to $t$, we recover the expression of $X_2$ in \eqref{trog}. Hence, using the differential equation for $X_2$ in \eqref{sis}, we conclude that $u$ is a positive  radial solution of \eqref{eq1} in $B_{r_0}(0)\setminus \{0\}$ satisfying \eqref{refi} thanks to \eqref{mib3}. This completes the proof of Lemma~\ref{exo2}. 
\end{proof}

\begin{rem} \label{rain} When $\chi<0$, the same approach as in the proof of Lemma~\ref{exo2} (with $\mathfrak a_1=0$)  yields directly that for every $\gamma\in (0,\infty)$, there exists $r_0>0$ such that for all $R\in (0,r_0)$, equation \eqref{eq1} in 
$B_R (0) \setminus \{0\}$ 
has infinitely many positive radial solutions satisfying \eqref{ros2}.  
	Indeed, the transformation in \eqref{trog} leads to the system in \eqref{sis}, which has $\mathbf 0$ as a saddle critical point. However, in this case the matrix $A$ has {\em two} negative eigenvalues and only one positive eigenvalue. Thus, when describing the behaviour of \eqref{sis} near $\mathbf 0$, the local stable manifold $S$ has dimension two. So, there exist $\varepsilon_0\in (0,\varepsilon)$ small and a differentiable function $w:(-\varepsilon_0,\varepsilon_0)\times (-\varepsilon_0,\varepsilon_0)\to (-\varepsilon_0,\varepsilon_0)$ such that
	$S$ is defined by the equation 
	$X_1=w(X_2,X_3)$. Then, for every $x_{02},x_{03} \in (-\varepsilon_0, \varepsilon_0)$, the system \eqref{sis}, subject to \eqref{mari2}, 
	has a solution $\mathbf{X}$ on $[0,\infty)$ satisfying \eqref{iot1}. 
	We
	fix $r_0\in (0,\varepsilon_0^{1/\eta})$. We choose $x_{03}=r_0^{\eta}$
	and let $x_{02}\in (-\varepsilon_0,\varepsilon_0)$ be arbitrary. We set $r=r_0\,e^{-t}$ for $t\geq 0$ and define $u$ as in \eqref{micc}, which 
	a positive radial solution of 
	\eqref{eq1} in $B_{r_0}(0)\setminus \{0\}$ satisfying \eqref{ros2}. We obtain 
infinitely many such solutions by varying $x_{02}\in (-\varepsilon_0,\varepsilon_0)$. 
\end{rem}

%%%

\section{The asymptotic profile $\Phi_{\rho,\lambda}^+$ for $\lambda=\rho^2$ and $\Theta>\Theta_{\pm}$ (Proof of Theorem \ref{th81})}

Throughout this section, we let \eqref{cond1} hold, $\rho\in \mathbb{R}, \lambda=\rho^2$ and $\Theta>\Theta_\pm=-\rho$. 

\begin{lemma}\label{biq1} Let $u$ be any positive radial solution of \eqref{eq1} in $B_R(0)\setminus \{0\}$ for $R>0$ such that 
\begin{equation}\label{zob}
\lim_{r\rightarrow 0^+}\frac{u(r)}{\Phi_{\rho,\lambda}^+(r)}=\lim_{r\rightarrow 0^+}\frac{r^{-\rho} \, u(r)}{\log \,(1/r)}=\gamma\in (0,\infty).
\end{equation} 
Then, by defining $\vartheta=\kappa \left(\Theta+\rho\right)>0$, there exists $\mathfrak{c}_0\in \mathbb R$ such that as $r\to 0^+$, we have
\begin{equation} \label{wimn}
u(r)=\left\{ \begin{aligned}
&\gamma \,r^{\rho}\log \,(1/r)+\mathfrak{c}_0\,r^\rho+\frac{\gamma^{q+m} |\rho|^m}{\vartheta^2} \,r^{\vartheta+\rho}\left|\log r\right|^{m+q}(1+o(1))&& \mbox{if } \rho\not=0,&	\\	
& \gamma\,\log \, (1/r)+\mathfrak{c}_0+
\frac{\gamma^{q+m}}{\vartheta^2} \,r^{\vartheta}\left| \log r\right|^q (1+o(1)) && \mbox{if } \rho=0.& 
\end{aligned} \right.
\end{equation}
\end{lemma}

\begin{proof}
We define $\mathfrak{H}$ as in \eqref{reti}, that is, $\mathfrak{H}(r):=r\left(r^{-\rho}\,u(r) \right)'$ for every $r\in (0,R)$ (since
$\Theta_\pm=-\rho$). As before, $\mathfrak H$ is non-decreasing on $(0,R)$.   
From \eqref{zob} and L'H\^opital's rule, we have
\begin{equation}\label{bun1} \lim_{r\to 0^+} \mathfrak H(r)=\lim_{r\to 0^+} \frac{r^{-\rho} u(r)}{\log r}=
-\gamma,
\end{equation}
which implies that 
 \begin{equation} \label{llim} \lim_{r\to 0^+} \frac{r^{1-\rho}\,u'(r)}{\log \,(1/r)}=\gamma\rho \ \mbox{ if } \rho\not=0,\quad 
\mbox{while } \lim_{r\to 0^+} ru'(r)=-\gamma \ \mbox{ if } \rho=0.   
\end{equation}
For every $ r\in (0,\min\{1,R\}) $, we define 
$ \mathfrak T(r):=r^{-\rho} u(r)+ \gamma \log r$. 
Using \eqref{bun1} and \eqref{llim}, we obtain that  $r\mathfrak T'(r)=\mathfrak H(r)+\gamma\to 0$ as $r\to 0^+$ and 
\begin{equation*} %\label{bun0}
\left(r\mathfrak T'(r)\right)'=\mathfrak H'(r)=r^{-\rho+1+\theta} u^q(r)|u'(r)|^m=
\left\{\begin{aligned}
	& \gamma^{q+m} \,|\rho|^m\, r^{\vartheta-1}\left|\log r\right|^{q+m}(1+o(1))
	&& \mbox{if } \rho\not=0,\\
	& \gamma^{q+m}\,
	r^{\vartheta-1}\left|\log r\right|^{q}(1+o(1))
		&& \mbox{if } \rho=0.
	\end{aligned} \right. 
	\end{equation*}
Hence, using  
L'H\^opital's rule, as $r\to 0^+$, we get 
\begin{equation} \label{bun3}  \vartheta \,\frac{\mathfrak T'(r)}{r^{\vartheta-1}}= 
\left\{\begin{aligned}
& \gamma^{q+m} |\rho|^m \left| \log r\right|^{q+m}(1+o(1))
&& \mbox{if } \rho\not=0,\\
&  \gamma^{q+m} \,
\left| \log r\right|^{q}(1+o(1))
&& \mbox{if } \rho=0.
\end{aligned} \right. 
\end{equation}
Since $\vartheta>0$, we conclude that there exists $\lim_{r\to 0^+} \mathfrak T(r):=\mathfrak c_0\in \mathbb R$ and \eqref{wimn} holds. 
 \end{proof}
 
 \begin{lemma}\label{biq2}
 For every $\gamma\in (0,\infty)$, there exists $R>0$ small such that \eqref{eq1} in $B_R(0)\setminus\{0\}$ has infinitely many positive radial solutions satisfying \eqref{zob}.
 \end{lemma}
 
\begin{proof}
 Let $u$ be a positive radial solution of \eqref{eq1} in $B_R(0)\setminus \{0\}$ for $R>0$ such that \eqref{zob} holds.
	We first fix $\eta_1>0$ small and then $\eta_3$ greater than (and close enough to) $-\rho$ such that
	\begin{equation} \label{chic}
	0<\eta_1<\vartheta \quad\mbox{and}\quad 0<\eta_3+\rho<\min\left\{ \eta_1, \frac{\vartheta-\eta_1}{m+q+1}\right\}.
	\end{equation}

	Let $\varepsilon\in (0,1)$ be small such that 
	$\varepsilon^2<\gamma/2$. When $\rho\not=0$, we also require that $\varepsilon(\gamma-\varepsilon^2)<|\rho|/2$.  
	Since $u$ satisfies \eqref{bun3} and \eqref{chic} holds, we can diminish $R>0$ to ensure that 
	\begin{equation*}
	\left|r^{-\eta_1}(\mathfrak H(r)+\gamma)\right|<\varepsilon, \quad \frac{r^{\rho}}{u(r)}<\varepsilon\quad \text{and}\quad  r^{\eta_3}u(r)<\varepsilon\quad \mbox{for every } 0<r\leq R.
	\end{equation*}
	
We set $t=\log \,(R/r)$ for $r\in (0,R]$ and define $\mathbf{X}(t)=(X_1(t),X_2(t),X_3(t))$ by
	\begin{equation}\label{s22}
	X_1(t)=r^{-\eta_1}\left(\mathfrak H(r)+\gamma\right),\quad X_2(t)=\frac{r^{\rho}}{u(r)}\quad\text{and}\quad  X_3(t)=r^{\eta_3}\,u(r).
	\end{equation}
		For all $\xi=(\xi_1,\xi_2,\xi_3)\in \mathbb{R}^3$ with $\max_{1\leq j\leq 3}|\xi_j|<\varepsilon$, we define $\mathbf{F}(\xi)=(F_1(\xi),F_2(\xi),F_3(\xi))$ by
	\begin{equation*}\begin{aligned}
	&F_1(\xi)=-|\xi_2|^{\frac{\vartheta-\eta_1}{
			\eta_3+\rho}-(m+q)}|\xi_3|^{\frac{\vartheta-\eta_1}{\eta_3+\rho}}
	\left |\rho-\xi_2\left(\gamma-\xi_1|\xi_2\,\xi_3|^{\frac{\eta_1}{\eta_3+\rho}}\right)\right |^m,\\
	&F_2(\xi)=\xi_2^2\left(-\gamma+\xi_1|\xi_2\,\xi_3|^{
		\frac{\eta_1}{\eta_3+\rho}}\right),\quad 
	F_3(\xi)=\gamma \xi_2\xi_3-\xi_1|\xi_2\xi_3|^{1+\frac{\eta_1}{\eta_3+\rho}}.
	\end{aligned}
	\end{equation*}	
	We obtain that 
	$\lim_{t\rightarrow \infty} \mathbf{X}(t)=0$ and $\mathbf{X}(t)$ solves for $t\geq 0$ the following system
	\begin{equation}\label{s22p}
	\mathbf{X}'(t)=A\mathbf{X}+\mathbf{F}(\mathbf{X}),
	\end{equation}
	where $A=\mbox{diag}\,[\eta_1,0,-(\rho+\eta_3)]$ is a diagonal matrix with the elements on the main diagonal $\eta_1$, $0$ and $-(\rho+\eta_3)$, respectively.
The function $\mathbf{F}$ is $C^1$ on the open ball $B_\varepsilon(\mathbf 0)\subset \mathbb R^3$ with $\mathbf{F}(\mathbf{0})=0$ and $D\mathbf{F}(\mathbf{0})=0$.
Then, $\mathbf 0$ is a critical point for the system \eqref{s22p}. However, we cannot apply the stable manifold theory here since $\mathbf 0$ is a non-hyperbolic critical point. Indeed, the matrix $A$ has a positive eigenvalue $\eta_1$, a null eigenvalue and a negative eigenvalue given by $-(\rho+\eta_3)$.  

Nevertheless, our choice of $\eta_1$ and $\eta_3$ in \eqref{chic} ensure that the hypotheses of Theorem 7.1 in the Appendix of \cite{CRV} are satisfied for the system \eqref{s22p}. Indeed, 
there exists a positive constant $C_1$ such that for all $\xi=(\xi_1,\xi_2,\xi_3)$ with $\max_{1\leq j\leq 3}|\xi_j|<\varepsilon$, we have
	\begin{equation}\label{hup22}
	\left \{
	\begin{aligned}
	&\sum_{j=1}^{3}|F_j(\xi)|\leq C_1\sum_{j=1}^{3}\xi_j^2\quad\text{and}\quad \sum_{j=1}^{3}|\nabla F_j(\xi)|\leq C_1\sum_{j=1}^{3}|\xi_j|\\
	&F_2(\xi)\leq -C_1|\xi_2|^2 \quad\text{and}\quad F_2(\xi_1,0,\xi_3)=0.
	\end{aligned}
	\right.
	\end{equation}

	Hence, there exist $\varepsilon_0\in (0, \varepsilon/2)$ and a Lipschitz function $w:[0,\varepsilon_0]\times[0,\varepsilon_0]\rightarrow [0,\varepsilon_0]$ such that for all $(x_{02},x_{03})\in [0,\varepsilon_0]\times[0,\varepsilon_0]$, the system \eqref{s22p} subject to the initial condition
	\begin{equation}\label{co22}
	\mathbf{X}(\mathbf{0})=(w(x_{02},x_{03}), x_{02},x_{03})
	\end{equation}
	has a solution $\mathbf{X}$ on $[0,\infty)$ satisfying $\lim_{t\rightarrow \infty}\mathbf{X}(t)=0$.
	Moreover, the parametrized surface $(X_2,X_3)\rightarrow (w(X_2, X_3), X_2, X_3)$ is stable in the sense that $X_1(t)=w(X_2(t),X_3(t))$ for all $t\geq 0$.
	
	We fix $R\in (0,\varepsilon_0^{2/(\eta_3+\rho)})$. 
	For $x_{02}\in (R^{\eta_3+\rho}/\varepsilon_0, \varepsilon_0)$ arbitrary, we choose $x_{03}$ such that
	\begin{equation*}
	x_{03}=\frac{R^{\eta_3+\rho}}{x_{02}}\in (0,\varepsilon_0).
	\end{equation*}
For these $x_{02}$ and $x_{03}$, the system \eqref{s22p} subject to \eqref{co22}, has a solution $\mathbf{X}(t)$ for all $t\geq 0$ satisfying $\lim_{t\to \infty}\mathbf{X}(t)=0$ and, moreover, $X_2(t)$ and $X_3(t)$ are positive for all $t\geq 0$. Using such a solution $\mathbf{X}$ and setting $r=R\,e^{-t}$, we define
	\begin{equation} \label{maxi}
	u(r)=\left [(X_2(t))^{-\eta_3}(X_3(t))^{\rho}\right ]^{\frac{1}{\eta_3+\rho}}\quad\text{for all } r\in (0,R].
	\end{equation}
	Observe that $X_2(0)X_3(0)=x_{02}x_{03}=R^{\eta_3+\rho}$ and $$ \frac{d}{dt}(X_2(t)X_3(t))=-(\eta_3+\rho)X_2(t)X_3(t)\quad \mbox{for }t\geq 0.$$ It follows that 
	$
	X_2(t)\,X_3(t)=R^{\eta_3+\rho}\,e^{-(\eta_3
		+\rho)t}=r^{\eta_3+\rho}$ for all $t\geq 0$.
	Using this fact in \eqref{maxi}, we obtain the expressions of $X_2$ and $X_3$ in \eqref{s22}. 
	 By differentiating $X_2(t)=r^\rho /u(r)$ with respect to $t$ and using the differential equation for $X_2$ in \eqref{s22p}, we obtain the expression of $X_1$ in \eqref{s22}. Then, by differentiating the expression of $X_1$ with respect to $t$ and using \eqref{s22p}, we conclude that $u$ is a positive radial  solution of \eqref{eq1} in $B_{R}(0)\setminus \{0\}$ satisfying  $\lim_{r\to 0^+} \mathfrak H(r)=-\gamma$ and $r^{-\rho}u(r)\to +\infty$ as $r\to 0^+$ since $\lim_{t\to \infty}\mathbf{X}(t)=0$. Hence,  l'H\^opital's rule yields \eqref{zob} for our solution. By varying $x_{02}\in (R^{\eta_3+\rho}/\varepsilon_0, \varepsilon_0)$, we get infinitely many positive solutions of \eqref{eq1} in $B_{R}(0)\setminus \{0\}$ with the desired asymptotics in \eqref{zob}.    
	This finishes the proof of Lemma \ref{biq2}.
\end{proof}

%%%%%

%%%%%%
\section{The asymptotic profile $\Phi_{\rho,\lambda}^-$ for $\lambda\leq \rho^2$ and $\Theta>\Theta_-\not=0$ (Proof of Theorem~\ref{mix1})}

Throughout this section, we let 
\eqref{cond1} hold, $\lambda\leq \rho^2$ and $\Theta>\Theta_-\not=0$. 
 We define $$ \vartheta=\kappa \left(\Theta-\Theta_-\right)>0.$$ 
 
The proof of Theorem~\ref{mix1} follows from Lemmas~\ref{kiop1} and \ref{viz1}. 

\begin{lemma} \label{kiop1} 
Let $u$ be any positive radial solution of \eqref{eq1} in $B_{R}(0)\setminus \{0\}$ for $R>0$ such that 
\begin{equation} \label{lawn2}
\lim_{r\to 0^+} r^{\Theta_-} \,u(r)=\gamma\in (0,\infty).
\end{equation}
Then, we have 
\begin{equation} \label{flo1}
 \lim_{r\to 0^+} \frac{r^{\Theta_-}u(r)-\gamma}{r^\vartheta}=
 \lim_{r\to 0^+} 
\frac{ r^{1+\Theta_-} u'(r)+\Theta_- r^{\Theta_-} u(r) }{ \vartheta\, r^\vartheta}=\frac{\gamma^{q+m} |\Theta_-|^m} { \vartheta \left( \vartheta+2\sqrt{\rho^2-\lambda} \right)} 
=:C_\gamma>0. 
\end{equation} 
\end{lemma}

\begin{proof}
We first remark that for every small $r>0$, we have $u'(r)\not=0$. Indeed, if 
$\lambda\not=0$, then this assertion follows as in the proof of Lemma~\ref{wisp0}.
In turn, if $\lambda=0$, then the assumption $\Theta_-\not=0$ implies that $\rho>0$ and $\Theta_-=-2\rho<0$. 
From \eqref{azi1}, $r^{1-2\rho} u'(r)$ is non-decreasing on $(0,R)$ and 
\begin{equation} \label{lola1} 
\lim_{r\to 0^+} r^{1-2\rho} u'(r)=L_0\in [-\infty,\infty).
\end{equation}   
Using \eqref{lawn2}, we get $ \lim_{r\to 0^+} u(r)=0 $ and
	$ \lim_{r\to 0^+}  r^{-2\rho} u(r)= L_0/(2\rho)\geq 0$. Hence, 
$u'(r)> 0$ for every $r\in (0,R)$. (If $u'(r_0)=0$ for some $r_0\in (0,R)$, then $u\equiv 0$ on $(0,r_0]$, which is impossible.)

\vspace{0.2cm}
For every $r\in (0,R)$, we define $\mathfrak H(r)$ as in \eqref{reti} and 
\begin{equation} \label{hsi} 
\quad \mathfrak G(r):= 
\frac{ r^{1+\Theta_-}u'(r)+\Theta_-r^{\Theta_-}u(r)}{ r^\vartheta}.
\end{equation} 
	Then, $\mathfrak H(r)$ and  $ \mathfrak D(r):= r^{\vartheta+2\sqrt{\rho^2-\lambda}}\,\mathfrak G(r)$ are non-decreasing on $(0,R)$
since 
\begin{equation} \label{dri} \mathfrak H'(r)=r^{\Theta_-+1+\theta} u^q(r)|u'(r)|^m \quad \mbox{and} \quad
\mathfrak{D}'(r)=r^{\Theta_++\theta+1}u^q(r)\,|u'(r)|^m  
\end{equation}
for all $ r\in (0,R)$.    
Recalling that $\Theta_-\not=0$ and \eqref{lawn2} holds, by L'H\^opital's rule, we get
\begin{equation} \label{forv}  \lim_{r\to 0^+} r^{1+\Theta_-}\,u'(r)=-\gamma\,\Theta_-.
\end{equation}
Thus, $\lim_{r\to 0^+} r^\vartheta \mathfrak G(r)=0$ and 
$\lim_{r\to 0^+} \mathfrak{D}(r)=0$. 
By \eqref{lawn2}, \eqref{dri} and \eqref{forv}, we find that
\begin{equation} \label{nove} \frac{\mathfrak D'(r)}{r^{\vartheta+2\sqrt{\rho^2-\lambda}-1}}=
 r^{\theta+2+\Theta_--\vartheta}u^q(r)|u'(r)|^m\to 
\gamma^{q+m}\, |\Theta_-|^m\quad \mbox{as } r\to 0^+.  
\end{equation}
Using  \eqref{nove} and L'H\^opital's rule, we get
$$  \mathfrak G(r)=\frac{\mathfrak D(r)}{r^{\vartheta+2\sqrt{\rho^2-\lambda}}}
\to \frac{\gamma^{q+m}\, |\Theta_-|^m}{\vartheta+2\sqrt{\rho^2-\lambda}}=\vartheta\,C_\gamma\quad \mbox{as } r\to 0^+.
$$ 
By the definition of $\mathfrak G$ in \eqref{hsi} and L'H\^opital's rule, we get \eqref{flo1}. 
This completes the proof. 
\end{proof}

\begin{lemma} \label{viz1}  
For every $\gamma\in (0,\infty)$, there exists $R>0$ small such that 
\eqref{eq1} in $B_{R}(0)\setminus \{0\}$ has a
positive radial solution satisfying \eqref{flo1}.  
\end{lemma}

\begin{proof}
Suppose that $u$ is a positive radial solution of \eqref{eq1} in $B_{R}(0)\setminus \{0\}$ with $R>0$ such that \eqref{flo1} holds. 
We define $\mathfrak G(r)$ as in \eqref{hsi} and $\vartheta_1:=\vartheta>0$.

Let $L_{\mathfrak G}=\lim_{r\to 0^+} \mathfrak G(r)=\vartheta\,C_\gamma>0$. 
We fix $\eta>0$ such that $\eta<\vartheta_1/2$.  
Let $r_0\in (0,R)$ be as small as needed. 
For every
$r\in (0,r_0]$, we set 
$ t=\log \,(r_0/r)$
and define 
\begin{equation} \label{trg}
X_1(t)=r^{-\eta}\left( r^{\Theta_-}u(r)-\gamma \right),\quad X_2(t)=\mathfrak{G}(r) -L_{\mathfrak G}\quad \mbox{and}\quad X_3(t)=r^\eta.
\end{equation}

By Lemma~\ref{kiop1}, we have $\mathbf X(t)=(X_1(t),X_2(t), X_3(t))\to (0,0,0)$ as $r\to 0^+$. Hence, for every $\varepsilon>0$, 
we can take $r_0=r_0(\varepsilon)>0$ small such that 
$|X_j(t)|<\varepsilon$ for every $r\in (0,r_0)$ and $j=1,2,3$.  
In what follows, we fix $\varepsilon>0$ small so that $g$ in \eqref{ceva1} is well-defined and smooth on the 
open ball $B_\varepsilon(\mathbf 0)$ in $\mathbb R^3$, where 
$\mathbf 0=(0,0,0)$. 
For every $\xi=(\xi_1,\xi_2,\xi_3)\in B_\varepsilon(\mathbf 0)$, we define 
\begin{equation} \label{ceva1} g(\xi)= |\Theta_-|^m (\xi_1\xi_3+\gamma)^m 
 \left ( 1- \frac{(\xi_2+L_{\mathfrak G})|\xi_3|^{\frac{\vartheta}{\eta}}}{ \left(\xi_1\xi_3+\gamma\right) \Theta_-} \right)^m  
\end{equation}

Let $a_2>0$ be defined by
\begin{equation} \label{aare} a_2= \vartheta+ 2\sqrt{\rho^2-\lambda}.
\end{equation}
Hence, for every $t\geq 0$, we arrive at 
\begin{equation}
\left\{\begin{aligned}
& X_1'(t)=\eta \,X_1(t) -\left(X_2(t)+L_{\mathfrak G}\right) \left(X_3(t)\right)^{\frac{\vartheta_1}{\eta} -1},\\
& X_2'(t)=a_2 \left(X_2(t)+L_{\mathfrak G}\right)
-\left( X_1(t) X_3(t)+\gamma \right)^q  g(\mathbf X(t)),\\
& X_3'(t)=-\eta \,X_3(t).
\end{aligned}
\right.
\end{equation}

We find that 
$\mathbf{X}(t)$ solves for $t\geq 0$ 
the following system
\begin{equation}\label{shs}
\mathbf{X}'(t)=A\mathbf{X}+\mathbf{F}(\mathbf{X}),
\end{equation}
where $A$ is a $3\times 3$ diagonal matrix given by $A={\rm diag}\, [\eta,a_2,-\eta]$. 
Now, $\mathbf F$ in \eqref{shs} is a $C^1$-function on the open ball $B_\varepsilon(\mathbf 0)\subset \mathbb R^3$ satisfying 
$\mathbf F(\mathbf 0)=\mathbf 0$ and $D\mathbf F(\mathbf 0)=0$ since $0<\eta<\vartheta_1/2$. Thus, $\mathbf 0$ is a saddle critical point for \eqref{shs} since $A$ has two positive eigenvalues and one negative eigenvalue. We can apply the Stable Manifold Theorem to \eqref{shs} so that the behaviour of \eqref{shs}   
near $\mathbf{X}=\mathbf 0$ is approximated by the behaviour of its linearization $\mathbf{X}'=A\mathbf X$ at $\mathbf X=\mathbf 0$. 
Let $\phi_t$ be the flow of the system \eqref{shs} in relation to which we obtain the existence of a (local) stable manifold $S$, which is 
one-dimensional.
So, there exist $\varepsilon_0>0$ small and two differentiable functions $w_1,w_2:(-\varepsilon_0,\varepsilon_0)\rightarrow (-\varepsilon_0,\varepsilon_0)$ such that the stable manifold $S$ is 
defined by  
$X_1=w_1(X_3)$ and $X_2=w_2(X_3)$ 
such that for all $t\geq 0$, we have $\phi_t(S)\subseteq S$ and for all $\mathbf{x}_0 \in S$, $\lim_{t\rightarrow \infty} \phi_t(\mathbf{x}_0)=0$. 
Thus,  
for every $x_{03}\in (-\varepsilon_0, \varepsilon_0)$, the system \eqref{shs} subject to the initial condition
\begin{equation}\label{mia2}
\mathbf{X}(0)=(w_1(x_{03}), w_2(x_{03}),x_{03})
\end{equation}
has a solution $\mathbf{X}$ on $[0,\infty)$ satisfying \begin{equation} \label{imc} \lim_{t\to \infty} \mathbf X(t)=\mathbf 0.  
\end{equation}
In particular, we find that 
$X_3(t)=x_{03}\,e^{-\eta t}$ for all $t\geq 0$. 

Let $0<r_0<\varepsilon_0^{1/\eta}$ be arbitrary. 
We construct a positive radial solution $u$ of \eqref{eq1} in $B_{r_0}(0)\setminus \{0\}$ satisfying \eqref{flo1}. 	
Let $\mathbf{X}$ be the above solution of \eqref{shs}, subject to \eqref{mia2}, corresponding to $x_{03}=r_0^\eta$. 
By setting $r=r_0\,e^{-t}$ for $t\geq 0$, we get $X_3(t)=r^\eta$. We define 
\begin{equation} \label{mib}
u(r)= \left(X_1(t)X_3(t)+\gamma\right)(X_3(t))^{-
	\frac{\Theta_-}{\eta}}\quad \mbox{for every } r\in (0,r_0].
\end{equation}
In view of \eqref{shs}, by differentiating \eqref{mib} with respect to $t$, we regain the expression of $X_2$ in \eqref{trg}. By differentiating such an expression with respect to $t$ and using the differential equation for $X_2$ in \eqref{shs}, we obtain that $u$ is a positive radial solution of \eqref{eq1} in $B_{r_0}(0)\setminus \{0\}$ satisfying \eqref{flo1} 
since \eqref{imc} holds. 
This finishes the proof of Lemma~\ref{viz1}.   
\end{proof}

%%%%

%%%%%

\section{The asymptotic profile $V_0$ when $\lambda<\rho^2$ and $\Theta=\Theta_-\not=0$ (Proof of Theorem~\ref{nom})}

Throughout this section, we assume that \eqref{cond1} holds, $\rho \in \mathbb{R}$, $\lambda<\rho^2$ and $\Theta=\Theta_-\neq 0$. Unless otherwise mentioned, we assume that 
$u$ is an arbitrary positive radial solution of \eqref{eq1} in $B_R(0)\setminus \{0\}$ with $R>0$ such that \eqref{mino} holds, that is 
\begin{equation}\label{mino1}
u(r)\sim 
\left( \mathfrak p \kappa\right)^{-\frac{1}{\kappa}}\,r^{-\Theta_-}\left[\log \, (1/r)\right]^{-\frac{1}{\kappa}}:=V_0(r)\ \mbox{as } r\to 0^+,\quad \mbox{where } \mathfrak p:=
\frac{|\Theta_-|^m}{2\sqrt{\rho^2-\lambda}}>0.
\end{equation}

For every $r\in (0,R)$, we define $\mathfrak B_1$ and $\mathfrak B_2$ as in \eqref{jim11}, which means that 
\begin{equation} \label{jima11} \mathfrak B_1(r)=r^{\Theta_{-}}\,u(r)\quad \mbox{and}\quad
\mathfrak B_2(r)=\frac{ru'(r)}{u(r)}+\Theta_{-}=\frac{r\mathfrak B_1'(r)}{\mathfrak B_1(r)}.\end{equation}

\begin{lemma} \label{azi5}
	There exists $r_0>0$ small such that $\mathfrak B_1'(r)>0$, $\mathfrak B_2(r)>0$ and $\mathfrak B_2'(r)>0 $ for every $r\in (0,r_0)$. 
	Moreover, we have 
	\begin{equation}  \label{alo1} \lim_{r\to 0^+} \mathfrak B_2(r)=0\quad \mbox{and} \quad \lim_{r\to 0^+}\mathfrak B_1^{-\kappa} (r)\, \mathfrak B_2(r)=\mathfrak p>0.
	\end{equation}	
\end{lemma}

\begin{proof}
We remark that there exists $r_0\in (0,R)$ such that $u'(r)\not=0$ for every $r\in (0,r_0)$. Indeed, when $\lambda\not=0$ this claim follows with a similar argument to that in Lemma~\ref{wisp0} in Section~\ref{uzpro}. If $\lambda=0$, then the assumption $\Theta=\Theta_-\not=0$ implies that 
$\rho>0$ and $\Theta_-=-2\rho<0$. So, using \eqref{mino1} and \eqref{azi1}, we get that $\lim_{r\to 0^+} u(r)=0$, as well as  	
$r^{1-2\rho} u'(r)$ is non-decreasing on $(0,R)$ and $\lim_{r\to 0^+} r^{1-2\rho} u'(r)=0 $ by L'H\^opital's rule. Moreover, $u'(r)> 0$ for every $r\in (0,R)$. 

We show that $\mathfrak B_1'(r)\not=0$ for every $r>0$ small. 
Suppose by contradiction that there exists a sequence $\{r_n\}_{n\geq 1}$ in $(0,R)$ decreasing to $0$ as $n\to \infty$ such that $\mathfrak B_1'(r_n)=0$ for all $n\geq 1$. With similar calculations to those in \eqref{azi2} and using that $\Theta=\Theta_-$, we obtain that
$$ r_n^2  \mathfrak B_1''(r_n)= |\Theta_-|^m \mathfrak{B}_1^{\kappa+1}(r_n)
\quad \mbox{for all } n\geq 1.$$
Hence, $\mathfrak B_1''(r_n)>0$ for all $n\geq 1$, that is, $r_n$ is a local minimum point for $\mathfrak B_1$, which is impossible. 	
	
Following the same reasoning as in Lemma~\ref{wisp1} in  Section~\ref{uzpro}, we infer that $\mathfrak B_2'(r)\not=0$ for every $r>0$ small. 
From \eqref{mino1}, we get that
$\log \mathfrak B_1(r)/\log (1/r)\to 0$ as $r\to 0^+$. Thus, using \eqref{jima11} and L'H\^opital's rule, we infer that $\lim_{r\to 0^+} \mathfrak B_2(r)=0$. 
	Hence, since $\Theta=\Theta_-$, we find that
	$$ \lim_{r\to 0^+} 
	\frac{  \frac{d}{dr} \left( r^{\Theta_+ } u(r)\,\mathfrak B_2(r)\right) }
	{ \frac{d}{dr} \left( r^{\Theta_+}u(r) \,\mathfrak B_1^{\kappa}(r)\right)}=\lim_{r\to 0^+} \frac{ |\mathfrak B_2(r)+\Theta_-|^m}{\Theta_+-\Theta_- + \left(\kappa+1\right) \mathfrak B_2(r)}=\mathfrak{p}.
	$$
By \eqref{mino1} and L'H\^opital's rule, we conclude the second limit of \eqref{alo1}. 
\end{proof}

We define
\begin{equation}\label{meee}
\mathfrak M:=\frac{\kappa+1}{2\sqrt{\rho^2-\lambda}}+\frac{m}{\Theta_-}=\frac{\left(\kappa+1\right)\mathfrak p}{|\Theta_-|^m}+ \frac{m}{\Theta_-}. 
\end{equation}

For $\mathfrak p$ given by \eqref{mino1} and every $r>0$ small, we define
\begin{equation} \label{dfx} 
t=(r^{\Theta_-} u(r))^{-\kappa}\quad \mbox{and} \quad
X(t)=t\left( \frac{ru'(r)}{u(r)}+\Theta_-\right)-\mathfrak{p}=\mathfrak B_1^{-\kappa}(r)\, \mathfrak B_2(r)-\mathfrak p. 
\end{equation}

	\begin{lemma}\label{bx}
	Assume that either $\mathfrak M\neq 0$ or $m\neq 1$. For $t$ and $X(t)$ defined in \eqref{dfx}, we have $t\rightarrow \infty$ and $X(t)\rightarrow 0$ as $r\rightarrow 0^+$. Moreover, for $t>0$ large, we have $|X(t)|<\mathfrak p/2$ and 
	\begin{equation} \label{dust}  \frac{dX}{dt}=\frac{2\sqrt{\rho^2-\lambda}}{\kappa \left(X(t)+\mathfrak p\right)}\,X(t)+\left(1+\frac{1}{\kappa}\right) \frac{X(t)+\mathfrak p}{t}+\frac{|\Theta_-|^m}{\kappa \left(X(t)+\mathfrak p\right)} \left[  1-\left( 1-\frac{X(t)+\mathfrak p}{t\,\Theta_-}\right)^m\right].
\end{equation}
In addition, it holds
\begin{equation}\label{miso}
\begin{aligned}
&\lim_{t\rightarrow \infty}t X(t)=-\mathfrak M \,\mathfrak p^2&\quad& \text{if }\mathfrak M\neq 0,\\
&\lim_{t\rightarrow \infty}t^2 X(t)=\frac{m\left(m-1\right)\mathfrak p^3}{2\left(\Theta_-\right)^2}&\quad& \text{if }\mathfrak M=0\text{ and } m\neq 1.
\end{aligned}
\end{equation}
	\end{lemma}
	
	\begin{proof}
	Since \eqref{mino1} holds, we have 
\begin{equation}\label{poll}
\frac{t}{\log\,(1/r)}\rightarrow \mathfrak p\kappa \quad \text{as } r\rightarrow 0^+.
\end{equation}
 This implies that $t\rightarrow \infty$ as $r\rightarrow 0^+$. 
 
By Lemma~\ref{azi5}, we know that $X(t)\to 0$ as $r\to 0^+$. 
	For $t>0$ large, we find that 
	\begin{equation} \label{comp} \frac{dt}{dr}=-\frac{\kappa}{r}\left(X(t)+\mathfrak p\right)\ \mbox{or, equivalently, } 
	\frac{dr}{dt}=-\frac{r}{\kappa\left(X(t)+\mathfrak p\right)}.
	\end{equation}
	Since $X(t)\rightarrow 0$ as $t\rightarrow \infty$, we can take $T_0>4\,\mathfrak p/|\Theta_-|$ large such that 
	$$ |X(t)|\leq \frac{\mathfrak p}{2}\quad \mbox{and}\quad  1-\frac{X(t)+\mathfrak p}{t\,\Theta_-}>0 
	\quad \mbox{for all } t\geq T_0.$$
	 Using \eqref{ez} and $\Theta_-^2+2\rho\, \Theta_-+\lambda=0$, we find that 
	 \begin{equation} \label{uv1}
	 t \,r\frac{d}{dr}\left(\frac{ru'(r)}{u(r)}\right)=
	-\frac{(X(t)+\mathfrak p)^2}{t}
	 -2\,\sqrt{\rho^2-\lambda}\left(X(t)+\mathfrak p\right)
	+|\Theta_-|^m \left( 1-\frac{X(t)+\mathfrak p}{t\,\Theta_-}
	 \right)^m.
	 \end{equation}
	 In view of \eqref{comp}, by differentiating \eqref{dfx} with respect to $t$, we get 
	 \begin{equation} \label{uv2} \frac{dX}{dt}=\frac{X(t)+\mathfrak p}{t}-\frac{t\,r}{\kappa \left( X(t)+\mathfrak p\right)}\,\frac{d}{dr}\left(\frac{ru'(r)}{u(r)}\right)\quad \mbox{for } t\geq T_0.
	 \end{equation}
	 Employing \eqref{uv1} into \eqref{uv2}, we arrive at \eqref{dust}.
	 
	 For $t>0$ large, we define $ \mathfrak I(t):=\mathfrak I_1(t)+\mathfrak I_2(t)$, where 
	 \begin{equation} 
	 \label{azi6} 
	 \mathfrak I_1(t):=	\left(1+\frac{1}{\kappa}\right) \frac{X(t)+\mathfrak p}{t}\quad \mbox{and}\quad 
	 \mathfrak I_2(t):=\frac{|\Theta_-|^m}{\kappa \left(X(t)+\mathfrak p\right)} \left[  1-\left( 1-\frac{X(t)+\mathfrak p}{t\,\Theta_-}\right)^m\right] .
	 \end{equation} 
	Hence, after multiplying \eqref{dust} by $e^{-\frac{2\sqrt{\rho^2-\lambda}}{\kappa\mathfrak p}t}$, we get
	  \begin{equation}\label{bezit}
	  \frac{d}{dt}\left ( X(t)\,e^{-\frac{2\sqrt{\rho^2-\lambda}}{\kappa\mathfrak p}t} \right )= \left [\frac{-2\sqrt{\rho^2-\lambda}}{\kappa\,\mathfrak p \left(X(t)+\mathfrak p\right)}\,X^2(t)+\mathfrak I(t)\right ]e^{-\frac{2\sqrt{\rho^2-\lambda}}{\kappa\mathfrak p}t}.
	  	\end{equation}
	Using \eqref{meee}, we observe that as $t\to \infty$, 
	 \begin{equation} \label{azi7} \begin{aligned} 
 & \mathfrak I_2(t)= \frac{m|\Theta_-|^m}{\kappa \Theta_-} \frac{1}{t}-\frac{m\left(m-1\right)\mathfrak p|\Theta_-|^m }{2 \kappa \left(\Theta_-\right)^2} \frac{1}{t^2} (1+o(1)),\\
	& \mathfrak I(t)=
	\frac{\mathfrak M\, |\Theta_-|^m}{\kappa}
	 \frac{1}{t} +\frac{\left(\kappa+1\right)}{\kappa} \frac{X(t)}{t} -\frac{m\left(m-1\right)\mathfrak p|\Theta_-|^m }{2 \kappa \left(\Theta_-\right)^2} \frac{1}{t^2} (1+o(1)). 
	 \end{aligned}
	 \end{equation} 	
	  	
We next prove that $X'(t)\not=0$ for every $t>0$ large. Here, $X'(t)$ stands for the derivative of $X(t)$ with respect to $t$. Suppose by contradiction that there exists an  
increasing sequence $\{t_n\}_{n\geq 1}$ in $[T_0,\infty)$ such that $\lim_{n\to \infty} t_n=\infty$ and 
	 $ X'(t_n)=0$ for all $ n\geq 1$. In light of \eqref{dust} and \eqref{azi7}, we obtain that
\begin{equation} \label{azi8}
\begin{aligned}
&\lim_{n\rightarrow \infty}t_n X(t_n)=-\mathfrak M \,\mathfrak p^2&\quad& \text{if }\mathfrak M\neq 0,\\
&\lim_{n\rightarrow \infty}t_n^2 X(t_n)=\frac{m\left(m-1\right)\mathfrak p^3}{2\left(\Theta_-\right)^2}&\quad& \text{if }\mathfrak M=0\text{ and } m\neq 1.
\end{aligned}
\end{equation}	
By differentiating \eqref{dust}, we arrive at 
$$ t_n^2 X''(t_n)= -\left(1+\frac{1}{\kappa}\right) \left(X(t_n)+\mathfrak p\right)-\frac{m|\Theta_-|^m}{\kappa \Theta_-}\left( 1-\frac{X(t_n)+\mathfrak p}{t_n\,\Theta_-}\right)^{m-1}.
$$
In view of \eqref{azi8}, we get that 	
  	\begin{equation} \label{azi9}
\begin{aligned}
&\lim_{n\rightarrow \infty}t^2_n X''(t_n)=-\frac{\mathfrak M |\Theta_-|^m}{\kappa}\not=0 &\quad& \text{if }\mathfrak M\neq 0,\\
&\lim_{n\rightarrow \infty}t_n^3 X''(t_n)=\frac{m\left(m-1\right)\mathfrak p |\Theta_-|^{m-2}}{\kappa}\not=0&\quad& \text{if }\mathfrak M=0\text{ and } m\neq 1.
\end{aligned}
\end{equation}	
	Hence, for every $n\geq 1$ large, $X''(t_n)$ has constant sign so that the critical points $t_n$ of $X$ will be either local maxima or local minima. This is impossible. Therefore, $X'
	(t)\not=0$ for $t>0$ large. Since $\lim_{t\to \infty} X(t)=0$, it follows that 
	$X(t)$ and $X'(t)$ have opposite signs for $t>0$ large.  Thus, from \eqref{dust}, we can find a constant $C>0$ such that for every $t>0$ large, 
	\begin{equation} \label{azi09} \begin{aligned} 
 &	t|X(t)|<C &\quad& \text{if }\mathfrak M\neq 0,\\
 &  t^2 |X(t)|<C &\quad& \text{if }\mathfrak M=0\text{ and } m\neq 1.
 \end{aligned}
	\end{equation}  	
In either case, we obtain that 
\begin{equation}\label{bezit1}
	\lim_{t\rightarrow \infty} t X^2(t)=0.
	\end{equation}
	  	
	  	\vspace{0.2cm}
First, we assume that $\mathfrak M\neq 0$. 	
	Using L'H\^opital's rule, \eqref{bezit} and \eqref{bezit1}, we obtain that 
	$$\lim_{t\rightarrow \infty} t X(t)=\lim_{t\rightarrow \infty} \frac{\frac{d}{dt}\left ( X(t)\,e^{-\frac{2\sqrt{\rho^2-\lambda}}{\kappa\mathfrak p}t}\right )}{\frac{d}{dt}\left (\frac{e^{-\frac{2\sqrt{\rho^2-\lambda}}{\kappa\mathfrak p}t}}{t}\right )}=\lim_{t\to \infty} \frac{t \mathfrak I(t)}{-\frac{2\sqrt{\rho^2-\lambda}}{\kappa\mathfrak p}}
	=-\mathfrak M\, \mathfrak p^2\,.$$

Second, we assume that $\mathfrak M=0$ and $m\neq 1$. In this case, 
	\eqref{azi09} implies that $\lim_{t\rightarrow \infty} t X(t)=0$. Then, using \eqref{bezit}, \eqref{azi7} and  L'H\^opital's rule, we infer that 
	$$\lim_{t\rightarrow \infty} t^2 X(t)=\lim_{t\rightarrow \infty} \frac{\frac{d}{dt}\left ( X(t)\,e^{-\frac{2\sqrt{\rho^2-\lambda}}{\kappa\mathfrak p}t}\right )}{\frac{d}{dt}\left (\frac{e^{-
\frac{2\sqrt{\rho^2-\lambda}}{\kappa\mathfrak p}t}}{t^2}\right )}=\lim_{t\to \infty} \frac{t^2 \mathfrak I(t)}{-\frac{2\sqrt{\rho^2-\lambda}}{\kappa\mathfrak p}}
=\frac{m\left(m-1\right)\mathfrak p^3}{2\left(\Theta_-\right)^2}.$$
This completes the proof of Lemma~\ref{bx}. 
	\end{proof}

\begin{lemma}
If $\mathfrak M=0$ and $m=1$, then there exists a constant $C\in \mathbb R$ such that
 \begin{equation} \label{bre} u(r) =V_0(r) \left(1+\frac{C}{\log \frac{1}{r}} \right)^{-\frac{1}{\kappa}}\quad \mbox{for every } r>0\ \mbox{small}.
	 \end{equation}  
	
\end{lemma}

\begin{proof}
 Assume that $\mathfrak M=0$ and $m=1$. Then, from \eqref{dust}, we see that $X(t)$ satisfies 
	 \begin{equation} \label{unt} \frac{dX}{dt}=\left[  \frac{ 2\sqrt{\rho^2-\lambda}}{\kappa \left(X(t)+\mathfrak p\right)} +\left( 1+\frac{1}{\kappa}\right) \frac{1}{t}\right] X(t)\quad \mbox{for all } t\geq T_0.\end{equation} 
	 For the right-hand side of \eqref{unt}, we observe that $X(t)$ is multiplied by a positive quantity that converges to $ 2\sqrt{\rho^2-\lambda}/(\kappa\mathfrak p)$ as $t\to \infty$.   
	 As $X(t)\to 0$ as $t\to \infty$, the only solution of \eqref{unt} is $X\equiv 0$ on $[T_0,\infty)$. From the definition of $X$ in \eqref{dfx}, we conclude that 
	 $$ \frac{d}{dr} \left( (r^{\Theta_-} u(r))^{-\kappa} \right)=-\frac{\mathfrak p \kappa}{r} \quad \mbox{for every } r>0 \ \mbox{small}.
	 $$
	 Consequently, there exists a constant $C\in \mathbb R$ such that for every $r>0$ small, we have \eqref{bre}. 
\end{proof}

\begin{lemma}
We assume that either $\mathfrak M\neq0$ or $m\not=1$.  

$\bullet$ If $\mathfrak M\neq0$, then
\begin{equation} \label{azi11}
	\frac{u(r)}{V_0(r)}= 1+\frac{\mathfrak M}{\kappa^2}\,\frac{\log \log\, (1/r)}{\log \, (1/r)}(1+o(1))\quad \mbox{as } r\to 0^+.
\end{equation}

$\bullet$ If $\mathfrak M=0$ and $ m\neq 1$, then there exists a constant $C\in \mathbb R$ such that as $r\rightarrow 0^+$, 
\begin{equation} \label{sac12}
	\frac{u(r)}{V_0(r)}= 1+ \frac{C}{\log \,(1/r)}+
	\left[ \frac{m\left(m-1\right)}{2\kappa^3\,(\Theta_-)^2}+\frac{\left(\kappa+1\right) C^2}{2}\right] \frac{1}{\log^2 \,(1/r)}\,(1+o(1)) .
	\end{equation}
\end{lemma}	
	
\begin{proof}
For every $r>0$ small, we define $\Psi(r)$ by 
\begin{equation} \label{azi12} \Psi(r):=t-\mathfrak p \kappa \log\,(1/r)=(r^{\Theta_-}u(r))^{-\kappa}-\mathfrak p \kappa \log\,(1/r).
\end{equation}
By the definition of $X$ in \eqref{dfx}, for every $t\geq T_0$, we have
	 \begin{equation} \label{mis2} X(t)=-\frac{r}{\kappa}\,\frac{d}{dr} \left( (r^{\Theta_-}u(r))^{-\kappa}\right) -\mathfrak p=-\frac{r}{\kappa} \frac{d}{dr}( \Psi(r)).
	 \end{equation}

$\bullet$ First, we assume that  $\mathfrak M\neq0$. Using \eqref{miso}, \eqref{poll} and \eqref{mis2},  
we infer that    
	 \begin{equation*}  
	 \frac{\frac{d}{dr}\left (\Psi(r)\right)}{\frac{d}{dr}[\log \log \, (1/r)]}
=\kappa X(t) \log\, (1/r) \to -\mathfrak M\,\mathfrak p
 \quad \mbox{as } r\to 0^+.
	 \end{equation*}
This implies that $\lim_{r\to 0^+} \Psi(r)=-\infty\, {\rm sgn}\, (\mathfrak M) \,$ and by L'Hopital's rule,
\begin{equation} \label{star1} \lim_{r\rightarrow 0^+}\frac{\Psi(r)}{\log \log \, (1/r)}=-\mathfrak M\,\mathfrak p.
\end{equation}
The claim in \eqref{azi11} follows from \eqref{azi12} and \eqref{star1}.

\vspace{0.2cm}
$\bullet$
Second, we assume that  $\mathfrak M=0$ and $m\not=1$. 
By \eqref{mis2} and Lemma~\ref{bx}, we get that 
	 \begin{equation} \label{nut} \lim_{r\to 0^+} \frac{\frac{d}{dr} \left( \Psi(r)\right)}{\frac{d}{dr} \left [1/\log (1/r)\right]} =-\frac{m\left(m-1\right)\mathfrak p}{2\kappa \left(\Theta_-\right)^2}\not=0.
	 \end{equation} 
	 Thus, ${\rm sgn}\, (\Psi'(r))={\rm sgn} \,(1-m)$ and                            
there exists $ \lim_{r\to 0^+}\Psi(r)=L_1\in \mathbb R $. Moreover, by L'H\^opital's rule and \eqref{nut}, we get
$$ \Psi(r)-L_1= -\frac{m\left(m-1\right)\mathfrak p}{2\kappa \left(\Theta_-\right)^2} \frac{1}{\log \, (1/r)} (1+o(1))\quad \mbox{as } r\to 0^+.
$$
Hence, using \eqref{azi2}, we find that
$$ \frac{u(r)}{V_0(r)}=\left( 1+\frac{L_1}{\mathfrak p \kappa \log \,(1/r)} 
-\frac{m\left(m-1\right)}{2\kappa^2 \left(\Theta_-\right)^2} \frac{1}{\log^2 \, (1/r)} (1+o(1))
\right)^{-\frac{1}{\kappa}} \quad \mbox{as } r\to 0^+.
$$
This implies the claim in \eqref{sac12}, where $C=-L_1/(\mathfrak p \kappa^2)$. \end{proof}

\begin{lemma} \label{simp2} 
Let $R>0$ be arbitrary. Then, equation \eqref{eq1} in $B_R(0)\setminus \{0\}$ has infinitely many positive radial solutions $u$ satisfying \eqref{mino1}. 	
\end{lemma}

\begin{proof}	 
	We remark that for $T_0>0$ large enough, the differential equation in \eqref{dust} on $[T_0,\infty)$ admits a unique solution $X(t)$ satisfying $\lim_{t\to \infty} X(t)=0$. %(see Remark~\ref{rt1}). 
	
	For $r_0>0$ fixed arbitrary, we consider the differential equation for $r$ in \eqref{comp} for $t\in [T_0,\infty)$, subject to $r(T_0)=r_0$, namely,  
		\begin{equation} \label{compi} 
 \left\{ 	\begin{aligned} 
 & \frac{dr}{dt}=-\frac{r}{\kappa\left(X(t)+\mathfrak p\right)}\quad \mbox{for } t\geq T_0, \\
 & r(T_0)=r_0. 
\end{aligned}
\right.
	\end{equation}
This problem has a unique solution given by
	\begin{equation} \label{bil} r(t)=r_0 \exp\left( -\frac{1}{\kappa} \int_{T_0}^ t \frac{ds}{X(s)+\mathfrak p}\right) \quad \mbox{for } t\geq T_0. 
	\end{equation}
	Since $r=r(t)$ is invertible, from \eqref{bil}, we can express $t$ as a function of $r$. 
	We define $u(r)$ by 
	\begin{equation} \label{ba1} u(r)=t^{-\frac{1}{\kappa}}\,r^{-\Theta_-}\quad \mbox{for every } r\in (0,r_0].  
	\end{equation} Then, using \eqref{compi}, jointly with \eqref{dust} and $\lim_{t\to \infty} X(t)=0$, we obtain that $u$ is a positive solution of \eqref{ez} in $(0,r_0]$ satisfying $u(r_0)=T_0^{-1/\kappa} r_0^{-\Theta_-}$ and \eqref{mino1}. 
	
	\vspace{0.2cm}
	We can simply obtain infinitely many positive solutions of \eqref{ez} in $(0,r_0]$ satisfying \eqref{mino1} by following the above procedure in which $T_0$ is replaced by $T_1$ with $T_1>T_0$. This change will apply in \eqref{compi} and in \eqref{bil} so that the solution of \eqref{compi} on $[T_1,\infty)$, subject to $r(T_1)=r_0$, say $ \widetilde r$, is a constant multiple of the solution $r$ given in \eqref{bil}. We write $\widetilde r=\mu\, r$ for some constant $\mu>1$ given by 
	$\mu=\exp\left(\frac{1}{\kappa} \int_{T_0}^{T_1}\frac{ds}{X(s)+\mathfrak p}\right)$.  
	Then, using $\widetilde r$ instead of $r$ in \eqref{ba1}, we obtain 
	$$\widetilde u(\widetilde r)=t^{-\frac{1}{\kappa}}\,\widetilde r^{-\Theta_-}=\mu^{-\Theta_-}\, u(\widetilde r/\mu )\quad \mbox{for every } \widetilde r\in (0,r_0].$$
	Using the transformation $T_j[u]$ defined in \eqref{scale} with $j=1/\mu$ and the fact that $\Theta=\Theta_-$, we simply recover that $T_{1/\mu}[u](\widetilde r)=\widetilde u(\widetilde r)$ is also a positive solution of \eqref{ez} on $(0,r_0]$ with the same properties near zero as $u$, namely, $\widetilde u(r)/V_0(r)\to 1$ as $r\to 0^+$. Moreover, since $r\longmapsto r^{\Theta_-}u(r)$ is increasing on $(0,r_0]$, it is clear that $\widetilde u(r_0)<u(r_0)$. Thus, $u$ and $\widetilde u$ are different solutions of \eqref{ez} on $(0,r_0]$, although they are obtained one from the other via the above scaling transformation.  
	
	Since $r_0>0$ was arbitrary in the above reasoning, the proof of Lemma~\ref{simp2} is complete. 
\end{proof}

\section{The asymptotic profile $W_0$ when $\lambda=\rho^2$, $\Theta=\Theta_\pm\not=0$ (Proof of Theorem~\ref{noma})} 

Throughout this section, we assume that \eqref{cond1} holds, $\lambda=\rho^2\not=0$ and $\Theta=\Theta_\pm=-\rho$. 
Unless otherwise stated, 
$u$ is a positive radial solution of \eqref{eq1} in $B_R(0)\setminus \{0\}$ for $R>0$ satisfying \eqref{wzer}:
\begin{equation}
\label{wod}
u(r)\sim \left(\frac{2\left(\kappa+2\right)}{\kappa^{2}
|\Theta_-|^{m}}\right)^
{\frac{1}{\kappa}}r^{-\Theta_-}\left[\log \, (1/r)\right]^{-\frac{2}{\kappa}}:=
W_0(r)\quad \mbox{as } r\to 0^+.  
\end{equation}

The assertion of Theorem~\ref{noma} is proved in Lemmas~\ref{nods} and \ref{soad3} below.

Let $\mathfrak B_1$ and 	$\mathfrak B_2$ be given by \eqref{jim11} (see also \eqref{jima11}). 
Since $\Theta=\Theta_-=-\rho$, this means that
$$ \mathfrak B_1(r)= r^{-\rho} u(r)\quad \mbox{and}\quad 
\mathfrak B_2(r)=\frac{ru'(r)}{u(r)}-\rho\quad \mbox{for every } r\in (0,R).
$$

\begin{lemma} \label{octa1}
Then, there exists $r_0>0$ small such that $\mathfrak B_1'(r)>0$, $\mathfrak B_2(r)>0$ and $\mathfrak B_2'(r)>0 $ for every $r\in (0,r_0)$. 
	Moreover, we have 
	\begin{equation}  \label{alm1} \lim_{r\to 0^+} \mathfrak B_2(r)=0\quad \mbox{and} \quad \lim_{r\to 0^+}\mathfrak B_1^{-\frac{\kappa}{2}} (r)\, \mathfrak B_2(r)=
	\left( \frac{2 \,|\rho|^m}{\kappa+2}\right)^{\frac{1}{2}}:=\mathfrak Q>0.
	\end{equation}	
\end{lemma}

\begin{proof}
We follows the same argument as in the proof of Lemma~\ref{azi5} in which we replace $\lambda$ and
$\Theta_-$ by $\rho^2$ and $-\rho$, respectively. 
This leads to $\lim_{r\to 0^+} \mathfrak B_2(r)=0$. 
We need only prove the second limit in \eqref{alm1}.
From \eqref{wod}, we have 
\begin{equation} \label{foam1} \mathfrak B_1(r)=r^{-\rho} u(r)\sim  \left(\frac{2\left(\kappa+2\right)}{\kappa^{2}
|\rho|^{m}}\right)^
{\frac{1}{\kappa}}\left[\log \, (1/r)\right]^{-\frac{2}{\kappa}}\quad \mbox{as } r\to 0^+.
\end{equation}
Since $u$ is a positive solution of \eqref{ez}, 
using that $\lambda=\rho^2$ and $\Theta=\Theta_-=-\rho$, we get
\begin{equation} \label{wim} 
 r\frac{d}{dr} [\mathfrak B_1(r)\,\mathfrak B_2(r)]=
|\mathfrak B_2(r)+\rho|^m \mathfrak B_1^{\kappa+1}(r)\quad \mbox{for all } r\in (0,R).  
\end{equation}
 From \eqref{foam1} and \eqref{wim}, we see that
 $$\lim_{r\to 0^+} \frac{ \frac{d}{dr}[\mathfrak B_1(r) \mathfrak B_2(r)]}{ \frac{d}{dr} \left([\log \,(1/r)]^{-\frac{\kappa+2}{\kappa}}\right)}=\frac{\kappa\,|\rho|^{m}}{\kappa+2}  \left(\frac{2\left(\kappa+2\right)}{\kappa^{2}
|\rho|^{m}}\right)^
{\frac{\kappa+1}{\kappa}}.
 $$
 Hence, by L'H\^opital's rule and \eqref{foam1}, we conclude the proof of the second limit in \eqref{alm1}.  
\end{proof}

For every $r>0$ small, we define
\begin{equation} \label{dafx} 
e^t=\mathfrak B_1^{-\frac{\kappa}{2}}(r)\quad \mbox{and} \quad
X(t)=
e^t\left( \frac{ru'(r)}{u(r)}-\rho\right)-\mathfrak Q=e^t\,\mathfrak B_2(r)-\mathfrak Q.
\end{equation}

\begin{lemma} \label{nods}
The following assertions hold. 
\begin{itemize}
	\item[(a)] As $r\to 0^+$, we have $e^t/\log \,(1/r)\to  \kappa\, \mathfrak Q /2$ and $X(t)\to 0$.
\item[(b)] There exists $T_0>\log \left(4\mathfrak Q/|\rho|\right)$ large such that for all $t\geq T_0$, we have $|X(t)|\leq \mathfrak Q/2$ and
\begin{equation} \label{dust2}  \frac{dX}{dt}=
\left(1+\frac{2}{\kappa} \right)\left(X+ \mathfrak Q\right)\left[ 1-\frac{\mathfrak Q^2}{(X+\mathfrak Q)^2 } \left( 1+
\frac{X+\mathfrak Q}{\rho\, e^t }\right)^m
\right].
\end{equation}

\item[(c)] We have 
\begin{equation} \label{asfa1} e^t X(t)\to \frac{2m|\rho|^m}{\rho \left(3\kappa+4\right)}\quad \mbox{as } t\to \infty. \end{equation} 
\item[(d)] The claim in \eqref{samiii} holds, that is,
	\begin{equation} \label{sami}
	\frac{u(r)}{W_0(r)}=1-\frac{4m \left(2+\kappa\right)}{\rho\kappa^2 \left(3\kappa+4\right)}\,\frac{\log \log \, (1/r)}{\log \, (1/r)}(1+o(1)) \quad \mbox{as } r\to 0^+.
	\end{equation}
\end{itemize}
\end{lemma}

\begin{proof}
(a) The claim follows from \eqref{foam1} and Lemma~\ref{octa1}.   

(b) In view of (a), we can find $T_0>\log \left(4\mathfrak Q/|\rho|\right)$ large such that $|X(t)|\leq \mathfrak Q/2$ for all $t\geq T_0$ and $1+e^{-t} (X(t)+\mathfrak Q)/\rho>0$. 
From \eqref{dafx}, we derive that 
\begin{equation} \label{mic1}  
 \frac{1}{r} \,\frac{dr}{dt}=- \frac{2e^{t}}{\kappa\left( X(t)+\mathfrak Q\right)}. \end{equation} 
for all $t\geq T_0$. 
By using \eqref{mic1} and \eqref{muzia}, we see that 
$$ \begin{aligned} 
\frac{d}{dt} [\mathfrak B_2(r)]& =-\frac{2\,r e^t }{\kappa \left(X(t)+\mathfrak Q\right)} \frac{d}{dr}[\mathfrak B_2(r)]\\
&=-\frac{2 \,e^{-t}}{\kappa \left(X(t)+\mathfrak Q\right)} \left[ |\rho|^m \left( 1+\frac{X(t)+\mathfrak Q}{\rho\, e^t} 
\right)^{m} -(X(t)+\mathfrak Q)^2\right].
 \end{aligned} $$
Thus, by differentiating $X(t)$ in \eqref{dafx} with respect to $t$ and using $\mathfrak Q$ in \eqref{alm1}, we arrive at \eqref{dust2}.  

(c) From \eqref{dust2} and $\lim_{t\to \infty} X(t)=0$, we have
\begin{equation} \label{6b} X'(t)= \left(1+\frac{2}{\kappa}\right)\frac{X(X+2\mathfrak Q)}{X+\mathfrak Q}-\left(1+\frac{2}{\kappa}\right)\frac{m\mathfrak Q^2}{\rho}\,e^{-t} (1+o(1))\quad \mbox{as } t\to \infty.
\end{equation}

We first see that $X'(t)$ has a constant sign for every $t\geq T_0$ large enough. Indeed, if we assume the contrary, then $X'(t_n)=0$ for some sequence $\{t_n\}_{n\geq 1}$ in $[T_0,\infty)$
 such that $\lim_{n\to \infty} t_n=\infty$. Using that $\lim_{n\to \infty} X(t_n)=0$, we find that 
 $$ \lim_{n\to \infty} e^{t_n} X''(t_n)=\left(1+\frac{2}{\kappa}\right) \frac{m\mathfrak Q^2}{\rho}\not=0.
 $$
 Then, for every $n\geq 1$ large, we obtain that the sign of $X''(t_n)$ is that of $\rho$, which leads to a contradiction. 
 We have thus proved that $X'(t)$ has a constant sign for each $t\geq T_0$ large.
 
Since $\lim_{t\to \infty} X(t)=0$, it follows that  for large $t>0$, $X(t)$ and $X'(t)$ have opposite signs.  
So, using that $\mathfrak Q$, $m$ and $\kappa$ are positive, from \eqref{6b}, we get
${\rm sgn}\,(X(t))={\rm sgn}\,(\rho)=-{\rm sgn} (X'(t))  $ for every $t\geq T_0$ large. Moreover, for some constant $C>0$, we have 
$$ e^t |X(t)|\leq C\quad \mbox{for every large } t\geq T_0.$$
This implies that 
\begin{equation} \label{fzc} 
\lim_{t\to \infty} e^{t} X^2(t)=0.
\end{equation}

By multiplying \eqref{dust2} by $e^{-2(1+2/\kappa)t}$, we find that as $t\to \infty$, 
$$ \frac{d}{dt}\left( e^{-2\left(1+\frac{2}{\kappa}\right)t} X(t)\right) =
-\left(1+\frac{2}{\kappa}\right)  \left[ 
\frac{X^2(t) e^{-2\left(1+\frac{2}{\kappa}\right)t}}{X(t)+\mathfrak Q}
+\frac{m\mathfrak Q^2}{\rho} \, e^{-\left( 3+\frac{4}{\kappa}\right)t} (1+o(1))\right].
$$
Then, by L'H\^opital's rule and \eqref{fzc}, we derive that 
$$ \lim_{t\to \infty} e^t X(t)=\lim_{t\to \infty} \frac{e^{-2\left(1+\frac{2}{\kappa}\right)t} X(t)}{ e^{-\left(3+\frac{4}{\kappa}\right)t}} 
=\frac{\kappa+2}{3\kappa+4} 
\left( \frac{m\mathfrak Q^2}{\rho}+
\frac{1}{\mathfrak Q} \lim_{t\to \infty} e^t X^2(t)\right)
=\frac{2m|\rho|^m}{\rho\left(3\kappa+4\right)}.
$$
This proves the claim in \eqref{asfa1}.

(d)
Using \eqref{asfa1} and (a), we obtain that 
\begin{equation} \label{guw} \lim_{r\to 0^+} X(t) \log \, (1/r)=\frac{4m|\rho|^m}{\rho\kappa\left(3\kappa+4\right) \mathfrak Q}\not=0.
\end{equation}
From \eqref{dafx} and \eqref{mic1}, we observe that 
$$ X(t)=-\frac{2r}{\kappa} \frac{d}{dr} \left[  \left( r^{\rho} u(r)\right)^{-\frac{\kappa}{2}}-\frac{\kappa \mathfrak Q}{2} \log \, (1/r)
\right]. 
$$
Hence, using \eqref{guw} and L'H\^opital's rule, we infer that there exists
$$ \lim_{r\to 0^+} \frac{ \left( r^{\rho} u(r)\right)^{-\frac{\kappa}{2}}-\frac{\kappa \mathfrak Q}{2} 
\log \, (1/r)}{\log \log \, (1/r) } =\frac{2m|\rho|^m}{\rho\left(3\kappa+4\right)\mathfrak Q}, 
$$ 
which implies that \eqref{sami} holds. The proof of Lemma~\ref{nods} is now complete. \end{proof}

\begin{lemma} \label{soad3}
Let $R>0$ be arbitrary. Then, equation \eqref{eq1} in $B_R(0)\setminus \{0\}$ has infinitely many positive radial solutions $u$ satisfying \eqref{wod}. 	
\end{lemma}

\begin{proof}
We follow a similar argument to that in Lemma~\ref{simp2}, working here with \eqref{dafx} and \eqref{dust2} instead of \eqref{dfx} and \eqref{dust}, respectively. 
Hence, for $T_0>0$ large, we have a unique solution $X(t)$ of \eqref{dust2} on $[T_0,\infty)$ such that $\lim_{t\to \infty} X(t)=0$.  

For $r_0>0$ fixed arbitrary, we see that the differential equation in \eqref{mic1} for $t\in [T_0,\infty)$, subject to $r(T_0)=r_0$, 
has a unique solution $r=r(t)$ given by 
	\begin{equation} \label{bila} r(t)=r_0 \exp\left( -\frac{2}{\kappa} \int_{T_0}^ t \frac{e^s}{X(s)+\mathfrak Q}\,ds\right) \quad \mbox{for } t\geq T_0. 
	\end{equation}
	Moreover, we find that 
	$$ \lim_{r\to 0^+} \frac{\log \, (1/r)}{e^t} =-\lim_{r\to 0^+} \frac{\frac{dr}{dt}}{r\,e^t}=\frac{2}{\kappa \mathfrak Q}= \left( \frac{2\left(\kappa+2\right)}{\kappa^2 |\rho|^m}\right)^{\frac{1}{2}}.
	$$
	For every $r\in (0,r_0]$, 
	we define $u(r)$ by 
	\begin{equation} \label{bacv} u(r)=\,r^{\rho}\,e^{-\frac{2}{\kappa}t}.  
	\end{equation} 
	Then, $u$ is a positive radial solution of \eqref{eq1} in $B_{r_0}(0)\setminus \{0\}$ satisfying \eqref{wod}. 
Recall that $\Theta=\Theta_\pm=-\rho$. Thus, for each $j\in (0,1)$, by applying the transformation $T_j[u](r)=j^{-\rho} u(jr)$ for all $r\in (0,r_0/j)$, we get that $T_j [u]$ is another positive radial solution of \eqref{eq1} in $B_{r_0/j} (0)\setminus \{0\}$ satisfying \eqref{wod}. We have $T_{j_1}[u](r_0)<T_{j_2}[u](r_0)<u(r_0)$ for every $0<j_1<j_2<1$ since $r\longmapsto r^{-\rho} u(r) $ is increasing on $(0,r_0]$. 
This ends the proof of Lemma~\ref{soad3}. 
\end{proof}	

%%%%

\section{The asymptotic profile $Y_0$ when $\lambda=\Theta=0< \rho\left(m-1\right)$  (Proof of Theorem~\ref{ytam})}

Throughout this section, we assume that 
$$ \eqref{cond1} \ \mbox{holds}, \ \rho \in \mathbb{R}\setminus \{0\},\  \Theta=\lambda=0\ \mbox{and }\rho \left (m-1\right )>0.$$ 
We prove Theorem~\ref{ytam} by combining Lemmas~\ref{bx111} and \ref{gigg}. 
Unless otherwise stated, 
$u$ is an any positive radial solution of \eqref{eq1} in $B_R(0)\setminus \{0\}$ with $R>0$ such that \eqref{yoo} holds, that is 
\begin{equation}\label{yoo1}
u(r)\sim \left[  \mathfrak C\,\log \, (1/r)\right]^{\frac{m-1}{\kappa}}:=Y_0(r)\ \mbox{as } r\to 0^+,\quad \mbox{where } \mathfrak C:=\frac{\kappa\left (2|\rho|\right )^{\frac{1}{m-1}}}{|m-1|}>0.
\end{equation}

By \eqref{azi1}, if $m<1$, then $\lim_{r\rightarrow 0^+}u(r)=0$ and $u'(r)>0$ for every $r>0$ small. If $m>1$, then $\lim_{r\rightarrow 0^+}u(r)=\infty$ and $u'(r)<0$ for every $r>0$ small. Thus,  ${\rm sgn}\, (u'(r))={\rm sgn}\, (1-m)$.

For $\mathfrak C$ given in \eqref{yoo1} and every $r>0$ small, we define
\begin{equation} \label{dfx1} 
t=(u(r))^{\frac{\kappa}{m-1}}\  \mbox{ and}\ \  \frac{dt}{dr}=-\frac{X(t)+\mathfrak C}{r},\mbox{ where }
X(t)=\frac{\kappa}{1-m}\,t\,\frac{ru'(r)}{u(r)}-\mathfrak{C}. 
\end{equation} 
Remark that $X(t)+\mathfrak C>0$ for every $r>0$ small and \eqref{yoo1} yields that
 \begin{equation}\label{poll1}
\frac{t}{\log\,(1/r)}\rightarrow \mathfrak C\quad \text{as } r\rightarrow 0^+.
\end{equation}

\begin{lemma}\label{bx111}
Let $t$ and $X(t)$ be defined as in \eqref{dfx1}. Then, 
there exists $T_0>0$ large such that for every $t\geq T_0$, it holds $X(t)+\mathfrak C>0$ and
\begin{equation} \label{dust1}  \frac{dX}{dt}=\frac{q}{\kappa}\frac{X(t)+\mathfrak C}{t}-2\rho \left[  1-\left( 1+\frac{X(t)}{\mathfrak C}\right)^{m-1}\right].
\end{equation}
\begin{itemize}
\item [(a)] If $q\neq 0$, then we have
\begin{equation}\label{miso11}
\lim_{t\rightarrow \infty}t X(t)=-\frac{q\,\mathfrak C^2}{2\rho\kappa\,(m-1)},
\end{equation} 
which implies that \begin{equation}\label{tine}
\frac{u(r)}{Y_0(r)}=1-\frac{q}{2\rho\,\kappa^2}\frac{\log\,\log\,(1/r)}{\log\,(1/r)}\left (1+o(1)\right )\quad\text{as }r\rightarrow 0^+.
\end{equation}
\item [(b)] If $q=0$, then necessarily $m>1$ and there exists a constant $C\in \mathbb{R}$ such that 
\begin{equation}\label{cec1}
u(r)=(2\rho)^{\frac{1}{m-1}}\log\,(1/r)+C\quad\text{for every }r>0\text{ small}.
\end{equation}
\end{itemize}
	\end{lemma}
	
	\begin{proof}
	Let $T_0>0$ be large such that $X(t)+\mathfrak C>0$ for all $t\geq T_0$. Using \eqref{muzia} and  \eqref{dfx1}, we get 
	$$ \begin{aligned} 
	\frac{dX}{dt}=X'(t)&=\frac{X(t)+\mathfrak C}{t} -\frac{\kappa t\,r}{(1-m) (X(t)+\mathfrak C)} \,\frac{d}{dr} \left(\frac{r u'(r)}{u(r)}\right)\\
	& =\frac{X(t)+\mathfrak C}{t} -\frac{\kappa t}{(1-m) (X(t)+\mathfrak C)} \left[ 2\rho \frac{r u'(r)}{u(r)} -\left( \frac{r u'(r)}{u(r)}\right)^2
	+ u^\kappa \left|  \frac{r u'(r)}{u(r)}\right|^m
	\right]
	\end{aligned}
	$$ for every $t\geq T_0$. 
	Since $r u'(r)/u(r)=(X(t)+\mathfrak C) (1-m)/(\kappa t)$, we arrive at \eqref{dust1}. 
	
\vspace{0.2cm}	  	
(a) Let $q\neq 0$. We prove that $X'(t)\not=0$ for every $t>0$ large.  
Suppose by contradiction that there exists an  
increasing sequence $\{t_n\}_{n\geq 1}$ in $[T_0,\infty)$ such that $\lim_{n\to \infty} t_n=\infty$ and 
	 $ X'(t_n)=0$ for all $ n\geq 1$. In light of \eqref{dust1}, we obtain that
	 $$X''(t_n)=-\frac{q}{\kappa}\frac{X(t_n)+\mathfrak C}{t_n^2}< 0.$$
	Hence, all the critical points $t_n$ of $X$ will be local maxima. This is impossible. Therefore, $X'
	(t)\not=0$ for every $t>0$ large. Then, $\lim_{t\to \infty} X(t)=0$ by L'H\^opital's rule and \eqref{poll1} since 
	$$\mathfrak C= \lim_{r\to 0^+}  \frac{t}{\log \,(1/r)}=-\lim_{r\to 0^+}  r \frac{dt}{dr}=\lim_{t\to \infty} (X(t)+\mathfrak C). 
	$$
Hence, as $t\rightarrow \infty$, we find that
	\begin{equation} \label{dust11} 
	 \frac{dX}{dt}=\frac{q}{\kappa}\frac{X(t)+\mathfrak C}{t}+2\rho \left (m-1\right )\frac{X(t)}{\mathfrak C}+\frac{\rho\left (m-1\right )\left (m-2\right )}{\mathfrak C^2}X^2(t)\left (1+o(1)\right ).
\end{equation}	
	
Since $\rho \left(m-1\right)>0$ and ${\rm sgn}\, (X(t))=-{\rm sgn}\, (X'(t))$ for $t>0$ large, from \eqref{dust11}, we obtain 
a constant $C>0$ such that $t|X(t)|<C$ for every $t>0$ large, which leads to 
\begin{equation}\label{bezit11}
	\lim_{t\rightarrow \infty} t X^2(t)=0.
	\end{equation}
Hence, after multiplying \eqref{dust11} by $e^{-\frac{2\rho\left (m-1\right )}{\mathfrak C}t}$, we get
	  \begin{equation}\label{bezitt}
	  \frac{d}{dt}\left ( X(t)\,e^{-\frac{2\rho\left (m-1\right )}{\mathfrak C}t} \right )= \frac{q\, \mathfrak C}{\kappa \, t}
	   e^{-\frac{2\rho\left (m-1\right )}{\mathfrak C}t} (1+o(1))  \quad \mbox{as } t\to \infty.
	  	\end{equation}	  	
Using L'H\^opital's rule and \eqref{bezitt}, we conclude the claim of \eqref{miso11}.  
From \eqref{dfx1}, we see that 
$$X(t)=-r\frac{d}{dr}\left [u^{\frac{\kappa}{m-1}}(r)-\mathfrak C \log\,(1/r)\right ] \quad \mbox{for every } r>0\ \mbox{small}.$$
Hence, using \eqref{poll1}, \eqref{miso11} and L'H\^opital's rule,  we infer that there exists
$$\lim_{r\rightarrow 0^+}\frac{u^{\frac{\kappa}{m-1}}(r)-\mathfrak C \log\,(1/r)}{\log\,\log\,(1/r)}=\lim_{r\rightarrow 0^+}X(t)\log\,(1/r)=-\frac{q\,\mathfrak C}{2\rho\kappa \left (m-1\right )},$$
which implies that \eqref{tine} holds. 

\vspace{0.2cm}
(b) Let $q=0$. Since $\rho\left(m-1\right)>0$ and $\lim_{t\to \infty} X(t)=0$, there exists $T_1>T_0$ large such that %\eqref{dust11} gives
$$   \frac{dX}{dt}>\rho \left (m-1\right )\frac{X(t)}{\mathfrak C}\quad \mbox{for every } t\geq T_1.
$$
Then, $X\equiv 0$ on $[T_1,\infty)$. Indeed, if there exists $T\geq T_1$ such that $X(T)>0$ (resp., $X(T)<0)$, then $X'(T)>0$ (resp., $X'(T)<0)$. 
Therefore, $X(t)>X(T)$ for all $t\geq T$ (resp., $X(t)<X(T) $ for all $t\geq T$), which is a contradiction with $X(t)\rightarrow 0$ as $t\rightarrow \infty$. 

Since $q=0$ and $\kappa:=m+q-1>0$, then necessarily, $m>1$. Using \eqref{dfx1}, by direct calculations, we obtain \eqref{cec1}.
The proof of Lemma \ref{bx111} is complete. 
	\end{proof}
	
\begin{lemma} \label{gigg} Let $R>0$ be arbitrary. Then, equation \eqref{eq1} in $B_R(0)\setminus \{0\}$ has infinitely many positive radial solutions $u$ satisfying \eqref{yoo1}. 
\end{lemma}	
	
\begin{proof}	
If $q=0$, then the existence claim follows from Lemma~\ref{bx111} by taking 
$$C=(2\rho)^{1/(m-1)} \log R+C_0, $$ where $C_0\geq 0$ is any non-negative constant. 

We next assume that $q\not=0$.  
Let $T_0>0$  be large enough such that \eqref{dust1}  has a solution $X(t)$ satisfying $\lim_{t\to \infty} X(t)=0$. Then, for $r_0>0$ arbitrary, we consider the differential equation in \eqref{dfx1} for $t\geq T_0$, subject to $r(T_0)=r_0$.
This gives the unique solution 
$$ r(t)= r_0 \exp\left(- \int_{T_0}^t \frac{ds}{\mathfrak C+ X(s)}\right)\quad \mbox{for every } t\geq T_0. $$
We get that $t/\log \,(1/r)\to \mathfrak C$ as $r\to 0^+$. In light of Lemma~\ref{bx111}, we define $u(r)$ by
$$ u(r)=t^{\frac{m-1}{\kappa}} \quad \mbox{for every } r\in (0,r_0].
$$ Since $X$ solves \eqref{dust1}, we find that $u$ is a positive radial solution of \eqref{eq1} on $B_{r_0}(0)\setminus \{0\}$ such that \eqref{yoo1} holds. 
Then, for any $j\in (0,1)$, by applying the transformation $T_j$ 
in \eqref{scale} to the solution $u$ constructed above, we get another positive radial solution $T_j[u]$ of \eqref{eq1} in 
$B_{r_0/j}(0)\setminus \{0\}$ satisfying \eqref{yoo1}. 
Since $r_0>0$ is arbitrary, the proof of Lemma~\ref{gigg} is complete. 
\end{proof}

%%%%
\section{The asymptotic profile $Z_0$ for $\lambda=\Theta=\rho=0$, $m\in (0,2)$ (Proof of Theorem \ref{zzz0})}

Let \eqref{cond1} hold, $\Theta=\lambda=\rho=0$ and $m\in (0,2)$. Unless otherwise stated, $u$ is a positive radial solution of \eqref{eq1} in $B_R(0)\setminus \{0\}$ for $R>0$ such that \eqref{zoo} holds, that is 
\begin{equation}\label{zoo1}
u(r)\sim \left (\frac{(q+1)(2-m)^{1-m}}{\kappa^{2-m}}\right )^{\frac{1}{\kappa}}\left (\log\,(1/r)\right )^{\frac{m-2}{\kappa}}
:=Z_0(r)\quad \mbox{as } r\to 0^+. \end{equation}
From \eqref{eq1}, we infer that 
\begin{equation} \label{sbn} 
\left(ru'(r) \right)'=r^{\theta+1} u^q(r)|u'(r)|^m\geq 0\quad \mbox{for all }r\in (0,R).\end{equation}
Hence, we have $\lim_{r\to 0^+} ru'(r)=0$ in view of \eqref{zoo1} and L'H\^ opital's rule.
Since  $\lim_{r\to 0^+} u(r)=0$, it follows that $u'>0$ on $(0,R)$. For every $r\in (0,R)$, we define
\begin{equation}
\mathfrak{T}_m(r)=r^{2-m}(u'(r))^{2-m}-\frac{2-m}{q+1}\,
u^{q+1}(r).
\end{equation}
Then, $\mathfrak T_m(r)$ is differentiable on $(0,R)$ and  
$\mathfrak T_m'(r)=0$ for all $r\in (0,R)$. Hence,  
\begin{equation} 
\mathfrak T_m(r)= r^{2-m}(u'(r))^{2-m}-\frac{2-m}{q+1}\,u^{q+1}(r)=0\quad \mbox{for all } r\in (0,R). 
\end{equation}
 Thus, there exists a constant $\eta\geq 0$ such that for every $r\in (0,R)$, 
$$ u(r)=\left[ \frac{(2-m)^{1-m}(q+1)}{\kappa^{2-m}}\right]^\frac{1}{\kappa}
\left[ 
\log \left(\frac{R}{r}\right) + \eta
\right]^{-\frac{2-m}{\kappa}}. 
$$

%%%%

 \vspace{0.1cm}
{\em Acknowledgements.} The research of both authors was supported by the Australian Research Council (ARC) under grants DP190102948 and 
DP220101816.

\end{document}